\documentclass[a4paper,11pt]{article}          

\usepackage{amsmath,amsthm}     
\usepackage{amssymb}            
\usepackage{euscript}           
\usepackage{graphicx,enumerate,color}
\usepackage{graphics}
\usepackage[top=2.5cm, bottom=2.5cm, left=2.5cm, right=2.5cm]{geometry}
\usepackage{url}
\usepackage{soul} 
\usepackage[ colorlinks = true,
linkcolor = blue,
urlcolor  = blue,
citecolor = blue,
anchorcolor = blue]{hyperref}
\usepackage{stmaryrd}
\usepackage{tikz}
\usepackage{tikz-cd}
\usepackage{amscd}

\usepackage{lipsum}

\usepackage{xcolor}
\PassOptionsToPackage{dvipsnames}{xcolor}
\RequirePackage{xcolor}

\newtheorem{thm}{Theorem}[section]
\newtheorem{defn}[thm]{Definition}
\newtheorem{prop}[thm]{Proposition}

\newtheorem{cor}[thm]{Corollary}
\newtheorem{lemma}[thm]{Lemma}

\theoremstyle{definition}  

\newtheorem{remark}[thm]{Remark}

\title{ An Excision Theorem in Heegaard Floer Theory}
\author{Neda Bagherifard}
\newcommand{\Addresses}{{
		\bigskip
		\footnotesize
		\noindent Neda Bagherifard\\
		University of Oregon, Eugene, OR, 97403\\
		email: \texttt{nbagheri@uoregon.edu}
}}

\date{}
\begin{document}
\maketitle
\begin{abstract}
	Let $Y_1$ be a closed, oriented 3-manifold and $\Sigma$ denote a non-separating closed, orientable surface in $Y_1$ which consists of two connected components of the same genus. By cutting $Y_1$ along $\Sigma$ and re-gluing it using an orientation-preserving diffeomorphism of $\Sigma$ we obtain another closed, oriented 3-manifold $Y_2$.	When the excision surface $\Sigma$ is of genus one, we show that  twisted Heegaard Floer homology groups  of $Y_1$ and $Y_2$ (twisted with coefficients in the universal Novikov ring) are isomorphic. We use this excision theorem to demonstrate that certain manifolds are not related by the excision construction on a genus one surface. Additionally, we apply the excision formula to compute twisted Heegaard Floer homology groups of 0-surgery on certain two-component links, including some families of 2-bridge links.
\end{abstract}
\section{Introduction}
Excision formulas study the behavior of Floer  homologies under certain cutting and gluing of a 3-manifold along a surface. To explain the excision construction more precisely,
let $Y_1$ be a closed, oriented 3-manifold with at most two connect components and $\Sigma_i$, $i=1,2$, be two disjoint closed, connected, oriented, non-separating surfaces in $Y_1$ of the same genus. If $Y_1$ has two components $Y_{11}$ and $Y_{12}$, we require that $\Sigma_i$ is a surface in $Y_{1i}$, $i=1,2$. 
Let $h:\Sigma_1\rightarrow\Sigma_2$ be an orientation-preserving
diffeomorphism. Cut $Y_1$ along $\Sigma_1\cup\Sigma_2$ and denote the resulting manifold by $Y'$. $Y'$ is a manifold with four boundary components:
\[
\partial Y'=\Sigma_1\cup-\Sigma_1\cup\Sigma_2\cup-\Sigma_2.
\]
Glue $\Sigma_1$ to $-\Sigma_2$ and  $-\Sigma_1$ to $\Sigma_2$, through the diffeomorphism $h$ for each case, to obtain a closed manifold $Y_2$ (see Figure \ref{excision}).  We say that $Y_2$ is obtained from $Y_1$ by excision along the surfaces $\Sigma_1$ and $\Sigma_2$.

\begin{figure}[h!]
	\def\svgwidth{10cm}
	\begin{center}
\begingroup%
  \makeatletter%
  \providecommand\color[2][]{%
    \errmessage{(Inkscape) Color is used for the text in Inkscape, but the package 'color.sty' is not loaded}%
    \renewcommand\color[2][]{}%
  }%
  \providecommand\transparent[1]{%
    \errmessage{(Inkscape) Transparency is used (non-zero) for the text in Inkscape, but the package 'transparent.sty' is not loaded}%
    \renewcommand\transparent[1]{}%
  }%
  \providecommand\rotatebox[2]{#2}%
  \newcommand*\fsize{\dimexpr\f@size pt\relax}%
  \newcommand*\lineheight[1]{\fontsize{\fsize}{#1\fsize}\selectfont}%
  \ifx\svgwidth\undefined%
    \setlength{\unitlength}{222.46544545bp}%
    \ifx\svgscale\undefined%
      \relax%
    \else%
      \setlength{\unitlength}{\unitlength * \real{\svgscale}}%
    \fi%
  \else%
    \setlength{\unitlength}{\svgwidth}%
  \fi%
  \global\let\svgwidth\undefined%
  \global\let\svgscale\undefined%
  \makeatother%
  \begin{picture}(1,0.53361689)%
    \lineheight{1}%
    \setlength\tabcolsep{0pt}%
    \put(0,0){\includegraphics[width=\unitlength,page=1]{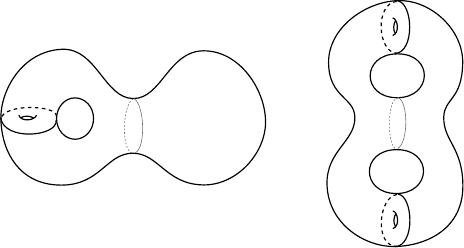}}%
    \put(0.24831606,0.11333593){\color[rgb]{0,0,0}\makebox(0,0)[lt]{\lineheight{1.25}\smash{\begin{tabular}[t]{l}$Y_1$\end{tabular}}}}%
    \put(0.67601594,0.02144124){\color[rgb]{0,0,0}\makebox(0,0)[lt]{\lineheight{1.25}\smash{\begin{tabular}[t]{l}$Y_2$\end{tabular}}}}%
    \put(0.0390375,0.20758863){\color[rgb]{0,0,0}\makebox(0,0)[lt]{\lineheight{1.25}\smash{\begin{tabular}[t]{l}$\Sigma_1$\end{tabular}}}}%
    \put(0.89290281,0.47242797){\color[rgb]{0,0,0}\makebox(0,0)[lt]{\lineheight{1.25}\smash{\begin{tabular}[t]{l}$\Sigma_1$\end{tabular}}}}%
    \put(0.89199581,0.05552913){\color[rgb]{0,0,0}\makebox(0,0)[lt]{\lineheight{1.25}\smash{\begin{tabular}[t]{l}$\Sigma_2$\end{tabular}}}}%
    \put(0.49696129,0.2021021){\color[rgb]{0,0,0}\makebox(0,0)[lt]{\lineheight{1.25}\smash{\begin{tabular}[t]{l}$\Sigma_2$\end{tabular}}}}%
    \put(0,0){\includegraphics[width=\unitlength,page=2]{FloerExcision-2.pdf}}%
  \end{picture}%
\endgroup%

		\caption{The construction of excision when $Y_1$ is connected.}
		\label{excision}
	\end{center}
\end{figure}

An excision formula which studies the behavior of instanton  homology under  cutting and gluing of a three manifold along a surface of genus one was introduced by Floer (see \cite{Braam1995}). Kronheimer and Mrowka proved a similar formula for the monopole and instanton Floer homology where the cutting and gluing is along higher genus surfaces  (see  \cite[Theorem 3.1 and Theorem 7.7]{Kron-Mrow}). When this surface is of genus one, there is a version of the excision formula for monopole Floer homology with local coefficients  (see \cite[Theorem 3.2 ]{Kron-Mrow}). 

Floer's excision theorem is a crucial tool in constructing monopole/instanton Floer homology for a sutured manifold and using it to define an invariant for a knot in a closed manifold.
In fact, to a sutured manifold $(M,\gamma)$, one can associate a closed three manifold $(Y,\overline{R})$ where $\overline{R}$ is a closed surface in $Y$. Kronheimer and Mrowka   (see \cite[Section 4]{Kron-Mrow}), defined monopole/instanton Floer homology groups for $(M,\gamma)$ as the monopole/instanton Floer homology groups of $(Y,\overline{R})$,  which are well-defined using Floer's excision theorem. They used this construction to define a monopole knot homology group $\textit{KHM}(Z,K)$ for a knot $K$ in a closed three manifold $Z$ (see \cite{Kronheimer2010KhovanovHI}), which provided a tool to prove that Khovanov homology detects the unknot (see \cite{Kronheimer2010KhovanovHI}). 

Another tool for studying knots and links in a 3-manifold is singular instanton Floer homology, which was defined for
a 3-manifold $Y$ with a link $L$ (see \cite{Kronheimer2010KhovanovHI,Kron-Knot}). 
The excision formula is valid in singular instanton Floer homology as long as the link $L$ is either disjoint from the excision surface or intersect it in an odd number of points (see \cite{Street,xie2019instanton,xie2023ring}).
This singular version of the excision theorem is used in \cite{xie2019instanton} to define an instanton Floer homology for sutured manifolds with tangles which is used later to show that annular Khovanov homology detects the unlink among other applications.

Ye's excision theorem (see \cite[Theorem A.1.30]{phdthesis}) relates the Heegaard Floer homology groups of $Y_1$ and $Y_2$ in some specific $\text{Spin}^c$ structures.
More precisely, let $Y$ be a closed, oriented 3-manifold and $F$ be a closed, oriented surface in $Y$ such that each component of $Y$ has a non-empty intersection with $F$. Let $F_i$, $i=1,\dots,m$, denote the connected components of $F$. Then 
\[
\text{Spin}^c(Y|F):=\{\mathfrak{s}\in\text{Spin}^c(Y)\ |\ \langle c_1(\mathfrak{s}),[F_i]\rangle=2g(F_i)-2,\  1\leq i\leq m\},
\]
\[
HF^+(Y|F):=\bigoplus_{\mathfrak{s}\in\text{Spin}^c(Y|F)}HF^+(Y,\mathfrak{s}),
\]
where $g(F_i)$ denotes the genus of $F_i$.
With this introduction, Ye's excision theorem is as follows.

\begin{thm}\label{FLoer-Ex}
	\cite[Theorem A.1.30]{phdthesis}
	Let $Y_2$ be constructed from $Y_1$ as explained above. Denote the excision surface in $Y_1$ and its corresponding copy in $Y_2$ by $F=\Sigma_1\cup\Sigma_2$ where $g(\Sigma_i)\geq2$, $i=1,2$. Then
	\[
	HF(Y_1|F)\cong HF(Y_2|F).
	\]
    Moreover, this isomorphism and its inverse are induced by restricted graph cobordisms. Here $HF(Y,\mathfrak{s})$ denotes $HF_{red}(Y,\mathfrak{s})=HF^+(Y,\mathfrak{s})\cong\mathbf{HF}^-(Y,\mathfrak{s})$ where $\mathbf{HF}^-$ is the completion of $HF^-$.
\end{thm}

Also, when the excision surface is of genus one, Ai and Peters proved a version of excision formula for a closed oriented 3–manifold which fibers over the circle (see \cite[Theorem 1.3]{A-P2010} and Theorem \ref{Thm-2} in Section 3). 
In this paper, we prove an excision theorem in $\omega$-twisted Heegaard Floer theory for arbitrary genus one surfaces. Our main theorem is:

\begin{thm}\label{MainThm}
	Let $Y_2$ be constructed from $Y_1$ as explained above and $F=\Sigma_1\cup\Sigma_2$. Suppose that $\iota_{F,i}:F\hookrightarrow Y_i$ and $\iota_{Y',i}:Y'\hookrightarrow Y_i$ denote the inclusions of $F$ and $Y'$ into $Y_i$.  Let $\omega_i\in H^2(Y_i;\mathbb{R})$ be such that $\iota_{F,1}^*(\omega_1)=\iota_{F,2}^*(\omega_2)\neq0$ and $\iota_{Y',1}^*(\omega_1)=\iota_{Y',2}^*(\omega_2)$. When $g(\Sigma_i)=1$, $i=1,2$,
	\[
	\underline{HF}(Y_1;\Lambda_{\omega_1})\cong  \underline{HF}(Y_2;\Lambda_{\omega_2}).
	\]
	 Moreover, this isomorphism and its inverse are induced by restricted graph cobordisms. Here $\underline{HF}(Y,\mathfrak{s};\Lambda_\omega)$ denotes \[
	\underline{HF}_{red}(Y,\mathfrak{s};\Lambda_\omega)= \underline{HF}^+(Y,\mathfrak{s};\Lambda_\omega)\cong\underline{HF}^-(Y,\mathfrak{s};\Lambda_\omega). 
	\]
\end{thm}
In this paper, we work over $\mathbb{F}_2=\mathbb{Z}/2\mathbb{Z}$ unless otherwise stated.
In Theorem \ref{MainThm}, $\Lambda_\omega$ denotes the universal Novikov ring, which is equipped with a $\mathbb{F}_2[H^1(Y;\mathbb{Z})]$-module structure via a 2-dimensional cohomology class $\omega\in H^2(Y;\mathbb{R})$.  $\underline{HF}(Y,\mathfrak{s};\Lambda_{\omega})$ denotes a version of Heegaard Floer homology groups which is twisted  with coefficients in $\Lambda_\omega$ (see the next section for more details). Note that Theorem \ref{MainThm} implies that the choice of the diffeomorphism $h$ in the excision construction  does not matter (see Remark \ref{Independence-diff} for a precise statement).

We can use Theorem \ref{MainThm} to check if two given 3-manifolds are not related by the excision construction. Let $L$ be a link in $S^3$ and $S^3_0(L)$ denote the 3-manifold obtained by performing 0-surgery on the components of $L$. 
\begin{cor}\label{Topology-result}
	Let $W_n$ denote the $n$-twisted Whitehead link (see Figure \ref{Whitehead double} on the right for $n>0$). If $n,m\neq0$ and $|n|\neq |m|$, then $S^3_0(W_n)$ and $S^3_0(W_m)$ are not related by the excision construction on a genus one surface. 
\end{cor}
Another application of Theorem \ref{MainThm} is in computing twisted Heegaard Floer homology groups of manifolds obtained by performing 0-surgery on certain families of 2-bridge links.

\begin{cor}\label{two-bridge}
	Let $C(m,\pm1,n)$ denote a 2-bridge link shown in Figure \ref{Two-Bridge} such that $|m+n|\leq1$ and $m,n\neq0$. 
	Then
	\[
	\underline{HF}^+(S^3_0(C(m,\pm1,n));\Lambda_\omega)\cong\begin{cases}
		\Lambda^{|mn|}&\textnormal{if}\ m=-n,\\
		\Lambda^{|mn|}\oplus\Lambda[U^{-1}]&\textnormal{if}\ |m-n|=1,
	\end{cases}
	\]
	where $[\omega]\in H^2(S^3_0(C(m,\pm1,n));\mathbb{R})$ is $\lambda\textnormal{PD}[\eta]\neq0$, $\lambda\in\mathbb{R}$. Here $\eta$ is the meridian of the red component of $C(m,\pm1,n)$ (see Figure \ref{Two-Bridge} on the bottom right).
\end{cor}

\emph{Outline of the paper.} In Section \ref{Section 1}, we review twisted Heegaard Floer theory where the Heegaard Floer homology groups are twisted with Novikov ring.  
In Section \ref{Section 2}, we modify the proof of Theorem \ref{FLoer-Ex} to the twisted case to provide a proof for Theorem \ref{MainThm}. In Section \ref{Section 3}, Theorem \ref{MainThm} is applied to compute Heegaard Floer homology groups of a 3-manifold obtained by operating 0-surgery on the $n$-twisted Whitehead link (see Figure \ref{Whitehead double}). As a result of this computation, we provide proofs for Corollary \ref{Topology-result} and Corollary \ref{two-bridge}.

\textbf{Acknowledgment. }The author would like to thank her advisor Robert Lipshitz for continuous guidance, helpful discussions and suggestions, and for carefully reading a draft of the paper. The research was supported by NSF grant DMS-2204214.

\section{Background on twisted Heegaard Floer theory}\label{Section 1}

Associated to a closed, connected, oriented 3-manifold $Y$, there are Heegaard Floer homology groups denoted $HF^\circ(Y)$, $\circ\in\{\pm,\wedge,\infty\}$, introduced by Ozsv\'{a}th-Szab\'{o} (see \cite{OS}), which are modules over the ring $\mathbb{Z}[U]$ or $\mathbb{F}_2[U]$. These modules split by $\text{Spin}^\text{c}$ structures as $\oplus_{\mathfrak{s}\in\text{Spin}^c(Y)}HF^\circ(Y,\mathfrak{s})$. The modules are constructed using a pointed Heegaard diagram \[\mathcal{H}=(\Sigma,\boldsymbol{\alpha}=\{\alpha_1,\dots,\alpha_g\},\boldsymbol{\beta}=\{\beta_1,\dots,\beta_g\},z),
\] $z\in\Sigma\setminus\boldsymbol{\alpha}\cup\boldsymbol{\beta}$, for $Y$ and a choice of a generic almost complex structure $J$ on $\text{Sym}^g(\Sigma)$ which satisfies some compatibility conditions with a symplectic form on $\text{Sym}^g(\Sigma)$ ($g$ denotes the genus of $\Sigma$). Roughly, the differential counts $J$-holomorphic disks $u:[0,1]\times\mathbb{R}\rightarrow \text{Sym}^g(\Sigma)$ with boundary on $\mathbb{T}_\alpha=\alpha_1\times\dots\times\alpha_g$ and $\mathbb{T}_\beta=\beta_1\times\dots\times\beta_g$ which are asymptotic to the generators which are elements in $\mathbb{T}_\alpha\cap\mathbb{T}_\beta$. 
There is an equivalent cylindrical reformulation, suggested by Lipshitz (see \cite{RL}), where $J$-holomorphic surfaces $u:S\rightarrow\Sigma\times[0,1]\times\mathbb{R}$ with cylindrical ends are counted (see \cite{RL} for more details).

Heegaard Floer homology fits into the framework of a $(3+1)$-dimensional TQFT (see \cite{OS4} and \cite{OZSVATH}). Let $W$ be a  smooth, oriented four manifold such that $\partial W=-Y_1\amalg Y_2$ where $Y_i$, $i=1,2$, are closed, oriented and connected 3-manifolds. In this case, we say $W$ is a cobordism from $Y_1$ to $Y_2$.  This cobordism induces a map 
\[
F^\circ_{W,\mathfrak{t}}:HF^\circ(Y_1,\mathfrak{s}_1)\rightarrow HF^\circ(Y_2,\mathfrak{s}_2),
\]
where $\mathfrak{t}\in\text{Spin}^c(W)$, and $\mathfrak{s}_i=\mathfrak{t}|_{Y_i}$ (see \cite[Theorem 1.1]{OZSVATH}).
In the following, we will work with cobordisms between disconnected 3-manifolds. For that purpose, we need the following definition.
\begin{defn}\cite[Definition 3.1]{Z15}
	A multi-pointed 3-manifold is a pair $(Y,\mathbf{w})$ consisting of a closed, oriented 3-manifold with a finite collection of basepoints $\mathbf{w}\subset Y$ such that each component of $Y$ contains at least one basepoint. 
\end{defn} 

\begin{defn}\cite[Definition 4.1]{Z15}
	Suppose $(Y,\mathbf{w})$ is a multi-pointed 3-manifold. A multi-pointed Heegaard diagram $\mathcal{H}=(\Sigma,\boldsymbol{\alpha},\boldsymbol{\beta},\mathbf{w})$
	for $(Y,\mathbf{w})$ is a tuple such that $\mathcal{H}=(\Sigma,\boldsymbol{\alpha},\boldsymbol{\beta})$ is a standard Heegaard diagram for $Y$ and $\mathbf{w}\subset\Sigma\setminus(\boldsymbol{\alpha\cup\boldsymbol{\beta}})$.
\end{defn}

Let $\mathbf{w}=\{w_1,\dots,w_n\}$, then
\[
\mathbb{F}_2[U_\mathbf{w}]:=\mathbb{F}_2[U_{w_1},\dots,U_{w_n}].
\]
Also, $\mathbb{F}_2[U_{\mathbf{w}},U^{-1}_{\mathbf{w}}]$ denotes the ring obtained by formally inverting each of the variables $U_{w_i}$. If $\mathbf{k}=(k_1,\dots, k_n)$ is an $n$-tuple, we write
\[U_\mathbf{w}^\mathbf{k}:=U_{w_1}^{k_1}\cdots U^{k_n}_{w_n}.
\]
We write $U_i$ for $U_{w_i}$. Also, $k[U_\mathbf{w}]$ and $k[U_\mathbf{w},U_\mathbf{w}^{-1}]$, where $k$ is a ring, are define similarly.
Similar to the connected case above, multi-pointed 3-manifolds can be described using multi-pointed Heegaard diagrams $\mathcal{H}=(\Sigma,\boldsymbol{\alpha},\boldsymbol{\beta},\mathbf{w})$  and there is a map
\[
\mathfrak{s}_\mathbf{w}:\mathbb{T}_{\boldsymbol{\alpha}}\cap\mathbb{T}_{\boldsymbol{\beta}}\rightarrow \text{Spin}^\text{c}(Y).
\]
We recall the definition of $CF^\circ(Y)$ when $Y$ is disconnected. Let $(Y,\mathbf{w})=(Y_1,\mathbf{w}_1)\amalg (Y_2,\mathbf{w}_2)$, where $(Y_i,\mathbf{w}_i)$, $i=1,2$ is a connected, multi-pointed 3-manifold. Let $\mathcal{H}_i$, $i=1,2$, be weakly admissible multi-pointed Heegaard diagrams for $(Y_i,\mathbf{w}_i)$. Define
\[
(CF^\circ(\mathcal{H}_1\amalg\mathcal{H}_2),\partial_{J_s}):=(CF^\circ(\mathcal{H}_1),\partial_{J_{s_1}})\otimes_{\mathbb{F}_2}(CF^\circ(\mathcal{H}_2),\partial_{J_{s_2}}).
\]
where  $J_{s}=J_{s_1}\amalg J_{s_2}$ is a generic path of almost complex structures in $Y$.

The graph TQFT of Zemke (see \cite{Z15}) generalizes to cobordisms with disconnected ends where he defines a notion of  \emph{ribbon graph cobordism} $(W,\Gamma)$ between $(Y_1,\mathbf{w}_1)$ and $(Y_2,\mathbf{w}_2)$. Here $\Gamma$ is an embedded graph with $\Gamma\cap Y_i=\mathbf{w}_i$, and is decorated with a \emph{formal ribbon structure}. Zemke proves that  there are two chain maps 
\[
F^A_{W,\Gamma,\mathfrak{t}}, F^B_{W,\Gamma,\mathfrak{t}}:CF^\circ(Y_1,\mathbf{w}_1,\mathfrak{s}_1)\rightarrow CF^\circ(Y_1,\mathbf{w}_2,\mathfrak{s}_2),
\]
which are diffeomorphism invariants of $(W,\Gamma)$ up to $\mathbb{F}_2[U]$ equivariant chain homotopy. Here $\mathfrak{t}\in \text{Spin}^c(W)$ and $\mathfrak{s}_i=\mathfrak{t}|_{Y_i}$. (See Subsection \ref{GraphCobordism} for more details). We denote the map induced on the chain complex by $f^\circ_{W,\Gamma,\mathfrak{t}}$ and the map induced on the homology by $F^\circ_{W,\Gamma,\mathfrak{t}}$. 

In the following subsections, 
following \cite[Section 12]{Z21} and \cite[Subsection 2.1]{JZ}, $\doteq$ and $\dot{\simeq}$  indicate equality or chain homotopy of morphisms, respectively, up to some unit. 
We work with 2-dimensional cohomology classes $[\omega]\in H^2(X;\mathbb{R})$ where $X$ is either a closed 3-manifold or a 4-dimensional cobordism. We consider $k[H^1(Y;\mathbb{Z})]$-modules $M$ where $k$ can be either $\mathbb{F}_2$ or $\mathbb{Z}$.

\subsection{Twisted Heegaard Floer homology groups}
In this part, we review twisted Heegaard Floer theory where the coefficients are weighted by a 2-dimensional cohomology class $\omega$. This is a special case of Heegaard Floer homology with twisted coefficients (see \cite{OS2}, \cite{JM}, \cite{Z21}).

\textbf{Novikov ring.} 
Let $k$ be a commutative ring and $\Gamma\subset\mathbb{R}$ be an additive subgroup of $\mathbb{R}$. The Novikov ring of $\Gamma$, $\text{Nov}(\Gamma)$,  is a ring consisting of formal sums  $\sum a_{r_i}t^{r_i}$ where $r_1<r_2<\cdots$, $r_i\in\Gamma$, and $r_i\rightarrow\infty$. Here $t$ is a formal variable. Equivalently,
\[\text{Nov}(\Gamma)=\{\sum_{r\in\Gamma}a_rt^r|a_r\in k,\ \ \#\{r|a_r\neq0,\ r<c\}<\infty,\ \forall c\in\mathbb{R}\}.\]
The multiplication is defined as 
\[
(\sum_{r\in\Gamma}a_{r}t^{r})\cdot(\sum_{r\in\Gamma}b_{r}t^{r})=\sum_{z\in\Gamma}(\sum_{r+s=z}a_{r}\cdot b_{s})t^{z}.
\]
When $k$ is a field, the Novikov ring is also a field. $\text{Nov}(\mathbb{R})$ is called the universal Novikov ring and is denoted by $\Lambda$.  In this paper, $k=\mathbb{F}_2$ unless otherwise stated.

Let $Y$ be a closed, oriented 3-manifold and fix $[\omega]\in H^2(Y;\mathbb{R})$. This cohomology class  induces a $\mathbb{F}_2[H^1(Y;\mathbb{Z})]$-module structure on $\Lambda$: for each $\eta\in H^1(Y;\mathbb{Z})$, define $t^\eta\cdot t^r:=t^{r+\int_{Y}\eta\cup\omega}$ where $\int_{Y}\eta\cup\omega=\langle \eta\cup\omega,[Y]\rangle$ and $[Y]$ denotes the fundamental class of $Y$. Let $\Lambda_\omega$ denote $\Lambda$ equipped with this $\mathbb{F}_2[H^1(Y;\mathbb{Z})]$-module structure. Let $a,b\in M$ where $M$ is a $\Lambda$-module. By $a\doteq b$, we mean that there exists $z\in\mathbb{R}$ such that $a=t^z\cdot b$.

\textbf{$\omega$-twisted chain complexes.} Let $(Y,\mathbf{w})$ denote a multi-pointed 3-manifold. 
Let $\mathcal{H}=(\Sigma,\boldsymbol{\alpha},\boldsymbol{\beta},\mathbf{w})$ denote a multi-pointed, weakly admissible Heegaard diagram for $(Y,\mathbf{w})$, $\mathfrak{s}\in \text{Spin}^c(Y)$, and $J$ be a generic almost complex structure. Fix a 2–cocycle
representative $\omega\in[\omega]$. The $\omega$-twisted chain complex, $\underline{CF}^\infty(\mathcal{H},\mathbf{w},\mathfrak{s};\Lambda_\omega)$, is a  $\Lambda[U_\mathbf{w},U^{-1}_\mathbf{w}]$-module which is freely generated by the points $\mathbf{x}\in \mathbb{T}_\alpha\cap \mathbb{T}_\beta$ with $\mathfrak{s}_{\mathbf{w}}(\mathbf{x})=\mathfrak{s}$. Let 
\[
n_\mathbf{w}(\phi):=(n_{w_1}(\phi),\cdots,n_{w_n}(\phi)).
\]
The differential is defined as follows:
\[
\underline{\partial} \mathbf{x}=\sum_{\mathbf{y}\in T_\alpha\cap T_\beta}\sum_{\substack{\phi\in\pi_2(\mathbf{x},\mathbf{y}),\\ \mu(\phi)=1}}\#\widehat{\mathcal{M}}(\phi) U_\mathbf{w}^{n_\mathbf{w}(\phi)}t^{\omega([\phi])}\ \mathbf{y},
\]
where $[\phi]$ (sometimes denoted by $\widetilde{D}(\phi)$) is a 2-chain in $Y$ obtained from $D(\phi)$ (which is the domain associated to $\phi$ and is a 2-chain in $\Sigma$) by coning the $\alpha$- and $\beta$-boundaries of $D(\phi)$ using gradient trajectories. Here $\omega([\phi])$ (sometimes denoted by $\int_{\widetilde{D}(\phi)}\omega$) is the evaluation of $\omega$ on $[\phi]$. Note that $\partial\widetilde{D}(\phi)$ depends only on $\mathbf{x}$ and $\mathbf{y}$.  Notice also that the
isomorphism class of the chain complex only depends on the cohomology class $[\omega]$ and is independent from the choice of the 2-cycle representative $\omega\in[\omega]$ (see \cite[Section 2.1]{A-P2010}). The homology group is denoted by $\underline{HF}^\infty(\mathcal{H},\mathfrak{s};\Lambda_\omega)$ which is a $\Lambda[U]$-module. Similar construction works for $\underline{CF}^-$, $\underline{CF}^+$, $\underline{\widehat{CF}}$.   
\begin{defn}\cite[Definition 2.3]{JZ}
	Suppose that $\mathcal{C}$ is a category and $I$ is a set. A transitive system in $\mathcal{C}$, indexed by
	$I$, is a collection of objects $(X_i)_{i\in I}$, as well as a distinguished morphism $\Psi_{i\rightarrow j} : X_i \rightarrow X_j$ for each
	$(i, j) \in I \times I$, such that
	\begin{enumerate}
		\item $\Psi_{i\rightarrow j}=\Psi_{j\rightarrow k}\circ\Psi_{i\rightarrow j}$
		\item 
		$\Psi_{i\rightarrow i}=\text{id}_{X_i}$.
	\end{enumerate}
\end{defn}
Here, we work with two categories. First category is the projectivized category of $\Lambda$-modules $\mathcal{C}= P(\Lambda- \text{Mod})$ where the objects are $\Lambda$-modules and
the morphism set $\text{Hom}_\mathcal{C}(X_1, X_2)$ is the projectivization of $\text{Hom}_\Lambda(X_1, X_2)$ under the action
of elements of $\Lambda$ of the form $t^z\in\Lambda$. In this category, given morphisms $f, g \in\text{ Hom}_\Lambda(X_1, X_2)$, we will use the notation $f\doteq g$ if $f = t^z\cdot g$ for some $z\in\mathbb{R}$. 
The second category is the projectivized homotopy category $\mathcal{C} = P(K(\Lambda– \text{Mod}))$. The objects of $\mathcal{C}$ are chain complexes over $\Lambda$. The morphism set $\text{Hom}_\mathcal{C}(X_1, X_2)$ is the projectivization of $H_*(\text{Hom}_\Lambda(X_1, X_2))$
under the action of elements of $\Lambda$ of the form $t^z$. In this category, given chain maps $\phi,\psi\in H_*(\text{Hom}_\Lambda(X_1, X_2))$, $\phi\ \dot{\simeq}\ \psi$ means $\phi\simeq t^z \psi$ for some $z\in\mathbb{R}$ and $\simeq$ means chain homotopy. When $\phi\ \dot{\simeq}\ \psi$, $\phi$ and $\psi$ are called projectively  equivalent (see \cite[Subsection 2.1]{JZ}).  A transitive system over one of
the above categories is called a projective transitive system.

$\underline{HF}^\circ(Y,\mathfrak{s};\Lambda_\omega)$, $\circ\in\{\pm,\infty,\wedge\}$,   
forms a projective transitive system of $\Lambda$-modules indexed by the set of pairs $(\mathcal{H},J)$, where $\mathcal{H}$ is an $\mathfrak{s}$
admissible diagram of $Y$, and $J$ is a generic almost complex structure. (see  \cite[Theorem 3.1]{JZ} and  \cite[Remark 12.1]{Z21}). 
Note that when $\omega$ is the zero cohomology class, we have 
\[
\underline{CF}^-(Y,\mathfrak{s};\Lambda_\omega)\cong CF^-(Y,\mathfrak{s})\otimes_{\mathbb{F}_2}\Lambda.
\]
Also when a three manifold is a disjoint union of two connected three manifolds $Y$ and $Y'$, we define 
\[
\underline{CF}^-(Y\amalg Y',\mathfrak{s}\oplus\mathfrak{s}';\Lambda_{\omega\oplus\omega'}):= \underline{CF}^-(Y,\mathfrak{s};\Lambda_\omega)\otimes_{\Lambda} \underline{CF}^-(Y',\mathfrak{s}';\Lambda_{\omega'}).
\]
\begin{remark}
	In this remark the coefficients are taken over $\mathbb{Z}$. There is a universally twisted Heegaard Floer homology, denoted by $\underline{HF}^\circ(Y,\mathfrak{s})$ which is a $\mathbb{Z}[H^1(Y;\mathbb{Z})]$-module (see \cite[Section 8]{OS2}). When $M$ is a  $\mathbb{Z}[H^1(Y;\mathbb{Z})]$-module, $\underline{HF}^\circ(Y,\mathfrak{s};M)$ is defined as the homology group of the chain complex 
	\[
	\underline{CF}^\circ(Y,\mathfrak{s};\mathbb{Z}[H^1(Y;\mathbb{Z})])\otimes_{\mathbb{Z}[H^1(Y;\mathbb{Z})]}M.
	\]
	If $M=\mathbb{Z}$ (where the elements of  $\mathbb{Z}[H^1(Y;\mathbb{Z})]$ act trivially), $\underline{HF}^\circ(Y,\mathfrak{s};M)=HF^\circ(Y,\mathfrak{s})$. From this point of view, $\underline{HF}^\circ$ is a lift of $HF^\circ$. When $M=\Lambda_\omega$, $\underline{HF}^\circ(Y,\mathfrak{s};M)$ is the $\omega$-twisted Heegaard Floer homology group of $Y$. Note that there is a twisted version, $\underline{HF}(Y,\mathfrak{s}; \omega)$ (see \cite[Section 3.1]{OS4}) which is different from the $\omega$-twisted version defined above. In fact, $\underline{HF}(Y,\mathfrak{s}; \omega)$ is obtained by taking $M=\mathbb{Z}[\mathbb{R}]$ as a $\mathbb{Z}[H^1(Y;\mathbb{Z})]$-module. Therefore, $\underline{HF}(Y,\mathfrak{s}; \Lambda_\omega)$ is the Novikov completion of $\underline{HF}(Y,\mathfrak{s}; \omega)$.
\end{remark}

\subsection{$\omega$-Twisted graph cobordisms}\label{GraphCobordism} 
In the following, $\omega$-twisted graph cobordisms are reviewed. (See  \cite{Z21,JZ} for more details). 
Let $(Y_i,\mathbf{w}_i)$, $i=1,2$, be two multi-pointed three manifolds. 
\begin{defn}\cite[Definition 3.2]{Z15}
	A ribbon graph cobordism $(W,\Gamma)$ from $(Y_1,\mathbf{w}_1)$ to $(Y_2,\mathbf{w}_2)$ consists of a cobordism $W$ from $Y_1$ to $Y_2$ and an embedded finite graph $\Gamma$ such that
	\begin{enumerate}
		\item there is no vertex of valence zero in $\Gamma$;
		\item  $\Gamma\cap Y_i=\mathbf{w}_i$, $i=1,2$.  Furthermore, each point of $\mathbf{w}_i$ has
		valence $1$ in $\Gamma$;
		\item each edge of $\Gamma$ is smoothly embedded;
		\item each vertex of $\Gamma$ is decorated with a cyclic ordering of all the edges that meet at it.
	\end{enumerate}
	Such a decoration is called a formal ribbon structure. 
\end{defn}
When $W$ is obtained from $Y\times I$ by attaching 4-dimensional $i$-handles, $i=1,2,3$, away from the basepoints $\mathbf{w}\subset Y$, and $\Gamma=\mathbf{w}\times I$, $(W,\Gamma)$ is called a \emph{restricted graph cobordism}.
Let $(\mathcal{H}_i,J_i)$, $i=1,2$, denote a weakly admissible Heegaard diagram for $(Y_i,\mathbf{w}_i)$.  Let $[\omega]\in H^2(W;\mathbb{R})$,  and $\omega_i=\omega|_{Y_i}$, $i=1,2$. The cobordism $(W,\Gamma)$ induces a map
\[
\underline{f}^-_{W,\Gamma,\mathfrak{t};\Lambda_\omega}:\underline{CF}^-(\mathcal{H}_1,\mathbf{w}_1,\mathfrak{t}|_{Y_1};\Lambda_{\omega_1})\rightarrow \underline{CF}^-(\mathcal{H}_2,\mathbf{w}_2,\mathfrak{t}|_{Y_2};\Lambda_{\omega_2})
\]
on the chain complexes with $\omega$-twisted coefficient where $\mathfrak{t}\in\text{Spin}^c(W)$ (see \cite[Section 3]{JZ} and \cite[Subsection 12.3]{Z21}). This map is well-defined in the projectivized homotopy category which means that it is independent, up to homotopy, from different choices made in its definition and the following diagram is commutative up to homotopy and up to an overall factor $t^z$.
\begin{center}
	$\begin{CD}
	\underline{CF}^-(\mathcal{H}_1,\mathbf{w},\mathfrak{t}|_{Y_1};\Lambda_{\omega_1}) @> \underline{f}^-_{W,\Gamma,\mathfrak{t};\Lambda_\omega}>> \underline{CF}^-(\mathcal{H}_2,\mathbf{w},\mathfrak{t}|_{Y_2};\Lambda_{\omega_2})\\		@V\underline{\Psi}_{(\mathcal{H}_1,J_1)\longrightarrow(\mathcal{H}'_1,J'_1),\mathfrak{t}|_{Y_1}}VV @VV\underline{\Psi}_{(\mathcal{H}_2,J_2)\longrightarrow(\mathcal{H}'_2,J'_2),\mathfrak{t}|_{Y_2}}V\\
	\underline{CF}^-(\mathcal{H}'_1,\mathbf{w}, \mathfrak{t}|_{Y_1};\Lambda_{\omega_1}) @>\underline{f}^-_{W,\Gamma,\mathfrak{t};\Lambda_\omega}>> \underline{CF}^-(\mathcal{H}'_2,\mathbf{w}, \mathfrak{t}|_{Y_2};\Lambda_{\omega_2})
	\end{CD}$
\end{center}
where $\underline{\Psi}_{(\mathcal{H}_i,J_i)\longrightarrow(\mathcal{H}'_i,J'_i),\mathfrak{t}|_{Y_i}}$ denotes a transition map between two pairs of weakly admissible Heegaard diagrams $(\mathcal{H}_i,J_i)$ and $(\mathcal{H}'_i,J'_i)$ of $(Y_i,\mathbf{w}_i,\mathfrak{t}|_{Y_i})$ (see \cite[Section 7]{JZ}). 

\begin{remark}
	 When $\Gamma$ is a path (which is a connected graph with two vertices) in $W$ and $M$ is a $\mathbb{Z}[H^1(Y_1;\mathbb{Z})]$-module, there is an induced $\mathbb{Z}[H^1(Y_2;\mathbb{Z})]$-module $M(W)$ (see \cite[Subsection 2.7]{OZSVATH}) and the map 
	\[
	\underline{F}^\circ_{W,\mathfrak{t};M}:\underline{HF}^\circ(Y_1,\mathfrak{t}|_{Y_1};M)\rightarrow \underline{HF}^\circ(Y_2,\mathfrak{t}|_{Y_2};M(W))
	\]
	is uniquely defined up to multiplication by ±1, left-translation by an element of $H^1(Y_1;\mathbb{Z})$, and right translation by an element of $H^1(Y_2;\mathbb{Z})$ (see \cite[Theorem 3.8]{OZSVATH}). When $[\omega]\in H^2(W;\mathbb{R})$ and $M=\Lambda_{\omega_1}$, where $\omega_1=\omega|_{Y_1}$, then $M(W)=\Lambda_{\omega_2}$, where $\omega_2=\omega|_{Y_2}$.
\end{remark}

In the following, we briefly review the constructions of the induced cobordism maps. Note that there is a handle decomposition of $W$ obtained by  $i$-handle attachments, $0\leq i\leq 4$, away from the basepoints:
\[
W= W_0\cup W_1\cup W_2\cup W_3\cup W_4
\]
where $W_i$ consists of $i$-handles, $0\leq i\leq4$. This decomposition is induced by a Morse function $f$. Let $v$ denote a gradient-like vector field of $f$. We can assume that the graph $\Gamma$ is disjoint from the critical points of $f$, descending manifolds of index 1 critical points, ascending manifolds of index 3 critical points, and from both ascending and descending manifolds of index 2 critical points
(see  \cite[Lemma 10.4]{Z15}).
A ribbon graph cobordism $(W,\Gamma)$ induces two maps
\[
F^A_{W,\Gamma,\mathfrak{t}}, F^B_{W,\Gamma,\mathfrak{t}}:CF^-(Y_1,\mathbf{w}_1,\mathfrak{s}_1)\rightarrow CF^-(Y_2,\mathbf{w}_2,\mathfrak{s}_2),
\]
where $\mathfrak{t}\in \text{Spin}^c(W)$, $\mathfrak{s}_i=\mathfrak{t}|_{Y_i}$, $i=1,2$. These maps are diffeomorphism invariants of $(W,\Gamma)$, up to $\mathbb{F}_2[U]$-equivariant chain homotopy (see \cite[Theorem A]{Z15}).  When $\Gamma$ consists of a
collection of paths, each connecting $\mathbf{w}_1$ to $\mathbf{w}_2$, we have $F^A_{W,\Gamma,\mathfrak{t}}\simeq F^B_{W,\Gamma,\mathfrak{t}}$ (see \cite[Theorem B]{Z15}). 
$F^A_{W,\Gamma,\mathfrak{t}}$ and $F^B_{W,\Gamma,\mathfrak{t}}$ are a composition of maps induced by $i$-handle attachments, $0\leq i\leq4$, and \emph{graph action maps} $\mathfrak{U}_\mathcal{G}$ and $\mathfrak{B}_\mathcal{G}$, respectively, associated to $(Y,\mathcal{G}=(\Gamma,\mathbf{w}_0,\mathbf{w}_1))$ where $\mathcal{G}$ is an embedding of $\Gamma$ in $Y$ (see \cite{Z15} and the following discussion). There is a twisted cobordism map
\[ \underline{F}^B_{W,\Gamma,\mathfrak{t};\Lambda_\omega}:\underline{CF}^-(Y_1,\mathbf{w}_1,\mathfrak{s}_1;\Lambda_{\omega_1})\rightarrow \underline{CF}^-(Y_2,\mathbf{w}_2,\mathfrak{s}_2;\Lambda_{\omega_2}),
\]
which is described in \cite{Z21} and we briefly review the construction by describing the maps involved in the definition. We describe the maps that will be used in later sections in more details. We denote the map induced on the chain complex by $\underline{f}^-_{W,\Gamma,\mathfrak{t};\Lambda_\omega}$ and the map induced on the homology by $\underline{F}^-_{W,\Gamma,\mathfrak{t};\Lambda_\omega}$.
\begin{enumerate}
	\item $4$-dimensional $0$- and $4$-handle attachment are equivalent to adding or removing $(S^3,w_0)$. In this case, the associated map is induced by the isomorphism 
	\[
	\begin{split}
		\underline{CF}^-(Y\amalg S^3,\mathbf{w}\cup\{w_0\},\mathfrak{s}\oplus\mathfrak{s}_0;\Lambda_{\omega\oplus\omega_0})&\cong \underline{CF}^-(Y,\mathbf{w};\Lambda_\omega)\otimes_\Lambda \underline{CF}^-(S^3,w_0,\mathfrak{s}_0;\Lambda_{\omega_0})\\
		&\cong \underline{CF}^-(Y,\mathbf{w};\Lambda_\omega)\otimes_{\Lambda}\Lambda[U_{0}].
	\end{split}
	\]
	Here $\omega_0$ denotes the zero cohomology class. In fact, if $\mathcal{H}=(\Sigma,\boldsymbol{\alpha},\boldsymbol{\beta},\mathbf{w})$ is a weakly admissible Heegaard diagram for $Y$ and  $\mathcal{H}_0=(\Sigma_0,\{\alpha_0\},\{\beta_0\},w_0)$ is the standard Heegaard diagram for $S^3$ where $\alpha_0\cap\beta_0=\{x_0\}$, the associated map corresponding to a 0-handle attachment is 
	\[
	\begin{split}
		\underline{CF}^-(\mathcal{H},\mathfrak{s};\Lambda_\omega)&\rightarrow \underline{CF}^-(\mathcal{H}\amalg\mathcal{H}_0,\mathfrak{s}\otimes\mathfrak{s}_0;\Lambda_\omega\otimes\Lambda_{\omega_0})\\
		\mathbf{y}&\mapsto\mathbf{y}\otimes x_0,
	\end{split}
	\] 
	where $\mathbf{y}\in\mathbb{T}_{\boldsymbol{\alpha}}\cap\mathbb{T}_{\boldsymbol{\beta}}$ is a generator of $\underline{CF}^-(\mathcal{H},\mathfrak{s};\Lambda_\omega)$.
	\item If $(W,\Gamma):(Y,\mathbf{w})\rightarrow (Y',\mathbf{w})$ is a restricted graph cobordism given by $i$-handle additions, $i=1,2,3$, the induced maps are define in  \cite[Section 4]{OZSVATH} (see \cite[Sections 6 and 7]{JZ},  for more details).
\end{enumerate}

For a general ribbon graph cobordism, the graph action map $\mathfrak{U}_\mathcal{G}$ is a composition of maps associated to some \emph{elementary graphs}. For the case of restricted graph cobordisms, we only have one type of these elementary graphs called \emph{translations} for which the embedding $\mathcal{G}=(\Gamma,\mathbf{w}_0,\mathbf{w}_1)$ is such that $|\mathbf{w}_0|=|\mathbf{w}_1|$ and each edge of $\Gamma$ connects a vertex of $\mathbf{w}_0$ to a vertex of $\mathbf{w}_1$. There are two maps involved in the definition of a twisted $\underline{\mathfrak{U}}_\mathcal{G}$ associated to a translation: \emph{twisted free stabilization maps} $\underline{S}^\pm_w$ (which correspond to adding or removing a base point), and \emph{twisted relative homology maps} $\underline{A}_\lambda$ (which correspond to a path $\lambda$ between two basepoints $w_1$ and $w_2$).  A twisted graph action map associated to a restricted graph cobordism, where $\mathcal{G} = (\Gamma, \mathbf{w}_0, \mathbf{w}_1)$ corresponds to a translation, is as follows:
\[
\underline{\mathfrak{U}}_\mathcal{G}=(\prod_{w\in\mathbf{w}_0}\underline{S}^-_{w})\circ(\prod_{e\in E(\Gamma)} \underline{A}_e)\circ(\prod_{w\in\mathbf{w}_1}\underline{S}^+_{w}).
\]
See \cite{Z15} for more details. Here we recall the definition of twisted free stabilization maps (see \cite{Z21} for a definition of twisted relative homology maps).

\begin{figure}[h!]
	\def\svgwidth{5.5cm}
	\begin{center}
		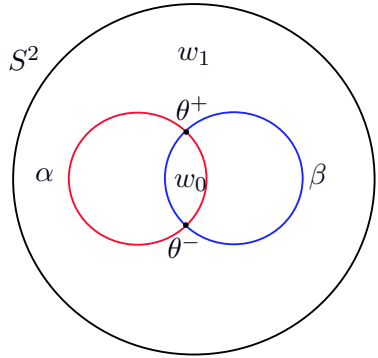
		\caption{The doubly pointed Heegaard diagram $\mathcal{H}_0$ used in the definition of free stabilization maps. The Heegaard surface is a sphere.}
		\label{HD-Sphere-2}
	\end{center}
\end{figure}

Let $z\in\mathbf{w}$ and $\mathcal{H}_0=(S^2,\{\alpha\},\{\beta\},\{w_0,w_1\})$ denote a Heegaard diagram with $\alpha\cap\beta=\{\theta^+,\theta^-\}$ where $\theta^+$ is the generator with the higher grading (see Figure \ref{HD-Sphere-2}). Let $\mathcal{H}'=(\Sigma,\boldsymbol{\alpha}\cup\{\alpha\},\boldsymbol{\beta}\cup\{\beta\},\mathbf{w}\cup\{w_0\})$ denote the Heegaard diagram for $(Y,\mathbf{w}\cup\{w_0\})$ obtained by the connected sum of $\mathcal{H}$ and $\mathcal{H}_0$. To perform the connected sum, we remove the interiors of two disks around $z\in\Sigma$ and $w_1\in S^2$ and identify the boundaries of these disks. We denote the  basepoint in this region by $z$.   
For appropriate choices of almost complex structure (see  \cite[\textcolor{blue}{Definition 6.2 and Proposition 6.3}]{Z15}), one can define the free stabilization maps
\[
\begin{aligned}
\underline{S}^+_{w_0}:\underline{CF}^-(Y,\mathbf{w},\mathfrak{s};\Lambda_\omega)&\rightarrow\underline{CF}^-(Y,\mathbf{w}\cup\{w_0\},\mathfrak{s};\Lambda_\omega)\\
&\underline{S}^+_{w_0}(\mathbf{x})=\mathbf{x}\times\theta^+,
\end{aligned}
\]
and
\[
\begin{aligned}
	&\underline{S}^-_{w_0}:\underline{CF}^-(Y,\mathbf{w}\cup\{w_0\},\mathfrak{s};\Lambda_\omega)\rightarrow\underline{CF}^-(Y,\mathbf{w},\mathfrak{s};\Lambda_\omega)\\
	&\underline{S}^-_{w_0}(\mathbf{x}\times\theta^-)=\mathbf{x},\ \ \ \underline{S}^-_{w_0}(\mathbf{x}\times\theta^+)=0.
\end{aligned}
\]
The twisted free stabilization maps are well-defined chain maps by the same argument as the
twisted 1-handle and 3-handle maps (see \cite[Section 7.3]{JZ} and \cite[Section 6]{Z15} for more details).

Note that when $W$, $Y_1$, and $Y_2$ are non-empty and connected, $\mathbf{w}_i=\{w_i\}$, $i=1,2$, and $\Gamma$ is a path from $w_1$ to $w_2$, $\underline{F}^-_{W,\Gamma,\mathfrak{s};\omega}$ coincides with the map defined by Ozsv\'{a}th and Szab\'{o} (see \cite{OZSVATH}).

\subsection{Absolute grading and surgery exact sequence}\label{Gradings}
For this subsection, we work over $\mathbb{Z}[H^1(Y;\mathbb{Z})]$.
When $\mathfrak{s}$ is torsion, $HF^\circ(Y,\mathfrak{s})$ is generated by homogeneous elements $\mathfrak{U}$ which are equipped  with a relative grading function
\[
\text{gr}:\mathfrak{U}\times\mathfrak{U}\rightarrow\mathbb{Z}.
\]
This relative $\mathbb{Z}$-grading is lifted to an absolute $\mathbb{Q}$-grading
\[
\widetilde{\text{gr}}:\mathfrak{U}\rightarrow\mathbb{Q},
\]
in the sense that if $\eta,\xi\in\mathfrak{U}$, 
\[
\text{gr}(\eta,\xi)=\widetilde{\text{gr}}(\eta)-\widetilde{\text{gr}}(\xi).
\]
See \cite[Theorem 7.1]{OZSVATH} for more details. A map induced by a cobordism $W$ from $Y_1$ to $Y_2$, which is endowed with a $\text{Spin}^c$ structure $\mathfrak{t}$ whose restriction to $Y_i$ is torsion, shifts the grading by the formula
\begin{equation}\label{shift-gr}
\widetilde{\text{gr}}(F_{W,\mathfrak{t}}(\xi))-\widetilde{\text{gr}}(\xi)=\frac{c_1(\mathfrak{t})^2-2\chi(W)-3\sigma(W)}{4},
\end{equation}
where $\sigma(W)$ denotes the signature of the intersection form 
\[
\begin{aligned}
Q:&H_2(W)\otimes H_2(W)\rightarrow\mathbb{Z}\\ &(a,b)\mapsto a\cdot b.
\end{aligned}
\] 
Note that every class $a\in H_2(W)$ can  be represented by a smoothly embedded
closed oriented surface and $a\cdot b$ denotes the signed intersection of the surfaces that represent $a$ and $b$. To define $c_1(\mathfrak{t})^2$, we recall some facts.  From the long exact sequence
\[
H^2(W,\partial W;\mathbb{Q})\rightarrow H^2(W;\mathbb{Q})\rightarrow H^2(\partial W;\mathbb{Q})
\]
and the assumption that $c_1(\mathfrak{t}|_{\partial W})$ is torsion, or equivalently $c_1(\mathfrak{t}|_{\partial W})=0\in H^2(\partial W;\mathbb{Q})$, there is $x\in H^2(W,\partial W;\mathbb{Q})$ whose image is $c_1(\mathfrak{t})$. From here, $c_1(\mathfrak{t})^2$ is defined as $\langle c_1(\mathfrak{t})\cup x,[W]\rangle$. 

Let $K$ be a knot in an integer homology sphere $Y$ and $Y_i=Y_i(K)$ be obtained from $Y$ by performing an $i$-surgery along $K$, $i=0,1$. There is a long exact sequence 
\[
\dots\rightarrow HF^\circ(Y)\xrightarrow{f_1^\circ} HF^\circ(Y_0)\xrightarrow{f_2^\circ} HF^\circ(Y_1)\xrightarrow{f_3^\circ}HF^\circ(Y)\rightarrow\dots
\]
where the maps $f_i^\circ$, $i=1,2,3$, are induced from cobordisms $W_i$ which correspond to attaching 2-handles (see \cite[Theorem 9.1]{OZSVATHAbsolutelygraded}). Therefore, for each cobordism, we have $\chi(W_i)=1$.  
To compute $\sigma(W_i)$, recall that when a 4-manifold $W$ with boundary is obtained  by integral surgery on a framed link $\mathcal{L}$ in $Y$, its intersection form is isomorphic to the linking matrix of $\mathcal{L}$ (when $Y=S^3$, this is Theorem 6.2 of \cite{Saveliev+2012} which is true when we replace $S^3$ with a homology sphere). In the case at study, since $W_1$  is obtained by attaching a 0-framed 2-handle to $Y\times[0,1]$, $\sigma(W_1)=0$. Also $W_2$ is obtained from $Y_0\times[0,1]$ by attaching a single 2-handle where we perform -1-surgery on an unknot $U$ which links $K$ geometrically once.  As a result, $W_1\cup_{Y_0} W_2$ is a 4-manifold obtained by performing integral surgery on the two component link $K\amalg U$ in $Y$, with linking matrix $\big(\begin{smallmatrix}
	-1&1\\
	1&0
\end{smallmatrix}\big)$. Therefore, $\sigma(W_1)+\sigma(W_2)=\sigma(W_1\cup W_2)=0$ (see \cite[Theorem 5.3]{Kirby}). This shows that $\sigma(W_2)=0$.

Note that $f^\circ_i$, $i=1,2$, are sums of maps associated to $W_i$, which means that 
\[
f^\circ_i=\sum_{\mathfrak{t}\in\text{Spin}^c(W_i)}F^\circ_{W_i,\mathfrak{t}}.
\]
Also, $\text{Spin}^c$
structures on both $W_1$ and $W_2$ are uniquely
determined by their restrictions to $Y_0$.
Let $\mathfrak{t}_0$ denote the $\text{Spin}^c$ structure over $Y_0$ with trivial first Chern class. Using Formula \eqref{shift-gr} and the fact that $\sigma(W_i)=0$, $i=1,2$, the component of $f_1^\circ$ that maps into $HF^\circ(Y_0,\mathfrak{t}_0)$ shifts degree by $-1/2$ and the restriction of $f_2^\circ$ to $HF^\circ (Y_0,\mathfrak{t}_0)$ shifts degree by $-1/2$. Also, $f^\circ_3$ is non-increasing on the grading.  Indeed, $f^\circ_3=F^\circ_{W_3}$ where $W_3$ is the reverse of a cobordism $\overline{W_3}$ obtained from $-Y$ by attaching a 2-handle along $K$ with framing 1. Since $\sigma(\overline{W_3})=1$ and the orientation of $W_3$ is the opposite of the orientation of $\overline{W_3}$, we have $\sigma(W_3)=-1$. Therefore,
\begin{equation}\label{EQ1}
\widetilde{\text{gr}}(F^\circ_{W_3,\mathfrak{t}}(\xi))-\widetilde{\text{gr}}(\xi)=\frac{c_1(\mathfrak{t})^2+1}{4}.
\end{equation}

We recall the computation of $c_1(\mathfrak{t})^2$ when $W$ is a cobordism between two homology spheres with the intersection form $(p)$. First recall that
$H^2(W)$ acts freely and transitively on $\text{Spin}^c(W)$ and the first Chern class is a map  $c_1:\text{Spin}^c(W)\rightarrow H^2(W)$ such that $c_1(h+\mathfrak{t})=2h+c_1(\mathfrak{t})$, $h\in H^2(W)$, $\mathfrak{t}\in\text{Spin}^c(W)$. A cohomology class $x\in H^2(W)$ is called a \emph{characteristic element} if $x\cup h\equiv h\cup h\mod 2$, for all $h\in H^2(W,\partial W)$.
Since $c_1(\mathfrak{t})\equiv w_2(\mathfrak{t})\mod 2$ ($w_i$ denotes the $i^\text{th}$-Stiefel-Whitney class), the Wu formula shows that $c_1(\mathfrak{t})$ is a characteristic element:
\[
\begin{split}
	&w_2=\sum_{i+j=2}Sq^i(v_j)=v_2+v_1\cup v_1=v_2,\\
	&w_2\cup h=v_2\cup h=Sq^2(h)=h\cup h.
\end{split}
\]
Here $h\in H^2(W,\partial W;\mathbb{F}_2)$, $v_i$ denotes the relative Wu class and $Sq^i$ is the Steenrod square (see \cite[Section 7]{Kervaire}). Note that $v_1=w_1=0$. 

Let $H^2(W)\cong\mathbb{Z}\langle y\rangle$ and $H^2(W,\partial W)\cong\mathbb{Z}\langle z\rangle$. Therefore, $c_1(\mathfrak{t})\cup z\equiv z\cup z\mod 2$.
From the long exact sequence
\[
0=H^1(\partial W)\rightarrow H^2(W,\partial W)\rightarrow H^2(W)\rightarrow H^2(\partial W)\cong\mathbb{Z}_p\rightarrow H^3(W,\partial W)\cong H_1(W)=0,
\]
the generator $z$ is mapped to $py\in H^2(W)$.
Let $c_1(\mathfrak{t})=ky$, for some $k\in\mathbb{Z}$. Since $c_1(\mathfrak{t})$ is a characteristic element, we have 
\[
\begin{split}
	& \langle c_1(\mathfrak{t})\cup z,[W]\rangle\equiv \langle z\cup z,[W]\rangle\mod2,\\
	& \langle z\cup z,[W]\rangle=p,\\
	&\langle y\cup z,[W]\rangle=1/p\langle z\cup z,[W]\rangle.
\end{split}
\]
This shows that $k\equiv p\mod2$. Let $c_1(\mathfrak{t}_j)=(2j+p)y$. We want to compute $c_1(\mathfrak{t})^2$.
\begin{equation}\label{EQ2} c_1(\mathfrak{t}_j)^2=\frac{(2j+p)^2}{p}\langle y\cup z,[W]\rangle=\frac{(2j+p)^2}{p^2}\langle z\cup z,[W]\rangle=\frac{(2j+p)^2}{p}.
\end{equation}
For the cobordism $W_3$, we have $p=-1$. If use the formulas in \eqref{EQ2} and \eqref{EQ1}, we see that $f^\circ_3$ is non-increasing on the grading.

To state the twisted surgery exact sequence, note that we have $\mathbb{Z}[H^1(Y_0)]\cong\mathbb{Z}[t,t^{-1}]$ where $t$ is a generator of $H^1(Y_0;\mathbb{Z})$. If $M$ denotes a $\mathbb{Z}[U]$-module, there is an induced $\mathbb{Z}[U,t,t^{-1}]$-module structure on $M[t,t^{-1}]=M\otimes\mathbb{Z}[t,t^{-1}]$. There is a $\mathbb{Z}[U,t,t^{-1}]$-equivariant long exact sequence
\[
\dots\rightarrow HF^+(Y)[t,t^{-1}]\xrightarrow{\underline{f}_1^+}\rightarrow\underline{HF}^+(Y_0)\xrightarrow{\underline{f}_2^+}\rightarrow HF^+(Y_1)[t,t^{-1}]\xrightarrow{\underline{f}_3^+}\dots
\] 
(see \cite[Theorem 9.21]{OS2}). There is a similar exact sequence for the hat version.

Note that for a closed oriented, $\text{Spin}^c$ 3-manifold $(N,\mathfrak{s})$, there is an $\mathfrak{s}$-grading on $\mathbb{Z}[H^1(N;\mathbb{Z})]$ 
\[
\text{gr}_\mathfrak{s}(x)=-\langle c_1(\mathfrak{s})\cup x,[N]\rangle.
\]
where $x\in H^1(N;\mathbb{Z})$, which makes $\mathbb{Z}[H^1(N;\mathbb{Z})]$ into a graded ring (see \cite[Definition 3.1]{JM}). This equips the fully twisted Heegaard Floer homology groups with a relative $\mathbb{Z}$-grading.

\begin{defn}\cite[Definition 3.2]{JM}
	 Let $(\Sigma, \boldsymbol{\alpha}, \boldsymbol{\beta}, z)$ be a pointed Heegaard triple describing the 3-manifold $N$.
	Fix a $\textnormal{Spin}^c$ structure $\mathfrak{s}$ for $N$ and an additive assignment $\{A_{\mathbf{x},\mathbf{y}}\}$ for the diagram. The relative
	$\mathbb{Z}$ grading between generators $[\mathbf{x},i]$ and $[\mathbf{y},j]$ for $\underline{CF}^\circ(N, \mathfrak{s}; \mathbb{Z}[H^1(N;\mathbb{Z})])$ is defined by
	\[
	\underline{\text{\normalfont gr}}([\mathbf{x}, i], [\mathbf{y}, j]) = \mu(\phi) + 2(i - j)-  2n_z(\phi) - \langle c_1(\mathfrak{s})\cup A_{\mathbf{x},\mathbf{y}}(\phi), [N ]\rangle,
	\]
	where $\phi$ is any element of $\pi_2(\mathbf{x}, \mathbf{y})$. More generally, if $r_1, r_2 \in\mathbb{Z}[H^1(N;\mathbb{Z})]$ are homogeneous elements,
	then we set
	\[
	\underline{\text{\normalfont gr}}(r_1 \cdot [\mathbf{x}, i], r_2 \cdot [\mathbf{y}, j]) =\underline{ \text{\normalfont gr}}([\mathbf{x}, i], [\mathbf{y}, j]) + \text{\normalfont gr}_\mathfrak{s}(r_1) -\text{\normalfont gr}_\mathfrak{s}(r_2).
	\]
\end{defn}

Here, since $Y$ and $Y_1$ are homology spheres, $\underline{\text{gr}}=\text{gr}$. Also, for the $\text{Spin}^c$ manifold $(Y_0,\mathfrak{t}_0)$, $\underline{\text{gr}}=\text{gr}$. Therefore, the component of $\underline{f}_1^+$ that maps into $\underline{HF}^+(Y_0,\mathfrak{t}_0)$ shifts degree by $-1/2$ and the restriction of $\underline{f}^+_2$ to $\underline{HF}^+ (Y_0,\mathfrak{t}_0)$ shifts degree by $-1/2$. Also, $\underline{f}^+_3$ is non-increasing on the grading.

\subsection{Sums and compositions of twisted cobordism maps} \label{SumsCom} 
\textbf{Completion of Heegaard Floer homology groups. }Let $A$ be a ring and $I$ be an ideal of $A$. Denote the completion of $A$ with respect to this ideal by $\boldsymbol{A}$. If $M$ is an $A$-module, the completion $\boldsymbol{M}$ of $M$ is an $\boldsymbol{A}$-module. As an example, when $A=\mathbb{F}_2[U]$ (resp. $A=\mathbb{F}_2[U,U^{-1}]$) and $I=(U)$, then $\boldsymbol{A}=\mathbb{F}_2\llbracket U\rrbracket$, the ring of formal power series (resp. $\boldsymbol{A}=\mathbb{F}_2\llbracket U,U^{-1}]$ ring of semi-infinite Laurent polynomials). $\boldsymbol{CF}^-$ (resp. $\boldsymbol{CF}^\infty$), the completion of the chain complex $CF^-$ (resp. $CF^\infty$), is a chain complex with the same generators as of $CF^-$ (resp. $CF^\infty$) with coefficients in $F_2\llbracket U\rrbracket$ (resp.  $\mathbb{F}_2\llbracket U,U^{-1}]$). Since completion is an exact functor, $\boldsymbol{HF}^-$, the homology group of $\boldsymbol{CF}^-$, is the completion of $HF^-$. Similarly, $\boldsymbol{HF}^\infty$ is the completion of $HF^\infty$. One can define the completions for $\widehat{HF}$ and $HF^+$ similarly but since the action of $U$ is nilpotent on each generator, the completion matches with the original group. Therefore, we have the following exact sequence for a closed, oriented three manifold $Y$:
\begin{equation}\label{ES-1}
\cdots\longrightarrow\boldsymbol{HF}^-(Y)\longrightarrow\boldsymbol{HF}^\infty(Y)\longrightarrow HF^+(Y)\longrightarrow\cdots.
\end{equation}
Let \[\boldsymbol{HF}_{\text{red}}(Y):=\text{Coker}(\boldsymbol{HF}^\infty(Y)\rightarrow HF^+(Y))\cong\text{Ker}(\boldsymbol{HF}^-(Y)\longrightarrow\boldsymbol{HF}^\infty(Y)).
\]

Note that for the definition of $HF^-$ and $HF^\infty$, one needs to restrict to strongly admissible Heegaard diagrams. However, Lemma 4.13 of \cite{OS} shows that the differentials for $\boldsymbol{CF}^-$ and $\boldsymbol{CF}^\infty$ are finite even for weakly admissible Heegaard diagrams. Also, note that for any torsion Spin$^\text{c}$ structure $\mathfrak{s}$, the groups
$HF^-(Y,\mathfrak{s})$ and $HF^\infty(Y,\mathfrak{s})$ are determined by $\boldsymbol{HF}^-$ and $\boldsymbol{HF}^\infty$, respectively, where $\mathfrak{s}\in\text{Spin}^\text{c}(Y)$ (see \cite[Section 2]{Manolescu2010HeegaardFH}). But for a non-torsion Spin$^\text{c}$ structure $\mathfrak{s}$, as $(1-U^N)HF^\infty(Y,\mathfrak{s})=0$, for some $N\geq1$ and $1-U^N$ is invertible in  $\mathbb{Z}\llbracket U\rrbracket$, we have $\boldsymbol{HF}^\infty(Y,\mathfrak{s})=0$. Therefore, from the long exact sequence in \ref{ES-1}, we have
$\boldsymbol{HF}^-(Y,\mathfrak{s})\cong HF^+(Y,\mathfrak{s})$, where $\mathfrak{s}$ is non-torsion (see \cite[Section 2]{Manolescu2010HeegaardFH}).

When we consider the completion of Heegaard Floer homology groups with respect to $(U)$ and with coefficients in $\Lambda_\omega$, we also have the long exact sequence
\[
\cdots\longrightarrow\underline{\boldsymbol{HF}}^-(Y,\mathfrak{s};\Lambda_\omega)\longrightarrow\underline{\boldsymbol{HF}}^\infty(Y,\mathfrak{s};\Lambda_\omega)\longrightarrow \underline{HF}^+(Y,\mathfrak{s};\Lambda_\omega)\longrightarrow\cdots.
\]
By Corollary 8.7 of \cite{JM}, when $(Y,\mathfrak{s})$ is a $\text{Spin}^c$ manifold, $c_1(\mathfrak{s})$ is torsion, and $[\omega]\neq0$, $\underline{\boldsymbol{HF}}^\infty(Y,\mathfrak{s};\Lambda_\omega)=0$. This implies that 
\[
\underline{\boldsymbol{HF}}^-(Y,\mathfrak{s};\Lambda_\omega)\cong \underline{HF}^+(Y,\mathfrak{s};\Lambda_\omega)
\]
It is worth mentioning that for torsion Spin$^\text{c}$ structures,
\[
\underline{HF}^-(Y,\mathfrak{s};\Lambda_\omega)\cong \underline{HF}^+(Y,\mathfrak{s};\Lambda_\omega).
\]
This follows from the fact that $\underline{HF}^\infty(Y,\mathfrak{s};\Lambda_\omega)=0$ when $c_1(\mathfrak{s})$ is torsion and $[\omega]\neq0$ (see \cite[Corollary 8.5]{JM}). Therefore,
\[
\underline{\boldsymbol{HF}}_{\text{red}}(Y,\mathfrak{s};\Lambda_\omega)=\underline{HF}^+(Y,\mathfrak{s};\Lambda_\omega)\cong\underline{\boldsymbol{HF}}^-(Y,\mathfrak{s};\Lambda_\omega),
\]
and
\[ \underline{HF}_{\text{red}}(Y,\mathfrak{s};\Lambda_\omega)=\underline{HF}^+(Y,\mathfrak{s};\Lambda_\omega)\cong\underline{HF}^-(Y,\mathfrak{s};\Lambda_\omega).
\]

\textbf{Sums and compositions of twisted cobordism maps.} Let $(W,\Gamma)$ be a restricted graph cobordism from $(Y_1,\mathbf{w}_1)$ to $(Y_2,\mathbf{w}_2)$. Let $[\omega]\in H^2(W;\mathbb{R})$ and $\omega_i=\omega|_{Y_i}$, $i=1,2$. Assume further that $\mathfrak{T}$ be a subset of $\{\mathfrak{t}\in\text{Spin}^\text{c}(W)\mid\mathfrak{t}|_{Y_1}=\mathfrak{s}_1, \mathfrak{t}|_{Y_2}=\mathfrak{s}_2\}$. The map
\[
\underline{\boldsymbol{f}}^-_{W,\Gamma,\mathfrak{T};\Lambda_\omega}:\underline{\boldsymbol{CF}}^-(\mathcal{H}_1,\mathbf{w}_1,\mathfrak{s}_1;\Lambda_{\omega_1})\rightarrow\underline{\boldsymbol{CF}}^-(\mathcal{H}_2,\mathbf{w}_2,\mathfrak{s}_2;\Lambda_{\omega_2})
\]
is defined as $\sum\limits_{\mathfrak{t}\in\mathfrak{T}}\underline{\boldsymbol{f}}^-_{W,\Gamma,\mathfrak{t};\Lambda_\omega}$, where $\mathcal{H}_i$, $i=1,2$, is a Heegaard diagram for $Y_i$. Since the ground ring is $\Lambda
\llbracket U\rrbracket$, there is no need for $\mathfrak{T}$ to be finite. Also note that by  \cite[Theorem 3.3]{OZSVATH}, this sum is well defined over the power series ring. 

\begin{lemma}\label{lem-Zemke}\cite[Lemma 12.4]{Z21}
	The map $\underline{\boldsymbol{f}}^-_{W,\Gamma,\mathfrak{T};\Lambda_\omega}$ is well defined up to an overall factor $t^z$, $z\in\mathbb{R}$. 		
\end{lemma}
This lemma states that when we change the auxiliary data used to construct the cobordism map, each $\underline{\boldsymbol{f}}^-_{W,\Gamma,\mathfrak{T};\Lambda_\omega}$ changes
by the same factor of $t^z$, where $z$ is independent of $\mathfrak{T}$.
In general, we need to require that all the Spin$^\text{c}$ structures $\mathfrak{t}\in\mathfrak{T}$ restrict to the same Spin$^\text{c}$ structures on $Y_1$ and $Y_2$. Indeed, when $\omega\neq0$, the twisted
Heegaard Floer groups  $\underline{HF}^\circ(Y,\mathfrak{s};\Lambda_\omega)$  are well-defined up to an overall factor $t^z$. Therefore, these groups are natural only when we restrict to one Spin$^\text{c}$ structure at a time. In general, one needs to check if the maps 
\[
\underline{f}^\circ_{W,\Gamma;\Lambda_\omega}=\sum_{\mathfrak{s}\in\text{Spin}^c(W)}\underline{f}^\circ_{W,\Gamma,\mathfrak{s};\Lambda_\omega}.
\]
are well defined.

We will need the following version of composition law.
\begin{lemma}\label{Composition}\cite[Lemma 12.5]{Z21}
Suppose that
\[
(W_1,\Gamma_1):(Y,\mathbf{w})\rightarrow(Y',\mathbf{w}'),\ \ \text{and},\ \ (W_2,\Gamma_2):(Y',\mathbf{w}')\rightarrow(Y'',\mathbf{w}'')
\]
are graph cobordisms. Write $(W,\Gamma)=(W_2,\Gamma_2)\circ(W_1,\Gamma_1)$. If $\mathfrak{S}_1\subset\textnormal{Spin}^c(W_1)$ and $\mathfrak{S}_2\subset\textnormal{Spin}^c(W_2)$
are sets of $\textnormal{Spin}^c$
structures which all have the same restrictions to $Y$, $Y'$, and $Y''$, write $\mathfrak{S}(W,\mathfrak{S}_1,\mathfrak{S}_2)$
for the set of $\textnormal{Spin}^c$
structures on $W$ which restrict to an element of $\mathfrak{S}_1$ and an element of $\mathfrak{S}_2$. Then
\[
\underline{\boldsymbol{F}}^B_{W,\Gamma,\mathfrak{S};\Lambda_\omega}\dot{\simeq}\underline{\boldsymbol{F}}^B_{W_2,\Gamma_2,\mathfrak{S}_2;\Lambda_{\omega_2}}\circ\underline{\boldsymbol{F}}^B_{W_1,\Gamma_1,\mathfrak{S}_1;\Lambda_{\omega_1}}.
\]
where 
$\omega_i=\omega|_{W_i}$. If $\omega$ vanishes on one of the 3-manifolds $Y$, $Y'$, or $Y''$, we may relax the requirement
that all elements of $\mathfrak{S}_1$ and $\mathfrak{S}_2$ have the same restriction to that 3-manifold.
\end{lemma}
\begin{remark}\label{Rmrk-Completed-vs-usual-Novikov}
	In the following section, we work with 3-manifolds $Y$ such that $\underline{HF}^\circ(Y,\mathfrak{s};\Lambda_\omega)\cong0$ where $0\neq[\omega]\in H^2(Y;\mathbb{R})$ and $c_1(\mathfrak{s})\neq0$. Therefore, the above argument shows that Lemma \ref{lem-Zemke} is still valid  when we replace $\underline{\boldsymbol{f}}^-_{W,\Gamma,\mathfrak{T};\Lambda_\omega}$ with $\underline{f}^-_{W,\Gamma,\mathfrak{T};\Lambda_\omega}$ (for such 3-manifolds). Also, Lemma \ref{Composition} is true when we replace $\underline{\boldsymbol{F}}^B_{W,\Gamma,\mathfrak{S};\Lambda_\omega}$ with $\underline{F}^B_{W,\Gamma,\mathfrak{S};\Lambda_\omega}$.
\end{remark}

\section{Proof of Theorem \ref{MainThm}}\label{Section 2}
The proof is a slight modification of the proof of Theorem \ref{FLoer-Ex}. We use the following result from \cite{A-P2010}. In this section, $\mathcal{H}$, $\mathcal{H}_0$, etc., will be used to describe Heegaard diagrams that are possibly different from the same notations used in the previous section.

\begin{thm}[\cite{A-P2010}]\label{Thm-2}
	Suppose $Y$ is a closed, oriented 3-manifold that fibers over the circle with torus fiber $\mathcal{F}$, and $[\omega]\in H^2(Y;\mathbb{R})$ is a cohomology class such that $\omega(\mathcal{F})\neq0$. Then we have an isomorphism of $\Lambda$-modules
	\[
	\underline{HF}^+(Y;\Lambda_\omega)\cong\Lambda.
	\]
\end{thm} 

Let $\mathcal{H}=(\Sigma,\boldsymbol{\alpha},\boldsymbol{\beta},\mathbf{w}_0\cup\{z\})$ be a multi-pointed Heegaard diagram for a multi-pointed 3-manifold $(Y,\mathbf{w}_0\cup\{z\})$ and $\mathcal{H}_0=(S^2,\alpha,\beta,\{w_0,w_1\})$ denote the doubly pointed Heegaard diagram in Figure \ref{HD-Sphere-2} with $\alpha\cap\beta=\{\theta^\pm\}$ where $\text{gr}(\theta^+,\theta^-)=1$. Let \[\mathcal{H}_1=(\Sigma\# S^2,\boldsymbol{\alpha}\cup\{\alpha\},\boldsymbol{\beta}\cup\beta,\mathbf{w}_0\cup\{w_0,w_1\})
\] 
be the connected sum of $\mathcal{H}$ and $\mathcal{H}_0$  where 
the connected sum is formed at the points $z$ and $w_1$. Let $[\omega]\in H^2(Y;\mathbb{R})$. The following theorem is a twisted version of \cite[Proposition 6.5]{OS3}.

\begin{thm}\label{Thm-1}
	Let $\mathcal{H}$, $\mathcal{H}_0$, and $\mathcal{H}_1$ be as above, then $\underline{CF}^-(\mathcal{H}_1,\mathfrak{s};\Lambda_\omega)$ is identified with the mapping cone of
	\[
	\underline{CF}^-(\mathcal{H},\mathfrak{s};\Lambda_\omega)\otimes_{\Lambda}\Lambda[U_{w_0}]\langle\theta^-\rangle\xrightarrow{U_{w_0}-U_{z}}\underline{CF}^-(\mathcal{H},\mathfrak{s};\Lambda_\omega)\otimes_\Lambda\Lambda[U_{w_0}]\langle\theta^+\rangle.
	\]
\end{thm}

Before stating the proof, we recall the definition of mapping cone. Let $(A,\partial_A)$ and $(B,\partial_B)$ denote two $\mathbb{F}_2$-graded chain complexes and $f:A\rightarrow B$ be a chain map. The mapping cone $M(f)$ is the chain complex with the underlying group $A\oplus B$ endowed with a differential $(a,b)\mapsto(\partial_A(a),(-1)^{\text{gr}(a)}f(a)+\partial_B(b))$.
There is a short exact sequence of chain maps
\begin{equation}\label{MappingCone}
	0\rightarrow B\xrightarrow{\iota} Cone(f)\xrightarrow{\pi} A\rightarrow0,
\end{equation}
such that the connecting homomorphism in the associated long exact sequence is the map on
homology induced by $f$.
\begin{proof}
	We modify the proof of  \cite[Proposition 6.5]{OS3} for $\omega$-twisted coefficients. Choose an almost complex structure such that the neck length of the connected sum is very big. We have $\underline{CF}^-(\mathcal{H}_1,\mathfrak{s};\Lambda_\omega)\cong C_{\theta^+}\oplus C_{\theta^-}$, where $C_{\theta^\pm}$ denote all the generators $\mathbf{x}\times\theta^\pm$, $\mathbf{x}\in\mathbb{T}_{\boldsymbol{\alpha}}\cap\mathbb{T}_{\boldsymbol{\beta}}$.
	
	Let $\phi\in\pi_2(\mathbf{x}\times\theta^+,\mathbf{y}\times\theta^+)$ or $\phi\in\pi_2(\mathbf{x}\times\theta^-,\mathbf{y}\times\theta^-)$. We have $\phi=\phi_1\#\phi_2$ where $\phi_1\in\pi_2(\mathbf{x},\mathbf{y})$ is a Whitney disk in $\mathcal{H}$ and $\phi_2\in\pi_2(\theta^+,\theta^+)$ or $\phi_2\in\pi_2(\theta^-,\theta^-)$ is a Whitney disk in $\mathcal{H}_0$. As discussed in \cite{OS3}, the only case where $\phi$ has a holomorphic representative is when
	\[
	\mu(\phi_2)=2n_{w_1}(\phi),\ \ n_{w_0}(\phi)=0.
	\]
	Using \cite[Theorem 5.1]{OS3}, the $\mathbf{y}\times \theta^\pm$ component of $\underline{\partial}^-(\mathbf{x}\times\theta^\pm)$ is 
	\[
	\sum_{\substack{\{\phi_1\in\pi_2(\mathbf{x},\mathbf{y})|\mu(\phi_1)=1\}}}\sum_{\substack{u_1\in\widehat{\mathcal{M}}(\phi_1)}}\sum_{\substack{\phi_2\in\pi_2(\theta^\pm,\theta^\pm)}}\#\{u_2\in\mathcal{M}(\phi_2)|\rho_1(u_1)=\rho_2(u_2)\}t^{\omega([\phi_1])+\omega([\phi_2])}.
	\]
	Here $\rho_i:\mathcal{M}(\phi_i)\rightarrow Sym^k(\mathbb{D})$ with $\rho_1(u)=u^{-1}(\{z\}\times Sym^{d_1-1}(\Sigma))$, $\rho_2(u)=u^{-1}(\{w_1\})$, $d_1$ is the number of curves in $\boldsymbol{\beta}$, $k=n_z(\phi_1)=n_{w_1}(\phi_2)$, and $\mathbb{D}$ is the unit disk in $\mathbb{C}$.
	Note that $\partial D(\phi_2)$ is a union of $\alpha$- and $\beta$-curves in $\mathcal{H}_0$ and $[\phi_2]$, which is obtained by coning the $\alpha$ and $\beta$ boundaries of $D(\phi_2)$, is a sum of copies of $S^2$, which bounds a 3-ball in $Y$. Therefore $\omega([\phi_2])=0$. This shows that the $\mathbf{y}\times \theta^\pm$ component of $\underline{\partial}^-(\mathbf{x}\times\theta^\pm)$ is identified with the $\mathbf{y}$ component of $\underline{\partial}^-(\mathbf{x})$. 
	
	The argument in the proof of Proposition 6.5 of \cite{OS3} shows that if  $\mathbf{y}\times\theta^-$ appears in $\underline{\partial}^-(\mathbf{x}\times\theta^+)$, domains of the Whitney disks that connect  $\mathbf{y}\times\theta^-$  to $\mathbf{x}\times\theta^+$ consist of the two bigons $D_1$ and $D_2$ connecting $\theta^-$ to $\theta^+$ away from the  basepoints in $\mathcal{H}_0$. The same argument as the previous case shows that $\omega([D_1-D_2])=0$. Therefore the corresponding disks in the differential cancel and the $C_{\theta^-}$ component of $\underline{\partial}^-(\mathbf{x}\times\theta^+)$ is trivial.
	
	Finally, as shown in the proof of Proposition 6.5 \cite{OS3}, the $C_{\theta^+}$ component of $\underline{\partial}(\mathbf{x}\times \theta^-)$ is given by 
	\[
	U_{z}\ t^{\omega([\phi_1])}\ \mathbf{x}\times\theta^++U_{w_0}\ t^{\omega([\phi_2])}\ \mathbf{x}\times\theta^+,
	\]
	where $\phi_2$ is a disk such that its domain is a bigon connecting $\theta^-$ to $\theta^+$ which contains $w_0$ and $\phi_1$ is a disk such that its domain is $\Sigma\#D_2$ where $D_2$ is a bigon in $\mathcal{H}_0$ connecting $\theta^-$ to $\theta^+$ and containing $w_1$. $\partial D(\phi_2-\phi_1)$ is a union of $\alpha$- and $\beta$-curves in $\mathcal{H}_1$ and $[\phi_2-\phi_1]$ is homotopic to the Heegaard surface $\Sigma\#S^2$. Therefore $\omega([\phi_1])=\omega([\phi_2])$ (as $\Sigma\#S^2$ is the boundary of a handlebody in $Y$). This completes the proof. 
\end{proof}
 \begin{cor}\label{C-1}
 	If $U_0\neq U_1$ in Theorem \ref{Thm-1},
 	 $\underline{S}^+_{w_0}$ induces isomorphism on $\underline{HF}^\circ$, 
 	 $\circ\in\{+,-,\infty\}$, and $\underline{S}^-_{w_0}$ induces zero maps. 
 \end{cor}
\begin{proof}
When $\circ\in\{-,\infty\}$, from Theorem \ref{Thm-1} and the long exact sequence induced from the short exact sequence in \ref{MappingCone}, we have 
\[
\underline{HF}^-(\mathcal{H}_1,\mathfrak{s};\Lambda_\omega)\cong\frac{\underline{HF}^-(\mathcal{H},\mathfrak{s};\Lambda_\omega)[U_{w_0}]\langle\theta^+\rangle}{U_{w_0}-U_z}.
\]
Then the result is obtained directly from the definition of free stabilization maps. The case $\circ\in\{+\}$ is implied from the fact that the free stabilization maps are compatible with the long exact sequence in \ref{ES-1}. 
\end{proof}

Let  $\mathcal{F}=S^1\times S^1$ and $Y=\mathcal{F}\times S^1$. Assume that $W$ is a cobordism which is obtained from  $\mathcal{F}\times D^2$ by removing a $4$-ball. 
More precisely, let $D_0\subset D^2$ be a small disk.  
Remove a small neighborhood of $D_0$ from $\mathcal{F}\times D^2$ to obtain the cobordism $W$ from $S^3$ to $Y$ such that $(\{p\}\times(\overline{D^2\setminus D_0}))\cap S^3$ is the circle $\partial D_0$ and $p\in \mathcal{F}$. 
Suppose $w_0\in S^3$ and $w\in Y$ are basepoints.  Let $\Gamma\subset W$ be any path connecting $w_0$ to $w$. Let $\eta=\{p\}\times (S^1=\partial D^2)\subset Y$ and $[\omega]=d\cdot\text{PD}([\eta])\in H^2(Y;\mathbb{R})$ where $d\neq0$.  Suppose $[\bar{\omega}]\in H^2(W;\mathbb{R})$ be $d\cdot\text{PD}(\{p\}\times (\overline{D^2\setminus D_0}))$. 

\begin{lemma}\label{Lem-2}
	With the above notation, the map 
	\[
	\underline{F}^-_{W,\Gamma;\Lambda_{\bar{\omega}}}:\Lambda[U_0]\cong \underline{HF}^-(S^3,w_0,\mathfrak{s}_0;\Lambda)\rightarrow \underline{HF}^-(Y,w;\Lambda_\omega)\cong\Lambda
	\]
	is well-defined up to an overall factor and is a nonzero map. 
\end{lemma}
\begin{figure}[h]
	\begin{center}
		\def\svgwidth{0.5\textwidth}
		\fontsize{13}{15}\selectfont
		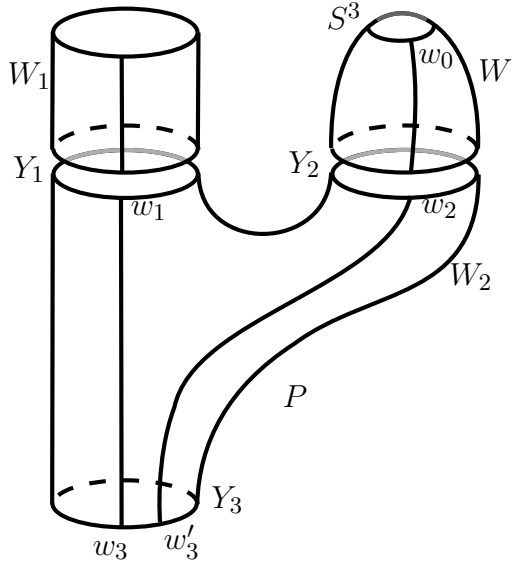
	\end{center}
\caption{This is Figure A.4 of \cite{phdthesis}.}
\label{Fig-Pair-of-Pants}
\end{figure}
\begin{proof}
	The proof is a straightforward modification of the proof of Lemma A.1.33 from \cite{phdthesis}. 
	Let $P$ denote a pair of pants as shown in Figure \ref{Fig-Pair-of-Pants}.
	Let $Y_i$, $i=1,2,3$, be three copies of $Y$, $w_i\in Y_i$, $i=1,2,3$, and $w'_3\in Y_3$ be basepoints. Let $W_1=Y_1\times I$, $W_2= \mathcal{F}\times P$, $\Gamma_1=\{w_1\}\times I\subset W_1$, and $\Gamma_2\subset W_2$ denote two paths  connecting $w_1$ to $w_3$ and $w_2$ to $w'_3$. Suppose $\eta_i$ and $\omega_i$, $i=1,2,3$, are copies of $\eta$ and $\omega$ in $Y_i$, $[\bar{\omega}_2]=d\cdot\text{PD}([\{p\}\times P])\in H^2(W_2;\mathbb{R})$, and $[\bar{\omega}_1]=d\cdot\text{PD}([\mu_1\times I])\in H^2(W_1;\mathbb{R})$.

	Note that the map $\underline{F}^-_{W_2,\Gamma_2;\Lambda_{\bar{\omega}_2}}$
	is a sum of maps. By Theorem \ref{Thm-2}, Remark \ref{Rmrk-Completed-vs-usual-Novikov}, and Lemma \ref{lem-Zemke}, 
	this map is well-defined up to an overall factor. Similarly, $\underline{F}^-_{W_1\sqcup W,\Gamma_1\cup
	\Gamma;\Lambda_{\bar{\omega}_1\otimes\bar{\omega}}}$ and $\underline{F}^-_{W,\Gamma;\Lambda_{\bar{\omega}}}$ are well-defined maps. Let $(W',\Gamma')$ be the composition of the cobordism $W_2$ with $W_1\sqcup W$ and then a 0-handle attachment.  $W'$ is the product cobordism $Y\times I'$, where $I'$ is a closed connected interval, and $\Gamma'$ is a two component path. This gives the free stabilization $\underline{S}^+_{w_0}$ 	(see \cite[Remark 1.19]{phdthesis}) which induces an isomorphism by Corollary \ref{C-1}. More precisely, let $[\omega']\in H^2(W';\mathbb{R})$ be $PD(\eta\times I')$, then $\omega'|_{W_2}=\overline{\omega}_2$, $\omega'|_{W_1\amalg W}=\overline{\omega}_1\oplus\overline{\omega}$, and $\underline{F}^-_{W',\Gamma';\Lambda_{\omega'}}=\underline{S}^+_{w_0}$ (see Lemma \ref{Composition}). We have  
\begin{equation*}
	\begin{split}
		\underline{F}^-_{W_1\sqcup W,\Gamma_1\cup
		\Gamma;\Lambda_{\bar{\omega}_1\oplus\bar{\omega}}}=
		\underline{F}^-_{W_1,\Gamma_1;\Lambda_{\bar{\omega}_1}}\otimes_{\Lambda}\underline{F}^-_{W,\Gamma;\Lambda_{\bar{\omega}}}=\\
		\text{id}_{W_1}\otimes_{\Lambda}\underline{F}^-_{W,\Gamma;\Lambda_{\bar{\omega}}}.
	\end{split}
\end{equation*}
Therefore, $\underline{F}^-_{W,\Gamma;\Lambda_{\bar{\omega}}}$ is nonzero. 
\end{proof}

Let  $\mathcal{H}_2=(\Sigma\# S^2,\boldsymbol{\alpha}'\cup\{\alpha'\},\boldsymbol{\beta}'\cup\{\beta'\},\mathbf{w}_0\cup\{w_0,w_1\}) $ be the connected sum of $\mathcal{H}$ and $\mathcal{H}_0$ where the  basepoint $z$ is replaced with $w_0$ (Here $\boldsymbol{\alpha}'$, $\alpha'$, $\boldsymbol{\beta}'$, and $\beta'$ are small hamiltonian isotopies of $\boldsymbol{\alpha}$, $\alpha$, $\boldsymbol{\beta}$, and $\beta$.) (See Figure \ref{Connected sum change base}).  

\begin{figure}[h!]
	\def\svgwidth{18cm}
	\begin{center}
		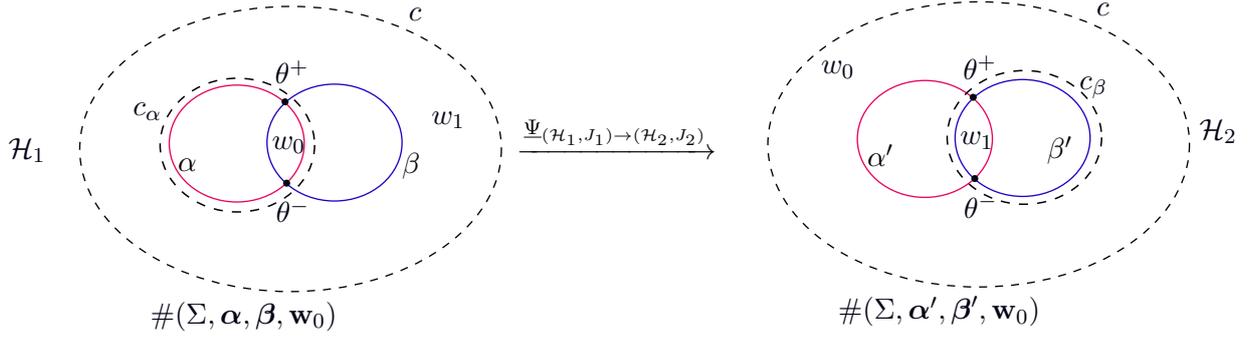
		\caption{This is Figure 14.1 of \cite{Z15}.}
		\label{Connected sum change base}
	\end{center}
\end{figure}

\begin{prop}\label{Prop-1}
	With the above notation, there are choices of almost complex structures $J_1$ and $J_2$ on $\mathcal{H}_1$ and $\mathcal{H}_2$, respectively, such that 
	\[
	\underline{\Psi}_{(\mathcal{H}_1,J_1)\rightarrow(\mathcal{H}_2,J_2),\mathfrak{s}}\doteq\begin{bmatrix}
		*&0\\
		(\underline{\Phi}_{\boldsymbol{\alpha}\rightarrow\boldsymbol{\alpha}'}^{\boldsymbol{\beta}'})_{U_{w_1}}^{U_z\rightarrow U_{w_0}}\circ(\sum U_{w_0}^iU_{w_1}^j(\underline{\partial}_{i+j+1})_{U_{w_0,w_1}})\circ (\underline{\Phi}_{\boldsymbol{\alpha}}^{\boldsymbol{\beta}\rightarrow\boldsymbol{\beta}'})_{U_{w_0}}^{U_z\rightarrow U_{w_1}}&*
	\end{bmatrix}.
	\]
\end{prop}
Here  $\underline{\Phi}_{\boldsymbol{\alpha}}^{\boldsymbol{\beta}\rightarrow\boldsymbol{\beta}'}$ denotes the transition map from $\underline{CF}^-(\Sigma,\boldsymbol{\alpha},\boldsymbol{\beta},\mathbf{w}_0\cup\{z\},\mathfrak{s};\Lambda_\omega)$ to $\underline{CF}^-(\Sigma,\boldsymbol{\alpha},\boldsymbol{\beta}',\mathbf{w}_0\cup\{z\},\mathfrak{s};\Lambda_\omega)$. $\underline{\Phi}_{\boldsymbol{\alpha}\rightarrow\boldsymbol{\alpha}'}^{\boldsymbol{\beta}'}$ is defined similarly. If $G : C_1 \rightarrow C_2$
is a map of $\Lambda[U_z]$-modules, write $G^{U_z\rightarrow U_w}$ for the induced map
\[
G^{U_z\rightarrow U_w} := G \otimes id_{\Lambda[U_z,U_w]/(U_w-U_z)}.
\]
If $R$ is a ring of characteristic 2 and $G : C_1 \rightarrow C_2$ is a map of $R$-modules, write $(G)_{U_w}$ for the map
of $R \otimes _{\Lambda} \Lambda[U_w]$-modules
\[
G_{U_w} := G \otimes id_{\Lambda[U_w]}
: C_1 \otimes_{\Lambda} \Lambda[U_w] \rightarrow C_2 \otimes_ {\Lambda} \Lambda[U_w].
\]
$\underline{\Psi}_{(\mathcal{H}_1,J_1)\rightarrow(\mathcal{H}_2,J_2),\mathfrak{s}}$ is the transition map from the chain complex $\underline{CF}^-(\mathcal{H}_1,\mathfrak{s};\Lambda_\omega)$ to $\underline{CF}^-(\mathcal{H}_2,\mathfrak{s};\Lambda_\omega)$. Note that the generators in $\mathcal{H}_i$, $i=1,2$, can be written as $C_{\theta^+}\oplus C_{\theta^-}$ and in the matrix presentation for  $\underline{\Psi}_{(\mathcal{H}_1,J_1)\rightarrow(\mathcal{H}_2,J_2),\mathfrak{s}}$, 
the first row and column of this matrix correspond to $\theta^+$, and the second row and column to $\theta^-$. 
In general, if $V$ denotes the 2-dimensional vector space
$\langle\theta^+,\theta^-\rangle$ and $G$ is a homomorphism
\[
G : C_1 \otimes_\Lambda V \rightarrow C_2 \otimes_\Lambda V,
\]
then we will write $G$ as a $2\times 2$ block matrix as stated above.
Also we can write the differential on $\underline{CF}^-(\mathcal{H}_i,\mathfrak{s};\Lambda_\omega)$ as $\underline{\partial}_{\mathcal{H}_{i}}=\sum_{j=0}^{\infty}\underline{\partial}_jU_z^j$.

Proposition \ref{Prop-1} is the twisted version of Theorem 14.1 of \cite{Z15} and the proof easily modifies to the twisted case. We briefly restate the steps of the proof according to \cite{Z15}.

Let $c$ denote a curve along the connected sum neck of $\Sigma$ and $S^2$, $c_\alpha$ be a small isotopy of $\alpha$, and $c_\beta$ be a small isotopy of $\beta$ such that $c_\alpha\cap c=\emptyset$ and $c_\beta\cap c=\emptyset$ (see Figure \ref{Connected sum change base}). Let $\mathbf{T}=(T_1,T_2)$ and $J_\alpha(\mathbf{T})$ denote an almost complex structure which is stretched along $c$ and $c_\alpha$ with neck-lengths $T_1$ and $T_2$ respectively. Lemma 14.2 of \cite{Z15} proves that when the neck-lengths are large enough, the relative neck-lengths $T_1$ and $T_2$ do not matter. More precisely, the twisted version of this lemma is as follows.

Let $J_1$ and $J_2$ denote two almost complex structures on $\Sigma\times[0,1]\times\mathbb{R}$. There is an almost complex structure $J'$ on $\Sigma\times[0,1]\times\mathbb{R}$ such that it agrees with $J_1$ on $\Sigma\times[0,1]\times(-\infty,-1]$ and with $J_2$ on $\Sigma\times[0,1]\times[1,+\infty)$ and a transition map $\underline{\Psi}_{(\mathcal{H},J_1)\rightarrow(\mathcal{H},J_2),\mathfrak{s}}$ is defined by counting $J'$-holomorphic disks $\phi$ of zero Maslov index in $\Sigma\times[0,1]\times\mathbb{R}$, weighted by $t^{\omega([\phi])}$, where $[\phi]$ is the associated 2-chain to $\phi$ obtained by coning off the $\alpha$- and the $\beta$- boundaries of $D(\phi)$. We say that $J'$ interpolates between $J_1$ and $J_2$.

\begin{lemma}\label{Lem-1}
	Let $\tilde{\mathcal{H}}$ be one of $\mathcal{H}_i$, $i=1,2$, or $\mathcal{H}_{1.5}$ (see Figure \ref{HD-halfway}).
	There is a constant
	$N$ such that if $\mathbf{T}$ and $\mathbf{T}'$ are two pairs of neck lengths, all of whose components are greater than
	$N$, then there is a non-cylindrical almost complex structure $J'$ interpolating $J_{\alpha}(\mathbf{T})$ and $J_\alpha(\mathbf{T}')$, respectively, $J_{\beta}(\mathbf{T})$ and $J_\beta(\mathbf{T}')$,
	satisfying
	\[
	\begin{matrix}
		\underline{\Psi}_{J_\alpha(\mathbf{T})\rightarrow J_\alpha(\mathbf{T'}),\mathfrak{s}}=\underline{\Psi}_{J',\mathfrak{s}}=\begin{pmatrix}
			id & 0\\
			0& id
		\end{pmatrix},\\
	\underline{\Psi}_{J_\beta(\mathbf{T})\rightarrow J_\beta(\mathbf{T'}),\mathfrak{s}}=\underline{\Psi}_{J',\mathfrak{s}}=\begin{pmatrix}
		id & 0\\
		0& id
	\end{pmatrix}.
	\end{matrix}
	\]
\end{lemma}
\begin{proof}
	The proof is almost the same as the proof of the Lemma 14.2 of \cite{Z15}. We consider the case $\tilde{\mathcal{H}}=\mathcal{H}_1$. The general plan is to take two sequences $\mathbf{T}_{i}=(T_{1,i},T_{2,i})$, $\mathbf{T}'_{i}=(T'_{1,i},T'_{2,i})$ of pairs of neck-lengths. Let $J'_i$ denote the almost complex structure which interpolates $J_{\alpha}(\mathbf{T}_{i})$ and $J_\alpha(\mathbf{T}_i')$ and is non-cylindrical only in a neighborhood of $c$ and $c_\alpha$. $\underline{\Psi}_{J'_i,\mathfrak{s}}$ counts the number of $J'_i$-holomorphic disks with Maslov index zero and has a matrix presentation as 
	\[
	\begin{pmatrix}
		A_i&B_i\\
		C_i&D_i
	\end{pmatrix}.
	\]

	Let $\phi\#\phi_0\in\pi_2(\mathbf{x}\times x,\mathbf{y}\times y)$. By the index formula (Formula 14.5 of \cite{Z15}), 
	\[
	\mu(\phi\#\phi_0)=\mu(\phi)+\text{gr}(x,y)+2 m_2(\phi_0),
	\]
	$A_i$ and $D_i$ correspond to generators  with $\text{gr}(x,y)=0$, $B_i$ correspond to generators with $\text{gr}(x,y)=-1$, and $C_i$ correspond to generators with $\text{gr}(x,y)=1$. Here $m_2(\phi_0)$ denotes the coefficient of the disk $\phi_0$ in the region containing $w_0$ (see Figure \ref{Connected sum change base} on the left). As the two pairs of neck-lengths approach infinity, for each sequence of $J'_i$-holomorphic disks $u_i$, we can find sub-sequences of disks  $u_i^l$, $u_i^m$, and $u_i^r$ which approach into broken holomorphic curves in  $(S^2\setminus\{p_0\})\times[0,1]\times\mathbb{R}$, in $S^1\times\mathbb{R}\times[0,1]\times\mathbb{R}$ ($S^1\times[0,1]$ is the cylinder with boundary components $c_\alpha$ and $c$), and in  $\Sigma\setminus\{p\}\times[0,1]\times\mathbb{R}$,  where $p_0$
	and $p$ denote the connected sum points corresponding to the circles $c_\alpha$ and $c$, respectively. Consequently, $\phi$ and $\phi_0$ admit broken homomorphic representatives on $(\Sigma,\boldsymbol{\alpha},\boldsymbol{\beta})$ and $(S^2,\alpha,\beta^l)$ where the curve $\beta^l$ is the result of cutting $\beta$ along its
	intersection with $c_\alpha$, and then collapsing the ends to a point.  
	The index formula and the fact that $\phi$ has a broken holomorphic representative imply that for the disks which contribute to $A_i$ and $D_i$, $\phi$ and $\phi_0$ are constant and therefore $\phi\#\phi_0$ has a unique $\tilde{J}_i$-holomorphic representative and $\omega([\phi])=\omega([\phi_0])=0$ which means that $A_i=D_i=1$. 
	The same argument as in the proof of Lemma 14.2 of \cite{Z15} shows that $B_i=C_i=0$.
\end{proof}

Let $\mathcal{H}_{1.5}$ denote the Heegaard diagram in Figure \ref{HD-halfway}. The following lemma is similar to Theorem \ref{Thm-1} but with different basepoints and different choices of almost complex structures. 

\begin{figure}[h!]
	\def\svgwidth{10cm}
	\begin{center}
		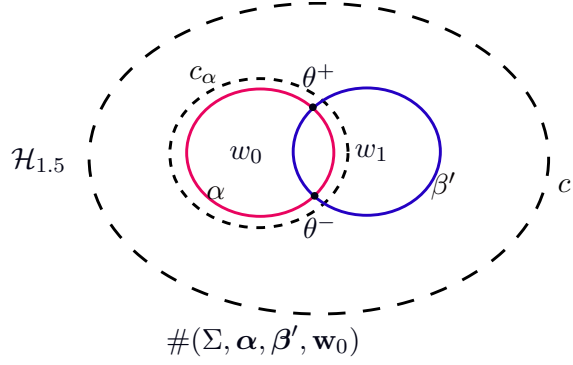
		\caption{This is Figure 14.2 of \cite{Z15}.}
		\label{HD-halfway}
	\end{center}
\end{figure}

\begin{lemma}[\cite{Z15}, Lemma 14.3]\label{Lemma-1}
	Let $J_\alpha$ denote an almost complex structure on the Heegaard diagram $\mathcal{H}_{1.5}$ which is stretched along $c$ and
	$c_\alpha$ (see Figure \ref{HD-halfway}). For sufficiently large neck lengths along $c$ and $c_\alpha$, we have
	\[
	\underline{\partial}_{\mathcal{H}_{1.5},J_\alpha,\mathfrak{s}}=\begin{pmatrix}
		(\underline{\partial}_{\mathcal{H}_0})^{U_z\rightarrow U_{w_1}}& t^w(U_{w_1}+U_{w_0})\\
		0& (\underline{\partial}_{\mathcal{H}_0})^{U_z\rightarrow U_{w_1}}
	\end{pmatrix},
	\]
for some $w\in\mathbb{R}$. If $J_\beta$ denotes an analogous almost complex structure stretched sufficiently along $c$ and $c_\beta$, then
	\[
\underline{\partial}_{\mathcal{H}_{1.5},J_\beta,\mathfrak{s}}=\begin{pmatrix}
	(\underline{\partial}_{\mathcal{H}_0})^{U_z\rightarrow U_{w_0}}& t^w(U_{w_1}+U_{w_0})\\
	0& (\underline{\partial}_{\mathcal{H}_0})^{U_z\rightarrow U_{w_0}}
\end{pmatrix},
\]
for some $w\in\mathbb{R}$.
\end{lemma}
\begin{proof}
	The proof of Lemma 14.3 of \cite{Z15} works here with a slight modification and we briefly restate it. By Lemma \ref{Lem-1}, the relative neck-lengths of $c$ and $c_\alpha$ does not affect the computation when the neck-lengths are sufficiently large. Let 
	\[
		\underline{\partial}_{\mathcal{H}_{1.5},J_\alpha,\mathfrak{s}}=\begin{pmatrix}
		A&B\\
		C&D
	\end{pmatrix}.
	\]
\begin{figure}[h!]
	\def\svgwidth{15cm}
	\begin{center}
		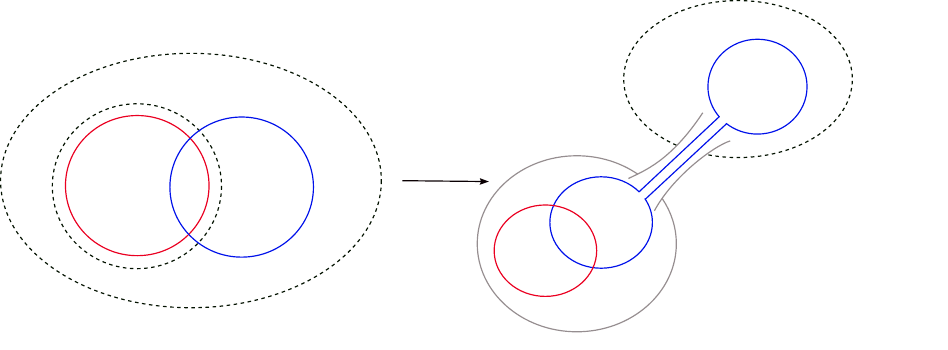
		\caption{Figure 14.5 of \cite{Z15}.}
		\label{Fig-6}
		
	\end{center}
\end{figure}

	For a homotopy class of a disk $\phi\#\phi_0\in\pi_2(\mathbf{x}\times x,\mathbf{y}\times y)$, we have the index formula
	\[
	\mu(\phi\#\phi_0)=\mu(\phi)-\text{gr}(x,y)+2m_2(\phi_0).
	\]
	Let $m_1$, $m_2$, $n_1$ and $n_2$ denote the multiplicities of the regions of $(S^2,\alpha,\beta',\{w_0,w_1\})$ shown in Figure \ref{Fig-6}. Classes with $\text{gr}(x,y)=1$ contribute to $C$, classes with $\text{gr}(x,y)=0$ contribute to $A$ and $D$, and classes with $\text{gr}(x,y)=-1$ contribute to $B$. By the index formula and stretching the neck along $c$, one can check that the only classes with $\text{gr}(x,y)=-1$ are the classes with domains one of the two bigons $D_i$, $i=1,2$, containing the  basepoints $w_0$ and $w_1$. Let $\phi_i$, $i=1,2$, denote the classes with $D(\phi_i)=D_i$. As  was argued in the proof of Theorem \ref{Thm-1}, $\omega([\phi_1])=\omega([\phi_2])$. Therefore, $B\doteq U_{w_0}+U_{w_1}$.
	
	To compute $A$ and $D$, note that stretching the neck around $c$ and then using the above index formula implies that $\mu(\phi)=1$ and $m_2(\phi_0)=0$. For a fixed neck-length around $c_\alpha$, the argument in the proof of Theorem \ref{Thm-1} computes the functions $A$ and $D$ without the consideration about the  basepoints. This is proved similarly to the untwisted case. In fact, as the neck around $c_\alpha$ approaches infinity, a two punctured sphere with one $\beta$ curve and no $\alpha$ curves degenerates (see Figure \ref{Fig-6}). By considering the contributions of the limiting curves on this degenerate diagram, one can prove that $m_1(\phi_0)\leq n_2(\phi_0)$. This observation and the index formula result in $n_1(\phi_0)=0$ and $n_2(\phi_0)=m_1(\phi_0)$. This implies that
	any curve which makes non-trivial contribution to $A$ or $D$ is counted with a factor of $U^{m_1(\phi)}_{w_1}$ and
	no factor of $U_{w_0}$. This completes the computation of $A$ and $D$. 
	
	For a sufficiently large neck-length around $c$ and using the index formula, one can check that the only possible disks that contribute to $C$ are either a disk $\psi:=\phi\#\phi_0$ where $\phi$ is a constant class and $\phi_0$ is a disk where domain is a bigon with $m_2(\phi_0)=1$ or a disk $\phi'\#\phi'_0$ with $\mu(\phi')=2$ and $m_2(\phi'_0)=0$.
	For each neck-length $T_i(c_\alpha)$ along $c_\alpha$, there is a point $d_i(\phi_0')\in[0,1]\times \mathbb{R}$ associated with the bigon $\phi_0'$ in $\mathcal{H}_0$ such that
	as $T_i(c_\alpha)$ approaches infinity, $d_i(\phi_0')$ approaches $\{0\}\times\mathbb{R}$. Therefore, using the maximum modulus principal, the domain of each limit curve $u$ is a component of $\Sigma\setminus\boldsymbol{\beta}'$. As $c$ stretches, one can use transversality results to show that $u$ is a representative for $\phi'$. 
    and we have $\omega([\phi'\#\phi'_0])=\omega([\psi])$. One can also check that $\#\widehat{\mathcal{M}}(\phi'\#\phi'_0)=\#\widehat{\mathcal{M}}(\psi)$. This proves that $C=0$.
\end{proof}

\begin{proof}[Proof of Proposition \ref{Prop-1}] Let $J_1$ and $J_2$ be the almost complex structures $J_\alpha$ and $J_\beta$  in Lemma \ref{Lemma-1}. One can use the computation for the differential in Lemma \ref{Lemma-1} above and Lemma 14.5 of \cite{Z15} to show that the transition map $\underline{\Psi}_{(\mathcal{H}_{1.5},J_\alpha)\rightarrow (\mathcal{H}_{1.5},J_\beta),\mathfrak{s}}$ is given by
	\[
	\begin{pmatrix}
		id&0\\
		*&id
	\end{pmatrix},
	\]
where $*\doteq\sum_{i,j\geq0}U^i_{w_0}U^j_{w_1}(\underline{\partial}_{i+j+1})_{U_{w_0},U_{w_1}}$. One can decompose $\underline{\Psi}_{(\mathcal{H}_{1},J_\alpha)\rightarrow (\mathcal{H}_{2},J_\beta),\mathfrak{s}}$ as
\[
\underline{\Psi}_{(\mathcal{H}_{1},J_\alpha)\rightarrow (\mathcal{H}_{2},J_\beta),\mathfrak{s}}=\underline{\Psi}_{(\mathcal{H}_{1.5},J_\beta)\rightarrow (\mathcal{H}_{2},J_\beta),\mathfrak{s}}\circ \underline{\Psi}_{(\mathcal{H}_{1.5},J_\alpha)\rightarrow (\mathcal{H}_{1.5},J_\beta),\mathfrak{s}}\circ \underline{\Psi}_{(\mathcal{H}_{1},J_\alpha)\rightarrow (\mathcal{H}_{1.5},J_\alpha),\mathfrak{s}}
\]
The twisted versions of Proposition 14.6 and Proposition 14.8 of \cite{Z15} hold (with the same proof) and one can use them to compute $\underline{\Psi}_{(\mathcal{H}_{1},J_\alpha)\rightarrow (\mathcal{H}_{1.5},J_\alpha),\mathfrak{s}}$ and  $\underline{\Psi}_{(\mathcal{H}_{1.5},J_\beta)\rightarrow (\mathcal{H}_{2},J_\beta),\mathfrak{s}}$. In fact, in matrix notation, $\underline{\Psi}_{(\mathcal{H}_{1},J_\alpha)\rightarrow (\mathcal{H}_{2},J_\beta)}$ equals
\begin{equation*}
	\begin{split}
&\begin{pmatrix}
	(\underline{\Phi}_{\boldsymbol{\alpha}\rightarrow\boldsymbol{\alpha}'}^{\boldsymbol{\beta}'})_{U_{w_1}}^{U_z\rightarrow U_{w_0}}&0\\
	0&(\underline{\Phi}_{\boldsymbol{\alpha}\rightarrow\boldsymbol{\alpha}'}^{\boldsymbol{\beta}'})_{U_{w_1}}^{U_z\rightarrow U_{w_0}}
\end{pmatrix}\begin{pmatrix}
	id&0\\
	*&id
\end{pmatrix}\begin{pmatrix}
(\underline{\Phi}_{\boldsymbol{\alpha}}^{\boldsymbol{\beta}\rightarrow\boldsymbol{\beta}'})_{U_{w_0}}^{U_z\rightarrow U_{w_1}}&0\\
0&(\underline{\Phi}_{\boldsymbol{\alpha}}^{\boldsymbol{\beta}\rightarrow\boldsymbol{\beta}'})_{U_{w_0}}^{U_z\rightarrow U_{w_1}}
\end{pmatrix}.
\end{split}
\end{equation*}This completes the proof. See \cite[Section 14]{Z15} for more details.
	
\end{proof}

At this step, we can prove the twisted version of Lemma A.1.35 of \cite{phdthesis}.

\begin{figure}[h!]
	\def\svgwidth{17cm}
	\begin{center}
		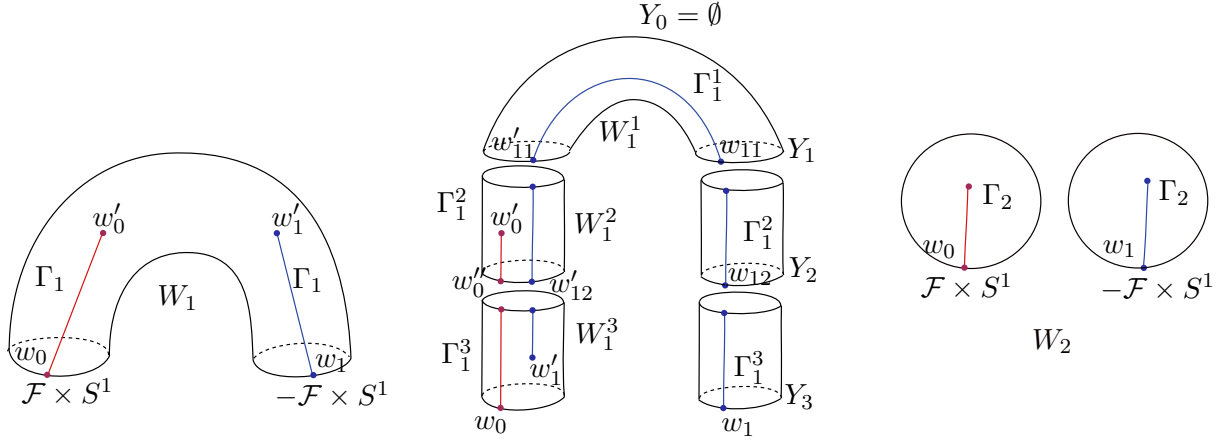
		\caption{This is Figure $A.5$ in \cite{phdthesis}, related to Lemma \ref{Main-Lem}}
		\label{Fig-12}
	\end{center}
\end{figure}

\begin{lemma}\label{Main-Lem}
	Let $Y=\mathcal{F}\times S^1$, $[\omega]=d\cdot\textnormal{PD}([\eta])\in H^2(Y;\mathbb{R})$ where $\mathcal{F}=S^1\times S^1$ $\eta=\{p\}\times S^1$, $p\in \mathcal{F}$, $d\neq0$. Let $W_1=Y\times I$ be a cobordism from $\emptyset$ to $Y\amalg(-Y)$ and $[\overline{\omega}_1]=d\cdot\textnormal{PD}([\eta\times I])\in H^2(W;\mathbb{R})$. Let $w_0\in Y$, $w_1\in -Y$, $w'_0,w'_1\in W_1$ and $\Gamma_1\subset W_1$ denote a two component path that connects $w_i$ and $w'_i$, $i=0,1$, as shown in the left of Figure \ref{Fig-12}. Let $(W_2,\Gamma_2)$ denote a graph cobordism, with $W_2\cong (\mathcal{F}\times D^2)\amalg(-\mathcal{F}\times D^2)$, from $\emptyset$ to  $Y\amalg(-Y)$, and $\Gamma_2$ denote a two component path such that the components connect two basepoints in $Y$ and  $-Y$ to the basepoints in $W_2$ as shown in the right of Figure \ref{Fig-12}, and $[\overline{\omega}_2]\in H^2(W_2;\mathbb{R})$ be such that whose restriction to each boundary component is  $ d\cdot\textnormal{PD}([\eta])$. Then 
	\[
	\underline{f}^-_{W_1,\Gamma_1;\Lambda_{\overline{\omega}_1}}\ \dot{\simeq}\ \underline{f}^-_{W_2,\Gamma_2;\Lambda_{\overline{\omega}_2}}.
	\]
\end{lemma}

\begin{proof}
	The proof is similar to the proof of Lemma A.1.35 of \cite{phdthesis}. We only need to modify it to the twisted case.
	First note that the two cobordisms $W_i$, $i=1,2$, induce maps, which are a sum of maps, from $\underline{CF}^-(\emptyset)$ to $\underline{CF}^-(Y\amalg-Y,\{w_1,w_0\};\Lambda_\omega)$ (see Subsection \ref{SumsCom}).  
	From
	\[
	\underline{CF}^-(Y\amalg-Y,\{w_1,w_0\};\Lambda_{\omega\oplus\omega})=\underline{CF}^-(Y,w_1;\Lambda_\omega)\otimes_{\Lambda}\underline{CF}^-(-Y,w_0;\Lambda_\omega)
	\]
	and
	\[
	\underline{HF}^-(Y,w_1;\Lambda_\omega)\simeq
	\bigoplus_{\mathfrak{s}\in Spin^c(Y)}\underline{HF}^-(Y,w_1,\mathfrak{s};\Lambda_\omega),
	\]
	and the fact that only one of the above summands is nonzero (by Theorem \ref{Thm-2}), the induced maps $\underline{F}^-_{W_i,\Gamma_1;\Lambda_{\overline{\omega}_i}}$, $i=1,2$, are well-defined by Lemma \ref{lem-Zemke} and Remark \ref{Rmrk-Completed-vs-usual-Novikov}. Let $\mathcal{R}=\Lambda\llbracket U_1,U_2\rrbracket$. We have  $\underline{CF}^-(\emptyset)=\mathcal{R}$ (see Remark A.1.6 of \cite{phdthesis}). $\underline{CF}^-(-Y,w_1;\Lambda_\omega)$  is chain isomorphic to $\text{Hom}_{\mathcal{R}}(\underline{CF}^-(Y,w_1;\Lambda_\omega),\mathcal{R})$ (\cite[Subsection 12.2]{Z21}). Therefore, there is the following canonical isomorphism
	\[
	\underline{CF}^-(-Y,w_1;\Lambda_\omega)\cong\underline{CF}^-(Y,w_1;\Lambda_\omega)^\vee:=\text{Hom}_{\mathcal{R}}(\underline{CF}^-(Y,w_1;\Lambda_\omega),\mathcal{R}).
	\]
	By Theorem \ref{Thm-2}, $\underline{CF}^-(Y,w_1;\Lambda_\omega)$ is chain homotopy equivalent to the chain complex
	\begin{equation}\label{eq-3}
		0\rightarrow\Lambda\llbracket U_1\rrbracket\langle x\rangle\xrightarrow{U_1}\Lambda\llbracket U_1\rrbracket\langle y\rangle\rightarrow0.
	\end{equation}
	By Lemma \ref{Lem-2}, $\underline{F}^-_{W_2,\Gamma_2;\Lambda_{\overline{\omega}_2}}$ is non-zero. Therefore, up to a unit in $\Lambda$ and chain homotopy, the map $\underline{f}^-_{W_2,\Gamma_2;\Lambda_{\overline{\omega}_2}}$ (on the chain level) sends the generator of $\underline{CF}^-(\emptyset)$ to $ y\otimes x^\vee$ where $x^\vee$ is dual of $x$. To compute $\underline{f}^-_{W_1,\Gamma_1;\Lambda_{\overline{\omega}_1}}$, we decompose $(W_1,\Gamma_1)$ into three pieces $(W_1^i,\Gamma_1^i)$ (see Figure \ref{Fig-12} in the middle). Let $\overline{\omega}_{1i}=\overline{\omega}_1|_{W_1^i}$.
	\begin{enumerate}
		\item $(W_1^1,\Gamma_1^1)$ is a cobordism from $\emptyset$ to $(Y_1\amalg-Y_1,\mathbf{w}_1)$ where $Y_1$ is a copy of $Y$, $\mathbf{w}_1$ intersects each of $Y$ and $-Y$ in a single  basepoint $\{w'_{11}\}$ and $\{w_{11}\}$, respectively,  and $\Gamma_1^1$ is a path that connects the two  basepoints $\{w'_{11}\}$ and $\{w_{11}\}$. This is a twisted cotrace map which, up to a unit in $\Lambda$, sends the generator of $\underline{CF}^-(\emptyset)=\mathcal{R}$ to $x\otimes x^\vee+y\otimes y^\vee$ where $x^\vee$ and $y^\vee$ are duals of $x$ and $y$  (see \cite[Theorem 1.7 and Subsection 12.2]{Z21}).
		\item $(W_1^2,\Gamma_1^2)$ is a two component cobordism from $(Y_1\amalg-Y_1,\mathbf{w}_1)$ to $(Y_2\amalg-Y_2,\mathbf{w}_2)$ where $Y_2$ is a copy of $Y$ and $\mathbf{w}_2$ intersects $Y_2$ in two  basepoints $\{w''_0,w'_{12}\}$ and intersect $-Y_2$ in a single  basepoint $\{w_{12}\}$. Two components of $\Gamma_1^2$ connect $w'_{11}$ and $w_{11}$ in $Y_1$ and $-Y_1$, respectively, to $w'_{12}$ and $w_{12}$ in $Y_2$ and $-Y_2$, respectively. Also, one component of $\Gamma_1^2$ connects $w'_0$ to $w''_0\in Y_2$. One component of this cobordism corresponds to identity and the other one is a free stabilization $S^+_{w''_0}$. Let $\mathcal{H}$ and $\mathcal{H}_1$ be Heegaard diagrams for $Y_1$ and $Y_2$, respectively, where $\mathcal{H}_1$ (see Figure \ref{Connected sum change base} on the left) is a connected sum of $\mathcal{H}$ with $\mathcal{H}_0$ (see Figure \ref{HD-Sphere-2})  such that a  basepoint in $\mathcal{H}$ is identified with $w'_{12}$. We have
		\[
		\underline{f}^-_{W_1^2,\Gamma_1^2;\Lambda_{\overline{\omega}_{12}}}:\underline{CF}^-(\mathcal{H};\Lambda_\omega)\otimes_\Lambda\underline{CF}^-(-\mathcal{H};\Lambda_\omega)\rightarrow \underline{CF}^-(\mathcal{H}_1;\Lambda_\omega)\otimes_\Lambda\underline{CF}^-(-\mathcal{H};\Lambda_\omega).
		\] 
		By Theorem \ref{Thm-1}, we can compute $\underline{CF}^-(\mathcal{H}_1;\Lambda_\omega)$ in terms of $\underline{CF}^-(\mathcal{H};\Lambda_\omega)$, where the later is chain homotopy equivalent to the chain complex  in \eqref{eq-3}. Therefore, the above map  sends a generator $u\otimes v$ to $ u\otimes v\otimes\theta^+$, up to a unit  (see the proof of Lemma A.1.35 of \cite{phdthesis} for more details). Here $-\mathcal{H}$ is a Heegaard diagram for $Y_1$ obtained from $\mathcal{H}$ by reversing the orientation of Heegaard surface. 
		\item $(W_1^3,\Gamma_1^3)$ is a two component cobordism from $(Y_2\amalg-Y_2,\mathbf{w}_2)$ to $(Y_3\amalg-Y_3,\mathbf{w}_3)$ where $Y_3$ is a copy of $Y$ and $\mathbf{w}_3$ intersects each of $Y_3$ and $-Y_3$ in $w_0$ and $w_1$, respectively. One component corresponds to the identity and the other one is a free stabilization map $S^-_{w'_{12}}$. Let $\mathcal{H}_2$ be a Heegaard diagram for $Y_2$ which is a connected sum of $\mathcal{H}$ with $\mathcal{H}_0$  where a  basepoint in $\mathcal{H}$ is identified with $w''_0$ (see Figure \ref{Connected sum change base} on the right). We have
		\[
		\underline{f}^-_{W_1^3,\Gamma_1^3;\Lambda_{\overline{\omega}_{13}}}:\underline{CF}^-(\mathcal{H}_2;\Lambda_\omega)\otimes\underline{CF}^-(-\mathcal{H};\Lambda_\omega)\rightarrow \underline{CF}^-(\mathcal{H};\Lambda_\omega)\otimes\underline{CF}^-(-\mathcal{H};\Lambda_\omega).
		\]
		By Theorem \ref{Thm-1} and the chain complex in \eqref{eq-3}, $\underline{f}^-_{W_1^3,\Gamma_1^3;\Lambda_{\overline{\omega}_{13}}}$, up to a unit, sends a generator $u\otimes v\otimes\theta^-$ to $u\otimes v$ and $u\otimes v\otimes\theta^+$ to $0$ (see the proof of Lemma A.1.35 of \cite{phdthesis} for more details).
	\end{enumerate} 
Therefore, 
\[
\begin{split}
\underline{f}^-_{W_1,\Gamma_1;\Lambda_{\overline{\omega}_1}}=\ \underline{f}^-_{W_1^3,\Gamma_1^3;\Lambda_{\overline{\omega}_{13}}}\circ\underline{\Psi}_{(\mathcal{H}_1,J_1)\rightarrow(\mathcal{H}_2,J_2)}\circ\underline{f}^-_{W_1^2,\Gamma_1^2;\Lambda_{\overline{\omega}_{12}}}\circ\underline{f}^-_{W_1^1,\Gamma_1^1;\Lambda_{\overline{\omega}_{13}}},
\end{split}
\]
where $J_i$, $i=1,2$, are the almost complex structures in Proposition \ref{Prop-1}. Note that 
\[
\begin{split}
A&=\underline{f}^-_{W_1^2,\Gamma_1^2;\Lambda_{\overline{\omega}_{12}}}\circ\underline{f}^-_{W_1^1,\Gamma_1^1;\Lambda_{\overline{\omega}_{13}}}(1)=\underline{f}^-_{W_1^2,\Gamma_1^2;\Lambda_{\overline{\omega}_{12}}}(x\otimes x^\vee+y\otimes y^\vee)\\
&=x\otimes x^\vee\oplus\theta^++y\otimes y^\vee\oplus\theta^+=(x\otimes x^\vee\oplus+y\otimes y^\vee,0),
\end{split}
\]
where $1$ denotes the generator of $\mathcal{R}$ and the last equality shows $A$ in the matrix notation with $\theta^\pm$ components. Since $\underline{f}^-_{W_1^3,\Gamma_1^3;\Lambda_{\overline{\omega}_{13}}}$ sends a generator $u\otimes v\otimes\theta^+$ to zero, by Proposition \ref{Prop-1}, we only need to find the image of the component $x\otimes x^\vee\oplus+y\otimes y^\vee$ of $A$ under the  entry 
\begin{equation}\label{eq-0}
(\underline{\Phi}_{\boldsymbol{\alpha}\rightarrow\boldsymbol{\alpha}'}^{\boldsymbol{\beta}'})_{U_{w'_{12}}}^{U_z\rightarrow U_{w''_0}}\circ(\sum U_{w''_0}^iU_{w'_{12}}^j(\underline{\partial}_{i+j+1})_{U_{w''_0,w'_{12}}})\circ (\underline{\Phi}_{\boldsymbol{\alpha}}^{\boldsymbol{\beta}\rightarrow\boldsymbol{\beta}'})_{U_{w''_0}}^{U_z\rightarrow U_{w'_{12}}}
\end{equation}
of $\underline{\Psi}_{(\mathcal{H}_1,J_1)\rightarrow(\mathcal{H}_2,J_2),\mathfrak{s}}$. Note that $x\otimes x^\vee\oplus+y\otimes y^\vee$ has no $U$-power, therefore,  $(\underline{\Phi}_{\boldsymbol{\alpha}\rightarrow\boldsymbol{\alpha}'}^{\boldsymbol{\beta}'})_{U_{w'_{12}}}^{U_z\rightarrow U_{w''_0}}$ and $(\underline{\Phi}_{\boldsymbol{\alpha}}^{\boldsymbol{\beta}\rightarrow\boldsymbol{\beta}'})_{U_{w''_0}}^{U_z\rightarrow U_{w'_{12}}}$ can be regarded as the identity, up to a factor $t^z$, $z\in\mathbb{R}$. By (\ref{eq-3}), $\underline{CF}^-(\mathcal{H}_2,\{z,w'_{12}\};\Lambda_\omega)\otimes\underline{CF}^-(-\mathcal{H},\{z\};\Lambda_\omega)$ with the differential $\underline{\partial}=\sum_{i\geq1}^{\infty}\underline{\partial}_iU_z^i$ is isomorphic with
\begin{center}
	$\begin{CD}
		\mathcal{R}\langle x,y^\vee\rangle @>U_z>> \mathcal{R}\langle x,x^\vee\rangle\\
		@VU_{w'_{12}}VV @VVU_{w'_{12}}V\\
		\mathcal{R}\langle y,y^\vee\rangle @>U_z>> \mathcal{R}\langle y,x^\vee\rangle.
	\end{CD}$
\end{center}
Therefore, the map in \ref{eq-0}, sends $x\otimes x^\vee\oplus+y\otimes y^\vee$ to $y\otimes x^\vee$, up to a unit, and we have
\[
\begin{split}
\underline{f}^-_{W_1^3,\Gamma_1^3;\Lambda_{\overline{\omega}_{13}}}\circ
\Psi_{(\mathcal{H}_1,J_1)\rightarrow(\mathcal{H}_2,J_2)}(A)=\underline{f}^-_{W_1^3,\Gamma_1^3;\Lambda_{\overline{\omega}_{13}}}(*\otimes\theta^++y\otimes x^\vee\otimes\theta^-)=y\otimes x^\vee.
\end{split}
\]
This observation completes the proof. 
\end{proof}

\begin{proof}[Proof of Theorem \ref{MainThm}]
	The proof is a slight modification of the proof of Theorem A.1.30 of \cite{phdthesis}. One can construct a cobordism $W$ from $Y_1$ to $Y_2$. In fact, let $Y'$ denote a three manifold obtained by cutting $Y_1$ along $F=\Sigma_1\cup\Sigma_2$ and $P$ be a saddle with boundary and corners (see Figure \ref{Fig-13} on the left). If we glue $Y'\times I$ to $P\times \Sigma$, where $\Sigma\cong\Sigma_1\cong\Sigma_2$, as in Figure \ref{Fig-13} on the right, we obtain the cobordism $W$  (see \cite{phdthesis} for more details).  Let $Y_1$ be disconnected (the proof for the connected case is similar ). Let $W'=-W$ be the upside-down cobordism from $Y_2$ to $Y_1$ and $W_A=W\cup_{Y_2}W'$ and $W_B=W'\cup_{Y_1}W$. We can obtain $W_A$ and $W_B$ from a product cobordism as follows. Suppose that $W''_A$ (resp. $W''_B$) is the product cobordism $Y_1\times [-1,1]$ (resp. $Y_2\times[-1,1]$). A neighborhood of $F\times\{0\}$ in $W''_A$ (resp. $W''_B$) can be identified with $F\times [-\epsilon,\epsilon]\times [-\epsilon,\epsilon]$ for some small $\epsilon>0$. If we remove the interior of  $F\times [-\epsilon,\epsilon]\times [-\epsilon,\epsilon]$ and identify its two boundary components with the boundary components of the manifold $\Sigma\times S^1\times[-\epsilon',\epsilon']$, where $\epsilon'>0$, we obtain $W_A$ (see Figure \ref{Figure-0}) (resp. $W_B$). Let $Y_A:=\Sigma\times S^1\subset W_A$ and  $N(Y_A)$ denote a neighborhood of $Y_A$. Similarly, $Y_B:=\Sigma\times S^1\subset W_B$ and  $N(Y_B)$ denote a neighborhood of $Y_B$.
	
	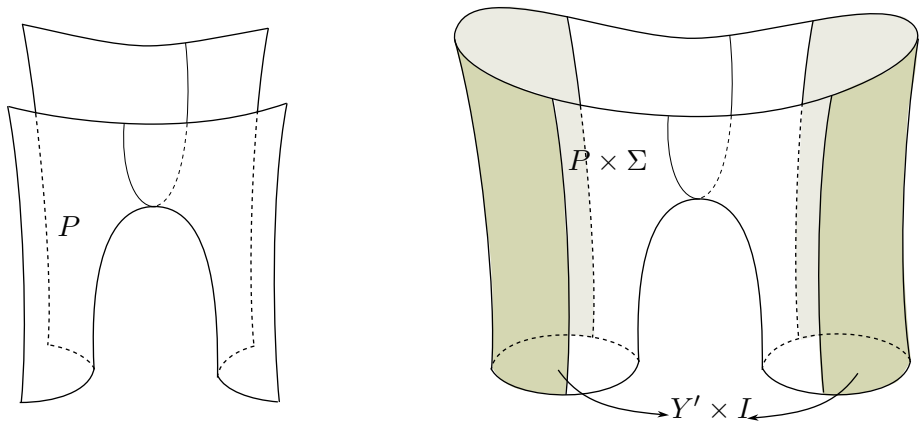
\begin{figure}[h!]
		\def\svgwidth{12cm}
		\begin{center}
\begingroup%
  \makeatletter%
  \providecommand\color[2][]{%
    \errmessage{(Inkscape) Color is used for the text in Inkscape, but the package 'color.sty' is not loaded}%
    \renewcommand\color[2][]{}%
  }%
  \providecommand\transparent[1]{%
    \errmessage{(Inkscape) Transparency is used (non-zero) for the text in Inkscape, but the package 'transparent.sty' is not loaded}%
    \renewcommand\transparent[1]{}%
  }%
  \providecommand\rotatebox[2]{#2}%
  \newcommand*\fsize{\dimexpr\f@size pt\relax}%
  \newcommand*\lineheight[1]{\fontsize{\fsize}{#1\fsize}\selectfont}%
  \ifx\svgwidth\undefined%
    \setlength{\unitlength}{389.19443686bp}%
    \ifx\svgscale\undefined%
      \relax%
    \else%
      \setlength{\unitlength}{\unitlength * \real{\svgscale}}%
    \fi%
  \else%
    \setlength{\unitlength}{\svgwidth}%
  \fi%
  \global\let\svgwidth\undefined%
  \global\let\svgscale\undefined%
  \makeatother%
  \begin{picture}(1,0.45788306)%
    \lineheight{1}%
    \setlength\tabcolsep{0pt}%
    \put(0,0){\includegraphics[width=\unitlength,page=1]{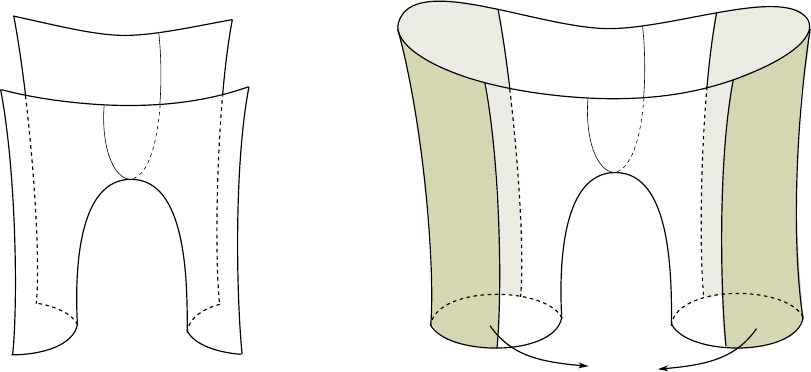}}%
    \put(0.05533772,0.20410373){\color[rgb]{0,0,0}\makebox(0,0)[lt]{\lineheight{1.25}\smash{\begin{tabular}[t]{l}$P$\end{tabular}}}}%
    \put(0.61504183,0.27714166){\color[rgb]{0.01176471,0.01176471,0.01176471}\makebox(0,0)[lt]{\lineheight{1.25}\smash{\begin{tabular}[t]{l}$P\times\Sigma$\end{tabular}}}}%
    \put(0.72748046,0.00334627){\color[rgb]{0.01176471,0.01176471,0.01176471}\makebox(0,0)[lt]{\lineheight{1.25}\smash{\begin{tabular}[t]{l}$Y'\times I$\end{tabular}}}}%
  \end{picture}%
\endgroup%

			\caption{Construction of the cobordism $W$ from $Y_1$ to $Y_2$. Left: A saddle $P$. Right: Gluing $Y'\times I$ to $P\times\Sigma$.}
			\label{Fig-13}
		\end{center}
	\end{figure}

\begin{figure}[h!]
	\def\svgwidth{14cm}
	\begin{center}
		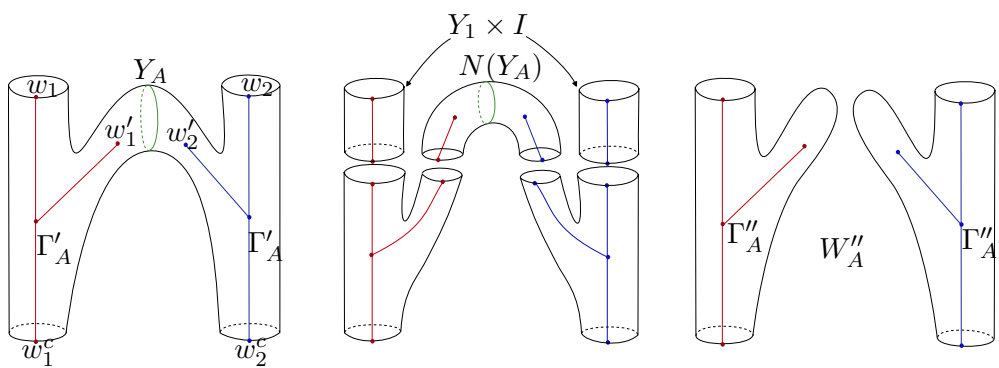
		\caption{ Figures A.7 and A.8 of \cite{phdthesis}}
		\label{Figure-0}
	\end{center}
\end{figure}

	Let $w_i$, $i=1,2$, denote two  basepoints, one in each connected component of $Y_1$. Note that $W_A$ is a cobordism from $Y_1$ to $-Y_1$. Let $w_i^c$, $i=1,2$, denote the corresponding  basepoints in $-Y_1$ and $\Gamma_A$ be a two component path which connects $w_i$ to $w_i^c$ in the boundaries of $W_A$. Let $z_i$, $i=1,2$, denote two  basepoints in $Y_2$.  Note that $W_B$ is a cobordism from $Y_2$ to $-Y_2$. Let $z_i^c$, $i=1,2$, denote the corresponding  basepoints in the copy $-Y_2$ and $\Gamma_B$ be a two component path which connects $z_i$ to $z_i^c$ in the two boundaries of $W_B$. Let $w'_i$, $i=1,2$, be two points in $W_A$ and $\Gamma'_A$ be obtained from $\Gamma_A$ by connecting one of the $w'_i$ by an arc to each component of $\Gamma_A$. Each component of $\Gamma'_A$ has a vertex of valence three. Fix an ordering for these vertices (see Figure \ref{Figure-0} on the left). From the Mayer-Vietoris sequence 
	\[
	\cdots\rightarrow H^2(W;\mathbb{R})\rightarrow H^2(Y'\times I)\oplus H^2(\Sigma\times P)\rightarrow\cdots
	\]
	and the assumptions on $\omega_i\in H^2(Y_i;\mathbb{R})$, there are $\overline{\omega}\in H^2(W;\mathbb{R})$, $\omega_A\in H^2(W_A;\mathbb{R})$, and $\omega_B\in H^2(W_B;\mathbb{R})$ such that $\overline{\omega}|_{Y_i}=\omega_i$, $\omega_A|_{Y_1}=\omega_1$, $\omega_B|_{Y_2}=\omega_2$, $\omega_A|_{Y_A}\neq0$, and $\omega_B|_{Y_B}\neq0$.

	One can decompose $(W_A,\Gamma'_A)$ into two cobordisms (see Figure \ref{Figure-0} in the middle). One cobordism consists of $N(Y_A)$ (as a cobordism from $\emptyset$ to $\Sigma\times S^1\amalg(-\Sigma\times S^1)$) and two copies of $Y_1\times I$. If we replace $N(Y_A)$ with a cobordism corresponding to $\Sigma\times D^2\amalg(-\Sigma\times D^2)$, we will obtain a cobordism $(W''_A,\Gamma''_A)$ (see Figure \ref{Figure-0} on the right). $\Gamma''_A$ is the graph induced by $\Gamma'_A$. Therefore, there are two edges in $\Gamma''_A$ with endpoints in the interior of $W''_A$ (one end point in each component is of valence one).  By \cite[Lemma 7.15]{Z15},  one can replace  $\Gamma''_A$ with a two component path which connects $w_i$ to $w^c_i$ in the boundaries
	of $W''_A$. We also denote this graph by $\Gamma''_A$. Therefore, $(W''_A,\Gamma''_A)$ is diffeomorphic to the product cobordism. If $\mathcal{H}_1$ denotes a Heegaard diagram for $Y_1$, the induced map
		\[
	\underline{F}^-_{W''_A,\Gamma''_A;\Lambda_{\omega'_A}}=\bigoplus_{i}\underline{F}^-_{W''_A,\Gamma''_A,\mathfrak{t}_{i};\Lambda_{\omega'_A}}
	\]
	where 
	\[
	\underline{F}^-_{W''_A,\Gamma''_A,\mathfrak{t}_{i};\Lambda_{\omega'_A}}:\underline{HF}^-(\mathcal{H}_1,\{w_1,w_2\},\mathfrak{s}_i;\Lambda_{\omega_1})\rightarrow\underline{HF}^-(\mathcal{H}_1,\{w_1,w_2\},\mathfrak{s}_i;\Lambda_{\omega_1}),
	\]
	is a diagonal matrix with units in $\Lambda$ on the diagonal entries. Here, $[\omega'_A]\in H^2(W''_A;\mathbb{R})$ is such that $\omega'_A|_{Y_1}=\omega_1$, and $\mathfrak{t}_{i}$ denotes the Spin$^\text{c}$ structure in $\text{Spin}^\text{c}(W''_A)$ whose restriction to the boundary components of $W''_A$, which are two copies of $Y_1$, is  $\mathfrak{s}_i\in\text{Spin}^\text{c}(Y_1)$.
	
	By Lemma \ref{Main-Lem}, 
	\[
	\underline{F}^-_{W_A,\Gamma'_A;\Lambda_{\omega_A}}\ \dot{\simeq}\ \underline{F}^-_{W''_A,\Gamma''_A;\Lambda_{\omega'_A}}.
	\]
	By \cite[Lemma 7.15]{Z15}, $(W_A,\Gamma_A)$ and $(W_A,\Gamma'_A)$ induce the same cobordism map up to a unit. Therefore, $\underline{F}^-_{W_A,\Gamma_A;\Lambda_{\omega_A}}$ is an invertible diagonal matrix $[\lambda_i]$, where $\lambda_i\in\Lambda$. 
	
	If $\mathcal{H}_2$ denotes a Heegaard diagram for $Y_2$, a similar argument proves that the induced map
	\[
	\underline{F}^-_{W_B,\Gamma_B;\Lambda_{\omega_B}}:\bigoplus_{i}\underline{HF}^-(\mathcal{H}_2,\{z_1,z_2\},\mathfrak{s}'_i;\Lambda_{\omega_2})\rightarrow\bigoplus_{i}\underline{HF}^-(\mathcal{H}_2,\{z_1,z_2\},\mathfrak{s}'_i;\Lambda_{\omega_2}),
	\]
	is a diagonal matrix $[\gamma_i]$ with units $\gamma_i\in\Lambda$ on the diagonal entries. Here, $\mathfrak{s}'_i\in\text{Spin}^\text{c}(Y_2)$.This proves that 	
		\[
		\underline{HF}(Y_1;\Lambda_{\omega_1})\cong  \underline{HF}(Y_2;\Lambda_{\omega_2}).
		\]
	This isomorphism is induced by a cobordism. In the following, we show that the map induced by the cobordism is well-defined. Let
	\[
	\mathfrak{T}_{ij}=\{\mathfrak{t}\in\text{Spin}^\text{c}(W)\mid\mathfrak{t}|_{Y_1}=\mathfrak{s}_i,\ \mathfrak{t}|_{Y_2}=\mathfrak{s}'_j\}
	\ \ \textnormal{and}\ \ 
	\mathfrak{T}'_{ij}=\{\mathfrak{t}\in\text{Spin}^\text{c}(W')\mid\mathfrak{t}|_{Y_2}=\mathfrak{s}'_i,\ \mathfrak{t}|_{Y_1}=\mathfrak{s}_j\}.
	\] 
	Let $F=[F_{ij}]$ and $G=[G_{ij}]$ where
	\[
	F_{ji}=\sum_{\mathfrak{t}\in\mathfrak{T}_{ij}}\underline{F}_{W,\Gamma,\mathfrak{t};\Lambda_{\overline{\omega}}}:\underline{HF}^-(\mathcal{H}_1,\{w_1,w_2\},\mathfrak{s}_i;\Lambda_{\omega_1})\rightarrow\underline{HF}^-(\mathcal{H}_2,\{z_1,z_2\},\mathfrak{s}'_j;\Lambda_{\omega_2}),
	\] 
	\[
	G_{ji}=\sum_{\mathfrak{t}\in\mathfrak{T}'_{ij}}\underline{F}_{W',\Gamma,\mathfrak{t};\Lambda_{\overline{\omega}}}:\underline{HF}^-(\mathcal{H}_2,\{z_1,z_2\},\mathfrak{s}'_i;\Lambda_{\omega_2})\rightarrow\underline{HF}^-(\mathcal{H}_1,\{w_1,w_2\},\mathfrak{s}_j;\Lambda_{\omega_1}),
	\] 
	and $\Gamma$ consists of a pair of paths that connects $w_i$ to $z_i$, $i=1,2$. We have
	\[
	\underline{F}^-_{W',\Gamma;\Lambda_{\overline{\omega}}}=\bigoplus_j(\sum_iG_{ij}), 
	\]
	and we show that $\pi_{s_j}\circ\underline{F}^-_{W',\Gamma;\Lambda_{\overline{\omega}}}$ is well-defined up to a unit. Indeed, by Lemma \ref{lem-Zemke}, $F_{ij}$ and $G_{ij}$ are well-defined. Let $\mathcal{H}'_i$, $i=1,2$, denote two Heegaard diagrams for $Y_i$ with transition maps
	\[
	\underline{\Psi}_{i}=\underline{\Psi}_{(\mathcal{H}_1,J_1)\longrightarrow(\mathcal{H}'_1,J'_1),\mathfrak{s}_i},\ \ 
	\underline{\Phi}_i=\underline{\Phi}_{(\mathcal{H}_2,J_2)\longrightarrow(\mathcal{H}'_2,J'_2),\mathfrak{s}'_i}.
	\]
	Let $G'=[G'_{ij}]$ where
	\[
	G'_{ji}:\underline{HF}^-(\mathcal{H}'_2,\{z_1,z_2\},\mathfrak{s}_i;\Lambda_{\omega_2})\rightarrow\underline{HF}^-(\mathcal{H}'_1,\{w_1,w_2\},\mathfrak{s}_j;\Lambda_{\omega_1}),
	\] 
	is the map induced by the  cobordism $W'$. From the above argument, 
	\begin{equation}\label{eq-1}
	F\circ G=[\gamma_i], \ \ \ \ \ G\circ F=[\lambda_i].
	\end{equation}
	Therefore, $G\circ[\gamma_i]=[\lambda_i]\circ G$. If $G_{ij}\neq0$, $\gamma_j=\lambda_i$. 
	Similarly, there are invertible diagonal matrices $[\gamma'_i]$ and $[\lambda'_i]$ such that
	\begin{equation}\label{eq-2}
	F\circ\underline{\Psi}^{-1}\circ G'\circ\underline{\Phi}=[\gamma'_i],\ \ \ \ \ \underline{\Psi}^{-1}\circ G'\circ\underline{\Phi}\circ F=[\lambda'_i],	
	\end{equation}  
    where $\underline{\Psi}$ and $\underline{\Phi}$ are  diagonal matrices with $\underline{\Psi}_{ii}=\underline{\Psi}_i$, $\underline{\Phi}_{ii}=\underline{\Phi_i}$. 
    This shows that $\gamma'_j=\lambda'_i$ if $(\underline{\Psi}\circ G'\circ\underline{\Phi}^{-1})_{ij}\neq0$. From Equations \eqref{eq-1} and \eqref{eq-2}, we have 
    \[
    [\lambda_i]\underline{\Psi}^{-1}\circ G'\circ\underline{\Phi}=G[\gamma'_i],
    \]
    which shows that $\lambda_i(\underline{\Psi}^{-1}\circ G'\circ\underline{\Phi})_{ij}=G_{ij}\gamma'_j$. Therefore, when $(\underline{\Psi}^{-1}\circ G'\circ\underline{\Phi}^{-1})_{ij}\neq0$,  $\gamma_j(\underline{\Psi}\circ G'\circ\underline{\Phi})_{ij}=G_{ij}\gamma'_j$. From here,
    \[
    \sum_iG_{ij}=\frac{\gamma_j}{\gamma'_j}\sum_i(\underline{\Psi}^{-1}\circ G'\circ\underline{\Phi})_{ij}\ \dot{\cong}\ \sum_i(\underline{\Psi}^{-1}\circ G'\circ\underline{\Phi})_{ij}.
    \] 
    This completes the proof.
\end{proof}
\begin{remark}\label{Independence-diff}
	Note that Theorem \ref{MainThm} implies that we can use two orientation-preserving diffeomorphisms $h,h':\Sigma_1\rightarrow\Sigma_2$ to identify $\Sigma_1$ with $-\Sigma_2$ using $h$ and identify $-\Sigma_1$ with $\Sigma_2$ using $h'$. Indeed, assume that  $Y_2$ is obtained from $Y_1$ by excision along the surfaces $\Sigma_1$ and $\Sigma_2$ (using the orientation-preserving diffeomorphism $h$ ). Let $Y=Y_2\cup T_{\phi}$, where $\phi=h^{-1}\circ h'$ and $T_\phi$ is the mapping torus of $\phi$.  
	By Theorem \ref{MainThm} and Theorem \ref{Thm-2},
	\[
	\underline{HF}(Y_3;\Lambda_{\omega_3})\cong \underline{HF}(Y_2;\Lambda_{\omega_2}),
	\]
	for a generic choice of $[\omega_3]\in H^2(Y_3;\mathbb{R})$ (as stated in Theorem \ref{MainThm}). Note that if we cut $Y_1$ along $\Sigma_1\cup\Sigma_2$ to obtain the three manifold $Y'$, and identify $\Sigma_1$ with $-\Sigma_2$ using $h$, and $-\Sigma_1$ with $\Sigma_2$ using $h'$, we obtain the 3-manifold $Y_3$.
\end{remark}

\section{Applications}\label{Section 3}
Let  $K$ be a genus one knot, and $Y_p=S^3_p(K)$ be the 3-manifold obtained by performing p-surgery on $K$. Therefore, $Y_0$ contains a non-separating torus. 
In this section, first  we compute $\underline{HF}^+(Y_0;\Lambda)$ when $K$ is a \emph{twist knot} which is, by definition, an  $n$-twisted Whitehead double of the unknot (see Figure \ref{Whitehead double}). Then  we use Theorem \ref{MainThm} to compute Heegaard Floer homology groups for the manifold obtained by cutting $Y_0$ along the non-separating torus and regluing it using a Dehn twist. We use this computation to prove Corollary \ref{Topology-result} and Corollary \ref{two-bridge}.

We recall some facts from knot Floer homology theory that are used in this section. There are several variants of bi-graded knot Floer homology groups, introduced independently by  Ozsv\'ath-Szab\'o and Rasmussen (see \cite{OZSVATH200458,Rasmussen2003FloerHA}).  Indeed, given a knot $K$, there exists a doubly pointed Heegaard diagram
$\mathcal{H}=(\Sigma,\boldsymbol{\alpha},\boldsymbol{\beta},z,w)$ that represents $K$.
The $\mathbb{Z}$-filtered chain complex, denoted $CFK^\infty(K)$, which is well-defined up to filtered chain homotopy
equivalence, is generated by elements in $\mathbb{T}_\alpha\cap \mathbb{T}_\beta$ as a $\mathbb{Z}[U,U^{-1}]$-module. Associated to each generator there are two gradings, called the Maslov grading $M$ and Alexander grading $A$. One can think of $CFK^\infty(K)$ as freely generated over $\mathbb{Z}$ by $[\mathbf{x},i,j]$ where $\mathbf{x}\in \mathbb{T}_\alpha\cap \mathbb{T}_\beta$, $i,j\in\mathbb{Z}$, and $A(\mathbf{x})=j-i$. $[\mathbf{x},i,j]$ corresponds  to the generator $U^{-i}x$. $CFK^\infty(K)$ is called the full knot Floer complex and each generator $[\mathbf{x},i,j]$ is presented by a dot located at the $(i,j)$ coordinates on the plane.  The sub-complex generated by $[\mathbf{x},0,j']$ induces a filtration on $\widehat{CF}(S^3)$ 
\[
\dots\subset\mathcal{F}(K,j-1)\subset\mathcal{F}(K,j)\subset \mathcal{F}(K,j+1)\dots\subset\widehat{CF}(S^3),
\]
where $\mathcal{F}(K,j)$ is freely generated over $\mathbb{Z}$ by $[\mathbf{x},0,j']$ such that  $j'\leq j$. 
The homology of the associated graded complex $\oplus_{j}\mathcal{F}(K,j)/\mathcal{F}(K,j-1)$ is denoted by
\[
\widehat{HFK}(K)=\bigoplus_{i,s}\widehat{HFK}_i(K,s),
\]
where $i$ and $s$ are the Masolv grading and the Alexander grading. $\widehat{HFK}(K)$
categorifies the Alexander polynomial (see \cite{OZSVATH200458}) in the sense that
\[
\Delta_K(t)=\sum_{i,s}(-1)^i\dim\widehat{HFK}_i(K,s)t^s. 
\]
For alternating knots, $\widehat{HFK}$ is determined by the Alexander polynomial. Indeed, if $\Delta_K(t)=a_0+\sum_{s>0}a_s(T^s+T^{-s})$ denotes the symmetrized Alexander polynomial of $K$, then for the alternating knot $K$, $\widehat{HFK}(K,s)$ is supported in dimension $s+\frac{\sigma}{2}$ where $\sigma$ denotes the signature of $K$ and
\[
\widehat{HFK}(K,s)\cong\mathbb{Z}^{|a_s|}.
\]

One can define a third grading $\delta=A-M$, called the $\delta$-grading.
A knot $K$ is called \emph{Floer homologically thin}, or \emph{thin} for short, if $\widehat{HFK}(K)$ is  supported in a single $\delta$-grading. If the homology is supported on the diagonal $\delta=-\sigma/2$,
where $\sigma$ denotes the knot signature, then we say the knot is $\sigma$-thin.

Fixing a field $k$, the Heegaard Floer homology group $\widehat{HF}(S^3)$ of the three-sphere, with coefficients in $k$, is isomorphic to $k$, supported in homological
grading zero. One can define the following:
\[
\tau(K)=\min\{j\in\mathbb{Z}|i_*:H_*(\mathcal{F}(K,j))\rightarrow H_*(\widehat{CF}(S^3))\ \text{is non-trivial}\}.
\] 
Another sub-complex of $CFK^\infty(K)$ which is generated by $[\mathbf{x},i,j]$, $i\leq0$, is denoted by $CFK^-(K)$ and induces a filtration on $CF^-(S^3)$. The filtered
chain homotopy type of $CFK^-(K)$ is a knot invariant. If $K$ is thin, $CFK^-(K)$ is completely determined by $\tau(K)$ and $\Delta_K(t)$, the Alexander polynomial associated to $K$ (see \cite[Theorem 4]{Petkova2009CablesOT}). Note that for a $\sigma$-thin knot $\tau(K)=-\frac{\sigma}{2}$. 

The  negative $n$-twisted Whitehead double of the unknot, denoted $D_-(U,n)$ where $U$ is the unknot, is shown in Figure \ref{Whitehead double} on the left (the $"-"$ indicates the parity of the clasp). Let $D_+(U,n)$ denote the positive $n$-twisted Whitehead double of the unknot, where the clasp is positive. Note that $D_-(U,n)$ is an alternating knot with the Alexander polynomial $\Delta_{D_-(U,n)}(t)=n(t+t^{-1})+(-2n+1)$. Since alternating knots are $\sigma$-thin, the above argument determines $CFK^-(D_-(U,n))$ and therefore its full knot complex (see Figure \ref{Knot Complex}).

\begin{figure}[!h]
	\begin{center}
		\def\svgwidth{7cm}
		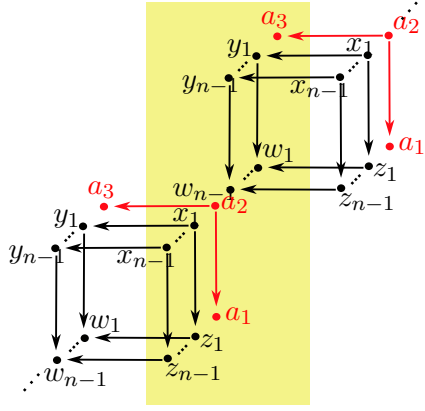
		\caption{This picture shows the knot complex for $D_-(U,n)$ where each generator $[x,i,j]$ is shown as $x$. To distinguish the generators, the $U^0$ column which contains the generators with $[x,0,j]$ is highlighted. In this column, $A(w_i)=A(x_i)=A(a_2)=0$, $A(a_3)=A(y_i)=1$, $A(a_1)=A(z_i)=-1$, $i=1,\cdots,n-1$, and $M(a_3)=0$. With this information, the generator $a_3$, for instance, which is not in the $U_0$ column represents $[a_3,-2,0]$. }
		\label{Knot Complex}
	\end{center}
\end{figure}

\begin{lemma}\label{Lem-Ex}
	Let $F$ be a genus one surface obtained by capping off the Seifert surface of $D_-(U, n)$ (resp. $D_+(U,-n)$), $n>0$, in $S^3_0(D_-(U,n))$ (resp. $S^3_0(D_+(U,-n))$) and $[\omega]\in H^2(S^3_0(D_\pm(U,\mp n));\mathbb{R})$ be a cohomology class such that $\omega([F])\neq0$. There is an isomorphism of $\Lambda$-modules
	\[
		\underline{HF}(S^3_0(D_\pm(U,\mp n));\Lambda_\omega)\cong\Lambda^n,
	\]
	supported in a single grading.
\end{lemma}
\begin{proof}
We use the idea in the proof of Lemma 8.6 of \cite{OZSVATHAbsolutelygraded} to compute $\underline{HF}^+(S^3_0(D_-(U,n)))$. We find it easier to compute $\widehat{HF}(Y_1)$ first where $Y_1=S^3_1(D_-(U,n))$.
The surgery formula (see \cite[Theorem 4.4]{OZSVATH200458}) relates $HF^\circ(S^3_p(K))$, $\circ\in\{\wedge,\pm,\infty\}$, to the homology of certain sub-complexes of $CFK^\infty(K)$. 
In fact,  $\widehat{HF}(S^3_p(K),[s])\cong H_*(C\{\max(i,j-s)=0\})$ and $HF^+(S^3_p(K),[s])\cong H_*(C\{\max(i,j-s)\geq0\})$ when $p\geq 2g(K)-1$ and $s\leq p/2$. Here $[s]\in\mathbb{Z}_p$ is the corresponding $\text{Spin}^c$ structure on $S^3_p(K)$. Therefore, $\widehat{HF}(Y_1,0)\cong\mathbb{Z}^{n-1}_{(-1)}\oplus\mathbb{Z}^n_{(-2)}$ generated by $w_i$, $y_i$, and $a_3$ where $M(w_i)=-1$ and $M(y_i)=M(a_3)=-2$, $i=1,\dots,n-1$ (see Figure \ref{Knot Complex}).

Let $t$ be a generator for $H^1(Y_0;\mathbb{Z})$, then we can think of $\mathbb{Z}[H^1(Y_0; \mathbb{Z})]$ as $\mathbb{Z}[t, t^{-1} ]$ or $L(t)$. For any $\mathbb{Z}[U]$-module $M$, let $M[t, t^{-1}]$ denote the induced module over $\mathbb{Z}[U, t,t^{-1}]$. There is a $\mathbb{Z}[U,t,t^{-1}]$-equivariant long exact sequence (see \cite[ Theorem 9.21]{OS2})
\begin{equation}\label{long-exact-1}
\dots\rightarrow\widehat{HF}(S^3)[t,t^{-1}]\xrightarrow{\widehat{f}_1} \underline{\widehat{HF}}(Y_0)\xrightarrow{\widehat{f}_2} \widehat{HF}(Y_1)[t,t^{-1}]\xrightarrow{\widehat{f}_3}\dots
\end{equation}
where the component of $\widehat{f}_1$ mapping into $\widehat{HF}(Y_0,\mathfrak{t}_0)$ (now thought of as absolutely $\mathbb{Q}$-graded)
has degree $-1/2$, the restriction of $\widehat{f}_2$ to $\widehat{HF}(Y_0, \mathfrak{t}_0)$ has degree $-1/2$, while $\widehat{f}_3$ is non-increasing in grading (see Subsection \ref{Gradings}). Therefore $\widehat{f}_3$ vanishes identically and $\underline{\widehat{HF}}(Y_0)\cong L(t)^n\oplus L(t)^n$, generated in dimensions $-\frac{1}{2}$ and $-\frac{3}{2}$. Indeed, by the adjunction inequality (see \cite[Theorem 7.1]{OS2}), $\underline{\widehat{HF}}(Y_0,\mathfrak{t})=0$ for $\mathfrak{t}\neq\mathfrak{t}_0$. Theorem 7.1 of \cite{OS2} is stated for $HF^+$, but the argument, as explained in \cite{OS2}, applies to the case of $\underline{\widehat{HF}}$. 

In general, if $\underline{\widehat{HF}}_k(Y)=0$, $k\geq m$,  the long exact sequence
\begin{equation}\label{LES}
	\cdots\xrightarrow{\delta}\underline{\widehat{HF}}_{i+1}\xrightarrow{i_*}\underline{HF}^+_{i+1}\xrightarrow{U}\underline{HF}^+_{i-1}\xrightarrow{\delta}\underline{\widehat{HF}}_{i}\cdots,
\end{equation}
implies that $U:\underline{HF}_{i+1}^+(Y)\rightarrow \underline{HF}_{i-1}^+(Y)$, $i\geq m$, is an isomorphism. In this sequence $\delta$ is the connecting map.
Therefore, $\underline{HF}_{i}^\infty(Y)\cong\underline{HF}^+_{i}(Y)$, $i\geq m-1$. Indeed, if $x\in \underline{HF}_i^+(Y)$, $i\geq m-1$, for each $n$, there is a $y_n\in\underline{HF}^+(Y)$ such that $x=U^ny_n$. Therefore, $\delta(x)=U^n\delta(y_n)=0$ since $\delta(y_n)=0$ for large $n$ where $\delta$ is the connecting map in the long exact sequence. This induces the following short exact sequence
\[
0\rightarrow \underline{HF}_i^-(Y)\rightarrow\underline{HF}_i^\infty(Y)\rightarrow\underline{HF}_i^+(Y)\rightarrow0,
\]
$i\geq m-1$. For large $i$, since $\underline{HF}_i^-(Y)=0$, we have $\underline{HF}_i^\infty(Y)\cong\underline{HF}_i^+(Y)$. From here and the isomorphism $U:\underline{HF}_{i+1}^+(Y)\rightarrow \underline{HF}_{i-1}^+(Y)$, $i\geq m$, we have $\underline{HF}_{i}^\infty(Y)\cong\underline{HF}^+_{i}(Y)$, $i\geq m-1$. The computation of $\underline{HF}^\infty(Y)$ (see Theorem 10.12 of \cite{OS2}) shows that there is a $\mathbb{Z}[U, U^{-1}]\otimes\mathbb{Z}[t,t^{-1}]$-module isomorphism $\underline{HF}^+_{i}(Y_0,\mathfrak{t}_0)\cong\mathbb{Z}[U,U^{-1}]$ where here the latter group is endowed with a trivial action by $\mathbb{Z}[t,t^{-1}]$. To determine the homological grading, note that  the cobordism  from $S^3$ to $Y_0$, obtained by adding a 2-handle, induces a map $f^\infty_1:\underline{HF}^\infty(S^3)[t,t^{-1}]\rightarrow \underline{HF}^\infty(Y_0,\mathfrak{t}_0)$ which shifts the grading by -1/2 (see Subsection \ref{Gradings}). Therefore, $\underline{HF}^+_{i}(Y_0)\cong\mathbb{Z}$, $i\equiv-1/2$ and $i\geq-1/2$. Using the long exact sequence that relates $\underline{\widehat{HF}}(Y_0)$ and $\underline{HF}^+(Y_0)$, we have
\[\begin{matrix}
	\underline{HF}^+_{-i}\cong\underline{HF}^+_{-\frac{5}{2}},&\text{if}\ i\equiv-5/2\mod 2\  \text{and}\ i\geq-5/2,\\
	\underline{HF}^+_{-i}\cong\underline{HF}^+_{-\frac{7}{2}},&\text{if}\ i\equiv-7/2\mod 2\  \text{and}\ i\geq-7/2.
\end{matrix}
\] 
Since $\underline{HF}^+_{-i}=0$, for large $i$, we have 
\[\begin{matrix}
	\underline{HF}^+_{-i}=0,&\text{if}\ i\equiv-5/2\mod 2\  \text{and}\ i\geq-5/2,\\
	\underline{HF}^+_{-i}=0,&\text{if}\ i\equiv-7/2\mod 2\  \text{and}\ i\geq-7/2.
\end{matrix}
\]
From here and the long exact sequence in \eqref{LES}, we have
\[
0\rightarrow\underline{\widehat{HF}}_{-3/2}(Y_0)\cong L(t)^n\rightarrow\underline{HF}^+_{-3/2}(Y_0)\rightarrow0.
\]
This completes the computation
\[
\underline{HF}^+(Y_0)_i\cong\begin{cases}
	\mathbb{Z},& \text{if}\  i\equiv-1/2\mod 2\  \text{and}\ i\geq-1/2,\\
	L(t)^n&\text{if}\ i=-3/2,\\
	0& \text{otherwise.}
\end{cases}
\]

To compute $\underline{HF}^+(Y_0;\Lambda_\omega)$, we follow the computation in \cite[Proof of Theorem 1.3]{A-P2010}. First note that the above computation is true with coefficients in $\mathbb{F}_2$
\[
\underline{HF}^+(Y_0;\mathbb{F}_2[t,t^{-1}])_i\cong\begin{cases}
	\mathbb{F}_2,& \text{if}\  i\equiv-1/2\mod2\  \text{and}\ i\geq-1/2,\\
	\mathbb{F}_2[t,t^{-1}]^n&\text{if}\ i=-3/2,\\
	0& \text{otherwise.}
\end{cases}
\]
Since
\[
\underline{CF}^+(Y_0;\Lambda_\omega)\cong\underline{CF}^+(Y_0;\mathbb{Z}[t,t^{-1}])\otimes_{\mathbb{Z}[t,t^{-1}]}\Lambda_\omega
\]
and $\mathbb{F}_2[t,t^{-1}]$ is a PID, by the universal coefficient theorem, there is the following exact sequence
\[
0\rightarrow\underline{HF}^+(Y_0;\mathbb{F}_2[t,t^{-1}])\otimes_{\mathbb{F}_2[t,t^{-1}]}\Lambda_\omega\rightarrow\underline{HF}^+(Y_0;\Lambda_\omega)\rightarrow\text{Tor}_1^{\mathbb{F}_2[t,t^{-1}]}(\underline{HF}^+(Y_0;\mathbb{F}_2[t,t^{-1}]),\Lambda)\rightarrow0.
\]
As shown in \cite{A-P2010}, $\text{Tor}_1^{\mathbb{F}_2[t,t^{-1}]}(\underline{HF}^+(Y_0;\mathbb{F}_2[t,t^{-1}]),\Lambda)=0$ and the result follows for $S^3_0(D_-(U,n))$. Note that $D_+(U,-n)$ is the mirror image of $D_-(U,n)$. Therefore, $S^3_0(D_+(U,-n))$ is homeomorphic to $-S^3_0(D_-(U,n))=-Y_0$. As discussed in \cite[Subsetion 12.2]{Z21}, the complex $\underline{CF}^-(-Y_0,\mathfrak{s};\Lambda_\omega)$ is chain homotopic to $\textnormal{Hom}_{\Lambda[U]}(\underline{CF}^-(Y_0,\mathfrak{s};\Lambda_\omega);\Lambda[U])$. 
By the Universal Coefficient Theorem, there is a split short exact sequence
\[
0\rightarrow\textnormal{Ext}^1_{\Lambda[U]}(\underline{HF}^-(Y_0;\Lambda_\omega),\Lambda[U])\rightarrow \underline{HF}^-(-Y_0;\Lambda_\omega)\rightarrow\textnormal{Hom}_{\Lambda[U]}(\underline{HF}^-(Y_0;\Lambda_\omega),\Lambda[U])\rightarrow0.
\]
We have
\[
\textnormal{Hom}_{\Lambda[U]}(\underline{HF}^-(Y_0;\Lambda_\omega),\Lambda[U])\cong\textnormal{Hom}_{\Lambda[U]}(\Lambda^n,\Lambda[U])=0.
\]
If we use the free resolution 
\[
0\rightarrow\Lambda[U]^n\xrightarrow{ U}\Lambda[U]^n\rightarrow0
\] 
and apply $\textnormal{Hom}_{\Lambda[U]}(-,\Lambda[U])$, 
\[
0\leftarrow\Lambda[U]^n\xleftarrow{ U}\Lambda[U]^n\leftarrow0,
\]
then $\textnormal{Ext}^1_{\Lambda[U]}(\underline{HF}^-(Y_0;\Lambda_\omega),\Lambda[U])$ is the cohomology at position $1$, which is isomorphic to $\Lambda^n$.
This completes the proof.

\end{proof}

\begin{figure}[!h]
	\def\svgwidth{16cm}
	\begin{center}
		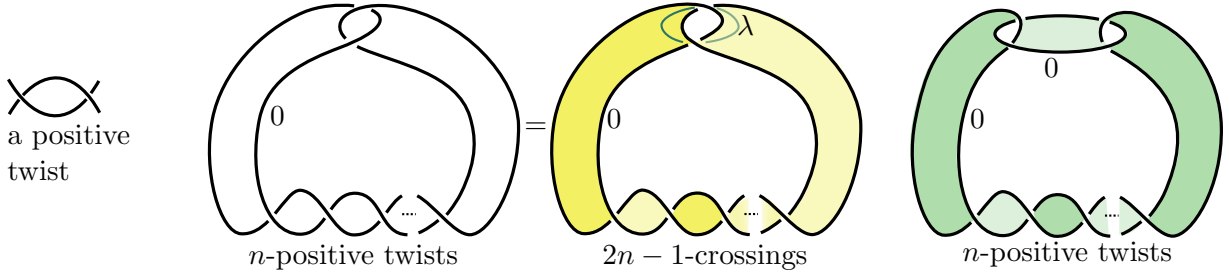
		\caption{Left: the three manifold $Y_0$ obtained by performing 0-surgery on $D_-(U,n)$. $\lambda$ is a non-separating curve on the non-separating torus in $Y_0$. Right: the three manifold obtained by performing 0-surgery on the $n$-twisted Whitehead link $W_n$.}
		\label{Whitehead double}
	\end{center}
\end{figure}

Figure \ref{Whitehead double}, in the middle, shows a non-separating torus $T\subset Y_0$ which is obtained by capping off the genus one Seifert surface for $D_-(U,n)$, shown in yellow. The zero-framing on lambda differs from the framing induced by the Seifert surface by 1. So, performing 0-surgery on lambda has the effect of composing the identification with a Dehn twist on lambda.  
This results in a manifold $S^3_0(W_n)$ obtained by performing 0-surgery on the $n$-twisted Whitehead link $W_n$, $n>0$, shown in Figure \ref{Whitehead double} on the right. Note that the Kirby diagram for $S^3_0(W_n)$, is obtained by a Kirby move from the Kirby diagram shown in the middle of Figure \ref{Whitehead double} where $\lambda$ has framing 0.
\begin{cor}\label{Whieheade-Link}
	Let $S^3_0(W_n)$ be a closed orientable three manifold obtained by performing 0-surgery on the $n$-twisted Whitehead link $W_n$, $n\neq0$. Suppose that $T$ is the non-separating torus in $S^3_0(W_n)$ obtained by capping off the genus 1 Seifert surface of $W_n$, which is shown in green on the right of Figure \ref{Borromean ring}. Let $[\omega]$ be a 2-dimensional cohomology class in $ H^2(S^3_0(W_n);\mathbb{R})$ such that $\omega([T])\neq0$. Then
	\[
	\underline{HF}(S^3_0(W_n);\Lambda_\omega)\cong\Lambda^{|n|}.
	\] 
\end{cor}
\begin{proof}
	Use Lemma \ref{Lem-Ex} and Theorem \ref{MainThm}.
\end{proof}

\begin{proof}[Proof of Corollary \ref{Topology-result}]
	Assume that $S^3_0(W_n)$ and $S^3_0(W_m)$ are related by the excision construction, and $[\omega_i]\in H^2(S^3_0(W_i);\mathbb{R})$, $i=n,m$, are 2-dimensional cohomology classes that satisfy the conditions of Theorem \ref{MainThm}. Furthermore, we can assume that $\omega_i([T_i])\neq0$ where $T_i$ is the non-separating torus in $S^3_0(W_i)$ obtained by capping off the genus 1 Seifert surface of $W_i$. Now Corollary \ref{Whieheade-Link} contradicts the result of Theorem \ref{MainThm}.  
\end{proof}

\begin{remark}
	That $\underline{HF}(S^3_0(W_n);\Lambda_\omega)\cong\Lambda^{|n|}$ can also be obtained using the techniques used in the proof of Theorem \ref{Thm-2}. See Corollary 2.3 of \cite{Ai-YiNi}.
\end{remark}

\begin{figure}[!h]
	\def\svgwidth{0.9\textwidth}
	\begin{center}
		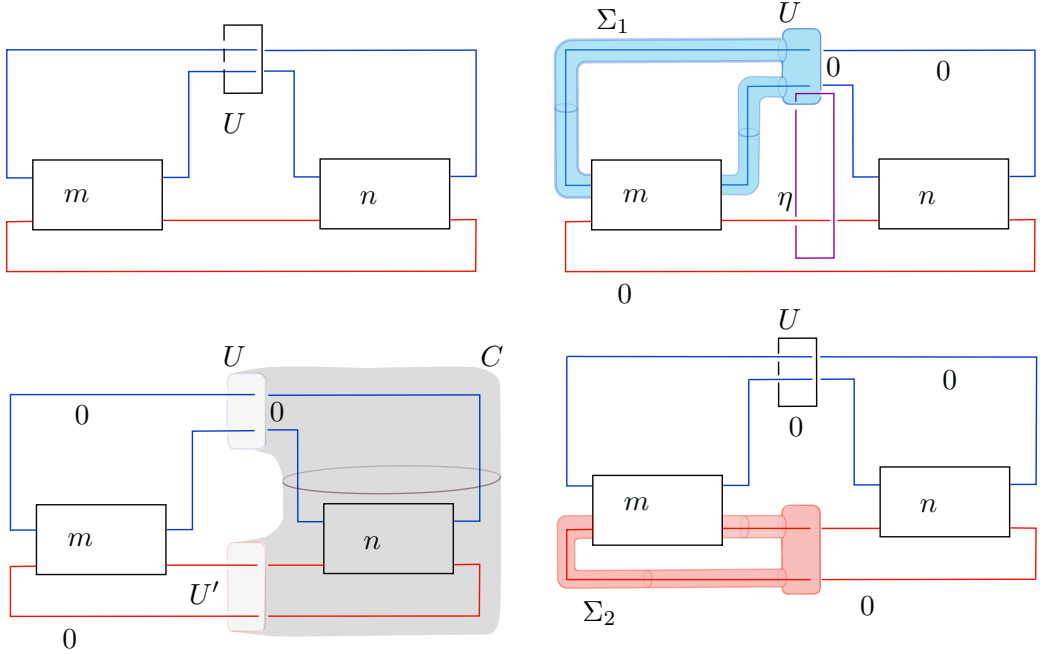
		\caption{Top left: $(m,n)$-twisted Borromean rings $B(m,n)$. Here, the boxes labeled $m$ and $n$ denote $m$ and $n$ full twists, respectively. Bottom left: a cylinder $C$ which bounds $U$ and $U'$. Top right and bottom right: two non-separating tori $\Sigma_1$ and $\Sigma_2$ in $S^3_0(B(m,n))$. Top right: the curve $\eta$ intersects $\Sigma_i$, $i=1,2$, geometrically once and is a meridian for the red curve.}
		\label{Borromean ring}
	\end{center}
\end{figure}

Let $B(m,n)$ denote \emph{$(m,n)$-twisted Borromean rings} as shown in Figure \ref{Borromean ring} on the top left.  $B(m,n)$ is a three component link such that each component is an unknot.  In Figure \ref{Borromean ring}, the boxes labeled $m$ and $n$ denote $m$ and $n$ full twists, respectively. Therefore, the blue and red components  are twisted $m+n$ times.  Let $U$ denote the component shown in black in Figure \ref{Borromean ring} on the top left. Note that $B(1,-1)$ represents the Borromean rings. Let $S^3_0(B(m,n))$ be a manifold obtained by performing 0-surgery on the components of $B(m,n)$.  There are two non-separating tori $\Sigma_1$ and $\Sigma_2$ in $S^3_0(B(m,n))$ shown in Figure \ref{Borromean ring} on the top right and on the bottom right. 
Indeed, the unknot $U$ with framing 0 specifies a 2-sphere $S$. The blue component intersects $S$ geometrically twice with opposite orientations. If we remove a neighborhood of two intersection points in $S$ and connect the two circle boundaries with a tubular neighborhood of the blue arc that connects them, we obtain $\Sigma_1$. 
Let $U'$ be the unknot as shown in Figure \ref{Borromean ring} on the bottom left. There is cylinder $C$ in $S^3_0(B(m,n))$ such that $\partial C=U\cup U'$ (see Figure \ref{Borromean ring} on the bottom left). Therefore, $U'$ specifies a 2-sphere which can be used to construct the surface $\Sigma_2$. Let $\eta$ be the curve shown in Figure \ref{Borromean ring} on the top left. The curve $\eta$ intersects $\Sigma_i$, $i=1,2$, geometrically once and is the meridian for the red component. 
\begin{lemma}\label{excision-on-borromean}
	$S^3_0(W_n)\amalg S^3_0(W_m)$  is obtained from $S^3_{0}(B(m,n))$ by excision construction along the surfaces $\Sigma_1\cup\Sigma_2$ shown in Figure \ref{Borromean ring}.
\end{lemma}

To prove this lemma, we need two auxiliary lemmas.
Let $K$ be a knot in a 3-manifold $Y$ and $\nu(K)\cong S^1\times D^2$ denote a tubular neighborhood of $K$ in $Y$. A simple closed curve $\lambda\subset\partial \nu(K)$ is called a framing for $K$. In fact, a framing of $K$ determines a new manifold $Y_\lambda(K)$ obtained by gluing a disk to $\lambda$ in $\partial(\nu(K))$.  We say that $\lambda$ is a Morse framing if $[\lambda]\in H_1(\nu(K))$ is a generator.
\begin{figure}[!h]
	\def\svgwidth{\textwidth}
	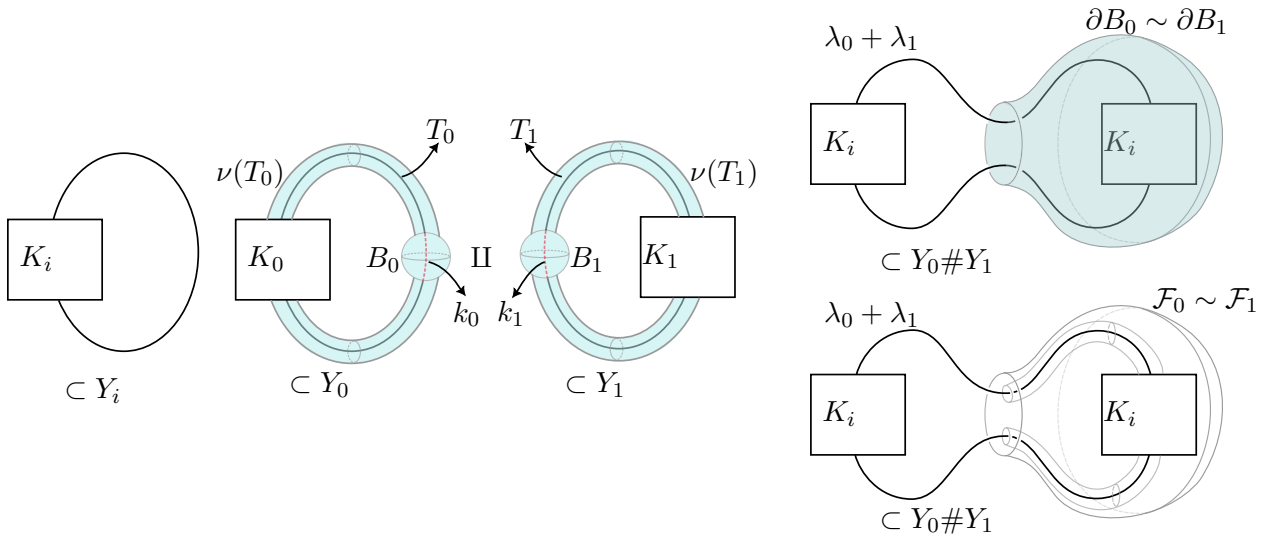
	\caption{Left: a knot $K_i$ in $Y_i$, $i=0,1$. Middle: $\nu(K_i)=B_i\cup\nu(T_i)$. Right-top: $\lambda_0+\lambda_1$-surgery on $K_0\#K_1$ in $Y_0\#Y_1$. Right-bottom: a genus one surface obtained by identifying the surfaces $\mathcal{F}_i=\partial(\nu(T_1)\cup B_i)\subset Y_i$, $i=0,1$, by $\phi$.}
	\label{Splice}
\end{figure}
\begin{lemma}\cite[Lemma 2.7 ]{Hendricks2024ANO}\label{Splice-lem}
	Let $(Y_0,K_0)$ and $(Y_1,K_1)$ be knots with Morse framings $\lambda_0$ and $\lambda_1$ respectively. Let 
	\[
	\phi:\partial(\nu(K_0))\rightarrow\partial(\nu(K_1))
	\] 
	be a gluing map which identifies the meridian $\mu_0$ with $\mu_1$, and which
	maps $\lambda_0$ to $-\lambda_1$. Then
	$(Y_0 \setminus \nu(K_0))\cup_\phi (Y_1 \setminus \nu(K_1))$
	is equal to $(Y_0\#Y_1)_{\lambda_1+\lambda_2}(K_0\#K_1)$. 
\end{lemma}
This Lemma is proved in \cite[Lemma 7.1 ]{Gordon1983DehnSA} and \cite[Lemma 6.1]{Shinji}. Here, we present a slightly different proof.
\begin{proof}
	Let $B_i\subset Y_i$, $i=0,1$, denote two three balls such that $B_i\cap K_i$, resp. $B_i\cap\lambda_i$ consists of an arc $k_i$, resp. $k'_i$, $i=0,1$. Let $T_i=K_i-k_i$ and $\lambda_i'=\lambda_i-k'_i$. We can write $\nu(K_i)=B_i\cup\nu(T_i)$ (see Figure \ref{Splice} in the middle). We can form $(Y_0 \setminus \nu(K_0))\cup_\phi (Y_1 \setminus \nu(K_1))$ in two steps. First, remove the interiors of $B_i$, $i=0,1$, and identify the boundary components (this identifies $\mu_0$ with $\mu_1$) then identify $\lambda'_0$ with $-\lambda'_1$. The first step results in the connected sum $(Y_0\#Y_1,K_0\#K_1)$. The second step glues a disk to  $\lambda_0+\lambda_1$. This means performing a $\lambda_0+\lambda_1$-surgery on $K_0\# K_1$ in $Y_0\#Y_1$ (see Figure \ref{Splice} on the right).  
\end{proof}
A genus one surface in $(Y_0\#Y_1)_{\lambda_1+\lambda_2}(K_0\#K_1)$ which is obtained by the identification of  $\mathcal{F}_0=\partial(\nu(K_0))$ with $\mathcal{F}_1=\partial(\nu(K_1))$ via $\phi$ is shown in Figure \ref{Splice} on the right-bottom.

Let $K_0$ and $K_1$ be two knots in a 3-manifold $Y$. We also show the copies of these knots in $Y\#(S^1\times S^2)$ by $K_0$ and $K_1$. Remove a small arc $k_i$, with endpoints $a_i$ and $b_i$, from $K_i$, then identify $a_0$ with $a_1$ and $b_0$ with $b_1$ to form a knot $K$. Let $a=a_0\sim a_1$ and $b=b_0\sim b_1$. We form $K$ such that it intersects $\{p\}\times S^2$ ($p\in S^1$) in  $a$ and $b$ with opposite signs. We denote $K$ with $K_0\natural K_1$ (see Figure \ref{Self-Splice} on the right-top). If $\lambda_i$ is a framing for $K_i$, there is an induced framing on $K_0\natural K_1$ denoted by $\lambda_0+\lambda_1$

\begin{figure}[!h]
	\def\svgwidth{\textwidth}
	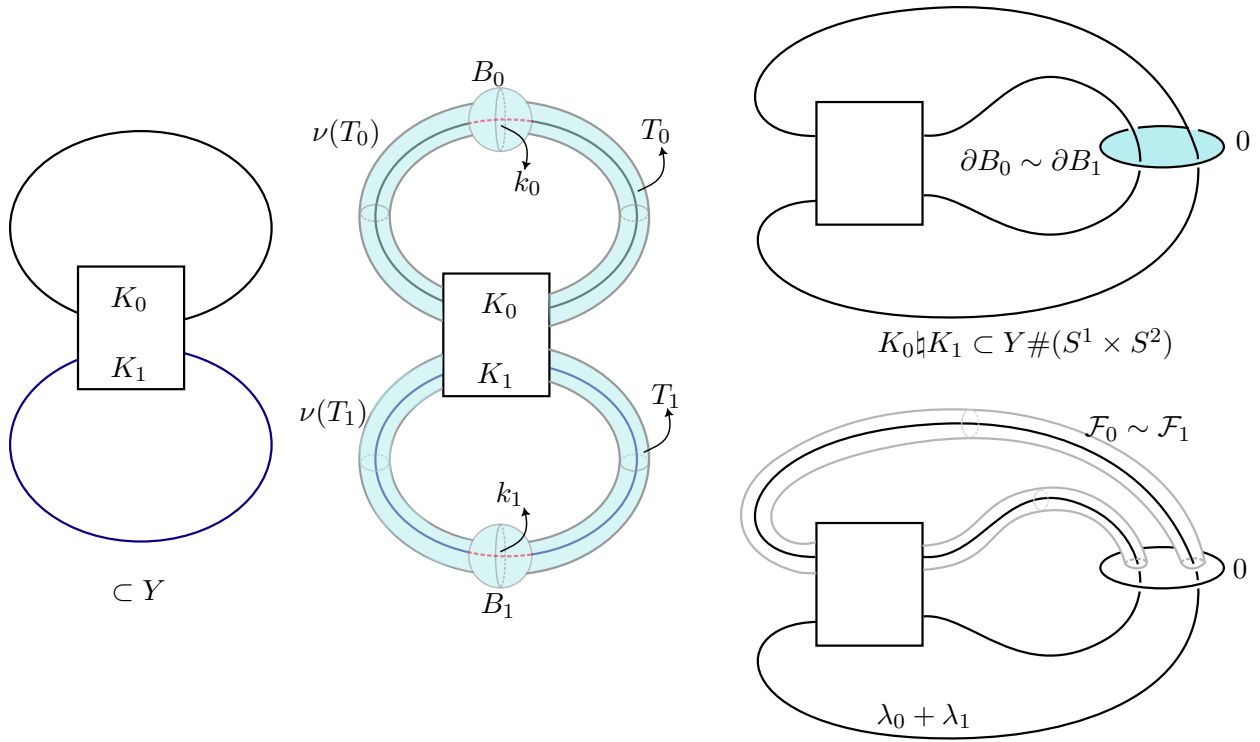
	\caption{Left: two knots $K_i$, $i=0,1$, in $Y$. Middle: $\nu(K_i)=B_i\cup\nu(T_i)$. Right-top: $K_0\natural K_1$ in $Y\#S^1\times S^2$. Right-bottom: a genus one surface  in $(Y\# S^1\times S^2)_{\lambda_0+\lambda_1}(K_0\natural K_1)$ obtained by identifying the surfaces $\mathcal{F}_i=\partial(\nu(T_1)\cup B_i)\subset Y$, $i=0,1$ by $\phi$.}
	\label{Self-Splice}
\end{figure}
\begin{lemma}\label{Self-Splice-Lem}
	Let $K_0,K_1\subset Y$ be knots with Morse framings $\lambda_0$ and $\lambda_1$ respectively. Let 
	\[
	\phi:\partial(\nu(K_0))\rightarrow\partial(\nu(K_1))
	\] 
	be a gluing map which identifies the meridian $\mu_0$ with $\mu_1$, and which
	maps $\lambda_0$ to $-\lambda_1$. Then
	\[
	\frac{Y \setminus (\nu(K_0)\cup \nu(K_1))}{x\sim\phi(x)}
	\]
	is equal to $(Y\#S^1\times S^2)_{\lambda_0+\lambda_1}(K_0\natural K_1)$. 
\end{lemma}
\begin{proof}
	Let $B_i\subset Y$, $i=0,1$, denote two three balls such that $B_i\cap K_i$, resp. $B_i\cap\lambda_i$ consists of the arc $k_i$, resp. an arc $k'_i$, $i=0,1$. Let $T_i=K_i-k_i$ and $\lambda_i'=\lambda_i-k'_i$. We can write $\nu(K_i)=B_i\cup\nu(T_i)$ (see Figure \ref{Self-Splice} in the middle). We can form
	\[
	\frac{Y \setminus (\nu(K_0)\cup \nu(K_1))}{x\sim\phi(x)}
	\]  
	in two steps. First, remove the interiors of $B_i$, $i=0,1$, and identify the boundary components (this identifies $\mu_0$ with $\mu_1$) then identify $\lambda'_0$ with $-\lambda'_1$. The first step results in $(Y\# S^1\times S^2,K_0\natural K_1)$. The second step glues a disk to  $\lambda_0+\lambda_1$. This means performing a $\lambda_0+\lambda_1$-surgery on $K_0\natural K_1$ in $Y\#S^1\times S^2$ (see Figure \ref{Self-Splice} on the right).  
\end{proof}
A genus one surface in $(Y\#S^1\times S^2)_{\lambda_0+\lambda_1}(K_0\natural K_1)$ which is obtained by the identification of  $\mathcal{F}_0=\partial(\nu(K_0))$ with $\mathcal{F}_1=\partial(\nu(K_1))$ via $\phi$ is shown in Figure \ref{Self-Splice} on the right-bottom.

\begin{figure}[!h]
	\def\svgwidth{\textwidth}
	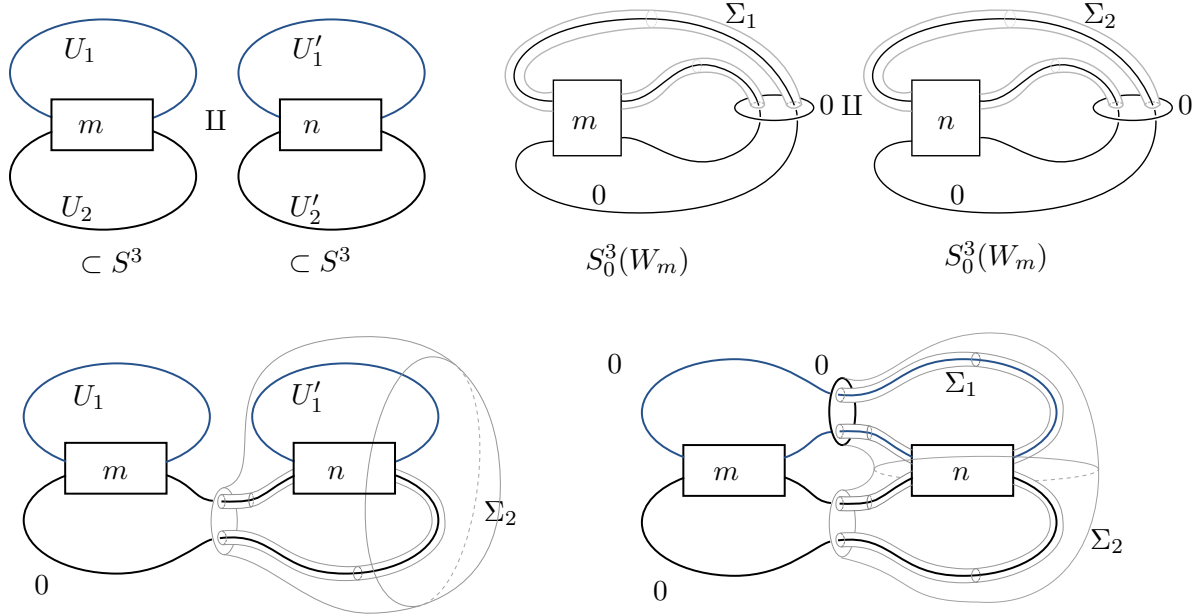
	\caption{Top left: torus links $T_m$ and $T_n$ in two copies of $S^3$. Here, the boxes labeled $m$ and $n$ denote $m$ and $n$ full twists, respectively. Top right: $S^3_0(W_m)\amalg S^3_0(W_n)$ obtained by applying Lemma \ref{Self-Splice-Lem} on the torus links $T_m$ and $T_n$. Bottom left: the manifold obtained by  applying Lemma \ref{Splice-lem} for $(S^3,U_2)$ and $(S^3,U'_2)$. Bottom right: $S^3_0(B(m,n))$ obtained by applying Lemma \ref{Self-Splice-Lem} to the manifold shown on the bottom left. }
	\label{S-S-S}
\end{figure}
\begin{proof}[Proof of Lemma \ref{excision-on-borromean}]
Let $T_n$ and $T_m$ denote two torus links in two copies of $S^3$ with components $U_i$ and $U'_i$, $i=1,2$, respectively (see Figure \ref{S-S-S} on the top left). Let $\mu_i$ and $\lambda_i$ denote the meridian and Seifert longitude of $U_i$. Similarly, $\mu'_i$ and $\lambda'_i$ denote the meridian and Seifert longitude of $U'_i$.  Let 
\[
\psi:\partial(\nu(U_1))\rightarrow\partial(\nu(U_2))
\]
be a gluing map which identifies the meridian $\mu_1$ with $\mu_2$ and which maps $\lambda_1$ to $-\lambda_2$. Similarly,  
\[
\psi':\partial(\nu(U'_1))\rightarrow\partial(\nu(U'_2))
\]
is a gluing map which identifies the meridian $\mu'_1$ with $\mu'_2$ and which maps $\lambda'_1$ to $-\lambda'_2$.
If we apply Lemma \ref{Self-Splice-Lem} for the components of the torus link $T_n$ (resp. $T_m$) with the gluing map $\psi$ (resp. $\psi'$) above, we obtain $S^3_0(W_n)$ with the surface $\Sigma_1$ (resp. $S^3_0(W_m)$ with the surface $\Sigma_2$) (see Figure \ref{S-S-S} on the top right). If we apply Lemma \ref{Splice-lem} for $(S^3,U_2)$ and $(S^3,U'_2)$, we obtain the manifold shown in Figure \ref{S-S-S} on the bottom left with the surface $\Sigma_2$. Let $U_1$ and $U'_1$ also denote the knots induced by $U_1$ and $U'_1$ in the resulting manifold. Finally, if we apply Lemma \ref{Self-Splice-Lem} for $U_1$ and $U'_1$ in the manifold obtained in the last step, we will obtain $S^3_0(B(m,n))$ with the surfaces $\Sigma_1$ and $\Sigma_2$ (see Figure \ref{S-S-S} on the bottom right). This completes the proof.
\end{proof}

\begin{cor}\label{Borromean}
	Let $[\omega]\in H^2 (S^3_0(B(m,n));\mathbb{R})$ be $\lambda \textnormal{PD}[\eta]$, where $\mu$ is the curve in Figure \ref{Borromean ring} on the top right and $\lambda\in\mathbb{R}\setminus\{0\}$.
	Then
	\[
	\underline{HF}(S^3_0(B(m,n));\Lambda_\omega)\cong \Lambda^{|mn|}.
	\]
\end{cor}
\begin{proof}
	Let $\Sigma_i$, $i=1,2$, denote the surface in $S^3_0(B(m,n))$ shown in Figure \ref{Borromean ring} on the top right and on the bottom right. It is clear that $[\Sigma_i]\neq0$, $i=1,2$, and $F=\Sigma_1\cup\Sigma_2$ is separating in $S^3_0(B(m,n))$. By Lemma \ref{excision-on-borromean}, if we perform excision in $S^3_0(B(m,n))$ along $F$, we obtain $S^3_0(W_m)\amalg S^3_0(W_n)$. By Theorem \ref{MainThm}, 
	\begin{equation}\label{1}
	\underline{HF}(S^3_0(B(m,n));\Lambda_\omega)\cong \underline{HF}(S^3_0(W_m)\amalg S^3_0(W_n);\Lambda_{\omega'}),
	\end{equation}
	where $[\omega']\in H^2(S^3_0(W_m)\amalg S^3_0(W_n);\mathbb{R})$ satisfies the conditions of Theorem \ref{MainThm}. Note that
	\[
	\underline{CF}^-(S^3_0(W_m)\amalg S^3_0(W_n);\Lambda_{\omega'})= \underline{CF}^-(S^3_0(W_m);\Lambda_{\omega'_1})\otimes_{\Lambda}\underline{CF}^-(S^3_0(W_n);\Lambda_{\omega'_2}),
	\]
	where $\omega'=\omega'_1\oplus\omega'_2$, $\omega'_1\in H^2(S^3_0(W_m);\mathbb{R})$, $\omega'_2\in H^2(S^3_0(W_n);\mathbb{R})$. Since $\Lambda$ is a field,
	\begin{equation}\label{2}
	\underline{HF}^-(S^3_0(W_m)\amalg S^3_0(W_n);\Lambda_{\omega'})= \underline{HF}^-(S^3_0(W_m);\Lambda_{\omega'_1})\otimes_{\Lambda}\underline{HF}^-(S^3_0(W_n);\Lambda_{\omega'_2}).
    \end{equation}
	The result follows from \eqref{1}, \eqref{2}, Corollary \ref{Whieheade-Link}, and the fact that $\underline{HF}(Y,\mathfrak{s};\Lambda_\omega)\cong \underline{HF}^-(Y,\mathfrak{s};\Lambda_\omega)$, when $c_1(\mathfrak{s})$ is torsion and $[\omega]\neq0$.
\end{proof}

Let $K$ be a framed knot in a 3–manifold $Y$ with framing $\lambda$ and meridian $\mu$.
Given an integer $r$, let $Y_r(K)$ denote the 3–manifold obtained from $Y$ by performing
Dehn surgery along the knot $K$ with framing $\lambda+r\mu$. Let $N(K)$ denote a small
tubular neighborhood of the knot $K$ and $\eta\subset Y-N(K)$ be a closed curve in the knot
complement. Then for any integer $r$, $\eta\subset Y-N(K)\subset Y_r(K)$ is a closed curve in
the surgered manifold $Y_r(K)$. We denote its Poincare dual by $[\omega_r]\in H^2(Y_r;\mathbb{R})$. There is a surgery exact sequence for $\omega$-twisted Floer homology \cite[Section 3]{A-P2010}
\begin{equation}\label{exact-seq}
	\cdots\rightarrow\underline{HF}^+(Y;\Lambda_\omega)\rightarrow\underline{HF}^+(Y_{0};\Lambda_{\omega_0})\rightarrow\underline{HF}^+(Y_{1};\Lambda_{\omega_1})\rightarrow\cdots,
\end{equation}
where $[\omega]=\textnormal{PD}[\eta]$.

\begin{figure}
	\def\svgwidth{0.9\textwidth}
	\begin{center}
	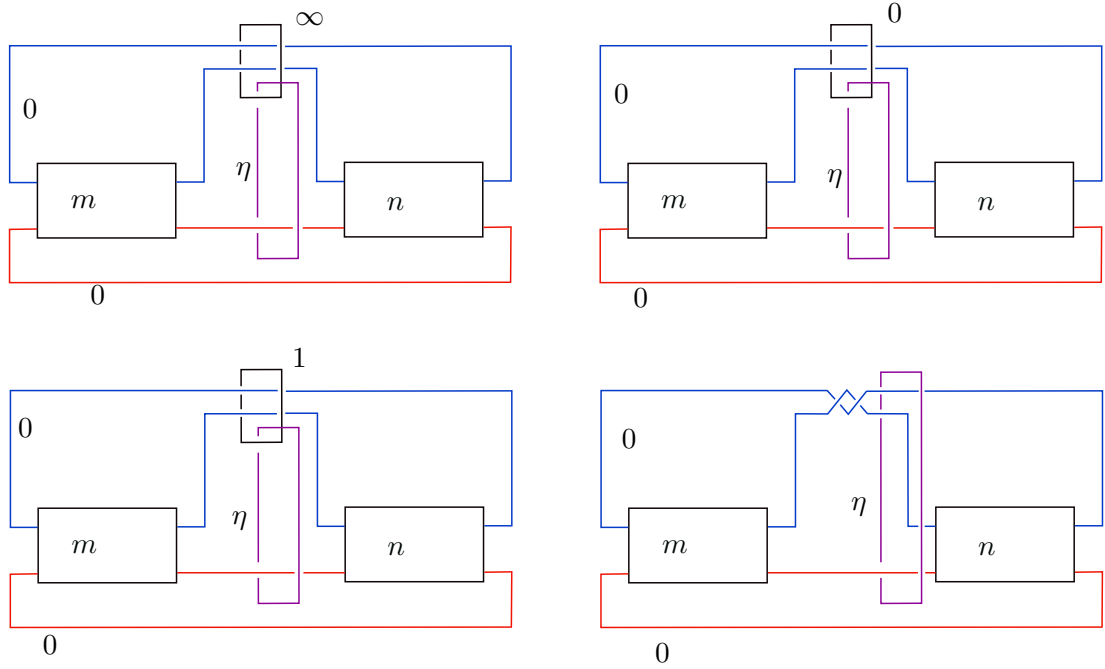
	\caption{Top right: $S^3_0(B(m,n))$. Bottom: $S^3_0(C(m,1,n))$. The boxes labeled $m$ and $n$ denote $m$ and $n$ full twists, respectively.}
	\label{Two-Bridge}
    \end{center}
\end{figure}

\begin{proof}[Proof of Corollary \ref{two-bridge}]
	Let $Y$ be the manifold obtained by performing 0-surgery on the red and blue components of $B(m,n)$ (see Figure \ref{Borromean ring} on the top). Let $K$ denote the knot induced from $U$ in $Y$. Note that
	\[
	Y=\begin{cases}
		\#^2(S^1\times S^2)&\textnormal{if}\ m=-n,\\
		S^3&\textnormal{if}\ |m-n|=1,
	\end{cases}
	\]
	$Y_0=Y_0(K)=S^3_0(B(m,n)$ and $Y_1=Y_1(K)=S^3_0(C(m,\pm1,n))$ (for appropriate choices of the framing). Note that the induced curve $\eta$ in $Y_0$ is such that $\omega_0=\lambda\textnormal{PD}[\eta]$ satisfies the condition in Corollary \ref{Borromean} (see Figure \ref{Two-Bridge} on the top left). When $m=-n$, let $S$ denote a 2-sphere in $Y$ obtained by capping off the Seifert disk of the 0-framed component in red. In this case, the induced 2-dimensional cohomology class $[\omega']$ in $Y$, is such that $\omega'([S])\neq0$ (see Figure \ref{Two-Bridge} on the top right). Therefore, $\underline{HF}(S^1\times S^2;\Lambda_{\omega'})=0$ (see \cite[Subsection 2.2]{A-P2010}). Using Corollary \ref{Borromean} and the long exact sequence in \eqref{exact-seq}, the result follows. When $|m-n|=1$, since the map
	\[
	\Lambda[U^{-1}]\cong \underline{HF}^+(Y;\Lambda_\omega)\rightarrow\underline{HF}^+(Y_0;\Lambda_{\omega_0})\cong\Lambda^{|mn|}
	\]
	in \eqref{exact-seq} is zero, the result follows.
\end{proof}

\bibliography{simplifying-structures}
\Addresses
\end{document}